\documentclass[11pt]{article}
\usepackage{amscd, amsmath, amssymb, amsthm, tikz, graphicx}
\usepackage[all,cmtip]{xy}
\usepackage[pagebackref]{hyperref}
\usepackage{caption}
\title{Terminal 3-folds that are not Cohen-Macaulay }
\author{Burt Totaro}
\date{  }

\def\Z{\text{\bf Z}}
\def\Q{\text{\bf Q}}

\def\P{\text{\bf P}}
\def\F{\text{\bf F}}

\def\arrow{\rightarrow}

\def\imp{\Rightarrow}

\DeclareMathOperator{\Spec}{Spec}
\DeclareMathOperator{\Aut}{Aut}
\DeclareMathOperator{\tr}{tr}
\DeclareMathOperator{\id}{id}
\DeclareMathOperator{\red}{red}

\DeclareMathOperator{\m}{\mathfrak{m}}
\def\Im{\operatorname{im}}

\renewcommand\labelenumi{(\arabic{enumi})}
\renewcommand\theenumi\labelenumi

\begin{document}
\maketitle
\newtheorem{theorem}{Theorem}[section]
\newtheorem{proposition}[theorem]{Proposition}
\newtheorem{corollary}[theorem]{Corollary}
\newtheorem{lemma}[theorem]{Lemma}

\theoremstyle{definition}
\newtheorem{definition}[theorem]{Definition}
\newtheorem{example}[theorem]{Example}

\theoremstyle{remark}
\newtheorem{remark}[theorem]{Remark}

\usetikzlibrary{patterns}

An important ingredient of the minimal model
program is that Kawamata log terminal singularities in characteristic zero
are rational, and in particular Cohen-Macaulay.
In the special case of cone singularities, this fact is related
to the Kodaira vanishing theorem restricted
to Fano varieties. It turns out that Kodaira vanishing
fails even for Fano varieties, in every characteristic $p>0$
\cite{Totarokodaira}. This led to examples
of klt, and even terminal, singularities that are not Cohen-Macaulay
\cite{Kovacs, Totarokodaira, Yasudaterminal}.
(Terminal singularities are the smallest class of singularities
that can be allowed on minimal models.)

The most notable example was a terminal singularity
{\it of dimension 3 }that is not Cohen-Macaulay.
Namely, let $X$ be the quotient $(A^1-0)^3/G$
over the field $\F_2$, where the generator $\sigma$
of the group $G=\Z/2$ acts by
$$\sigma(x_1,x_2,x_3)=\bigg( 
\frac{1}{x_1},\;
\frac{1}{x_2},\;
\frac{1}{x_3}\bigg).$$
Then $X$ is terminal but not Cohen-Macaulay \cite[Theorem 5.1]{Totarokodaira}.
This is the lowest possible dimension, because every terminal
(or just normal) surface is Cohen-Macaulay. Cohen-Macaulayness
and stronger properties such as $F$-regularity help
to construct contractions of varieties. Partly for this reason,
the MMP for 3-folds is known only in characteristics at least 5
\cite{HX,Birkar,HW}. By Arvidsson--Bernasconi--Lacini, klt singularities
in characteristic greater than 5 are Cohen-Macaulay,
whereas there are klt singularities that are not Cohen-Macaulay
in characteristics 2, 3, and 5 \cite{ABL,Bernasconi,CT}.

In this paper, we construct terminal 3-fold singularities
that are not Cohen-Macaulay in five new cases:
mixed characteristic $(0,2)$, characteristic 3,
mixed characteristic $(0,3)$, characteristic 5,
and mixed characteristic $(0,5)$ (Theorems
\ref{mainintro}, \ref{F3thm}, \ref{Z3thm}, \ref{F5thm},
and \ref{Z5thm}). This is optimal, in view of the result
of Arvidsson--Bernasconi--Lacini.
Indeed, the MMP for schemes of dimension 3 was developed in mixed
characteristic when the residue characteristic is greater than 5
\cite{B+,TY}. This raised the question of whether vanishing theorems
for 3-folds hold in mixed characteristic.
Given our counterexample over $\F_2$,
one might expect an example
of dimension 4, flat over the 2-adic integers $\Z_2$, with fiber
over $\F_2$ being the 3-fold singularity above.
In fact, each of our examples has dimension 3 as a scheme.
For example, over $\Z_2$ we have:

\begin{theorem}
\label{mainintro}
Let $Y = \{(x,y,i)\in A^3_{\Z_2}: x\neq 0, y\neq 0, i^2=-1\}$.
Let the group $G=\Z/2=\{1,\sigma\}$ act on $Y$
by
$$\sigma(x,y,i)=(1/x,1/y,-i).$$
Then the scheme $Y/G$ is terminal, not Cohen-Macaulay,
of dimension 3, and flat over $\Z_2$.
\end{theorem}

Note that an action of a $p$-group with an isolated fixed point
on a positive-dimensional smooth variety in characteristic $p$
is never formally isomorphic to a linear action,
because a nonzero representation $V$ of a $p$-group $G$ in characteristic $p$
has nonzero $G$-fixed subspace.
In fact, there are continuous families of non-equivalent actions
on smooth varieties in characteristic $p$,
and likewise on regular schemes of mixed characteristic.
For an action of $G=\Z/p$
on a regular scheme $W$ of dimension 3
with an isolated fixed point in characteristic $p$ (as in Theorem
\ref{mainintro}), it is common
for $W/G$ not to be Cohen-Macaulay, essentially because the cohomology
of $G$ contributes to the local cohomology of $W/G$. The difficulty
is to construct an example with $W/G$ terminal. For a more complicated
action of $G$, the quotient scheme
would usually not be terminal
or even log canonical. To find the examples in this paper,
the idea was to look for the simplest possible actions
of $G=\Z/p$ on a regular 3-dimensional scheme
with an isolated fixed point of residue characteristic $p$. 

Our examples build on
Artin's examples of the simplest $\Z/p$-actions on smooth surfaces
in characteristic $p$ with isolated fixed points \cite{Artin}.
Namely, he constructed
a $\Z/2$-action in characteristic 2 with quotient a du Val singularity
of type $D_4$,
a $\Z/3$-action in characteristic 3 with quotient an $E_6$ singularity,
and a $\Z/5$-action in characteristic 5 with quotient an $E_8$
singularity. These special group actions arise globally from actions
on del Pezzo surfaces, for example $\Z/5$ acting on the quintic
del Pezzo surface (as discussed in section \ref{char5}).

To show that our 3-dimensional quotients $W/G$ are terminal,
the obvious approach would be to resolve the singularities
of $W/G$ and make a calculation. Resolving these singularities
is hard, however.
We can greatly simplify the work by stopping at a partial resolution
of $W/G$ that has toric singularities (specifically,
$\mu_p$-quotient singularities, which we call
tame quotient singularities);
those are easy to analyze in combinatorial terms.
(Recent advances suggest that an efficient
substitute for resolving singularities in any characteristic
would be to seek a resolution by a tame stack, rather than
by a regular scheme \cite{ATW,McQuillan,AOV}.) Our key technical tool
is Theorem \ref{mup},
which gives a sufficient condition
for a quotient scheme $U/G$ (where $G=\Z/p$,
in positive or mixed characteristic) to have toric singularities.

These examples should lead to other failures of vanishing theorems.
In particular, by Baudin, Bernasconi, and Kawakami,
these examples imply that Frobenius-stable Grauert-Riemenschneider
vanishing fails for terminal 3-folds
in characteristic 2, 3, or 5 \cite[Remark 1.2]{BBK}.

This work was supported by NSF grant DMS-2054553,
Simons Foundation grant SFI-MPS-SFM-00005512,
and the Charles Simonyi Endowment
at the Institute for Advanced Study.

\section{Notation}

We use the notation ``$x=y+I$'', for elements
$x$ and $y$ of a ring and an ideal $I$, to mean that there is an $i\in I$
such that $x=y+i$. We also use variants of this notation such as
``$x=(y+I)(z+J)$''. Another variant (modeled on big-O notation
in analysis) is to write ``$x=y+O(z)$''
for ``$x=y+(z)$''.

We write $R\{ x_1,\ldots,x_n\}$ for
the free module over a ring $R$ with basis elements $x_1,\ldots,x_n$.

For a closed point $P$ in a regular scheme $U$ with residue field
$k_U$, we say that $f_1,\ldots,f_n$
are {\it coordinates }for $U$ at $P$ (or a {\it regular system
of parameters}) if $f_1,\ldots,f_n$ are elements
of the maximal ideal $\m$ of $O(U)$ (the regular functions vanishing at $P$)
that map to a basis for the $k_U$-vector space $\m/\m^2$.

For a group $G$ acting on a scheme $X$, $G$ acts on the ring of regular
functions $O(X)$ by $(g(f))(x)=f(g^{-1}x)$, or equivalently
$g(f)=(g^{-1})^*(f)$. The inverse is needed
because of our convention of writing group actions on the left.
Throughout the paper, we write $G=\Z/p=\langle \sigma:\sigma^p=1\rangle$
for the cyclic group of prime order $p$. We fix the name $\tau=\sigma^{-1}$,
because of the inverse that comes up in writing the $G$-action
on functions. Write $I(f)=\sigma(f)-f$ for a function $f$ on a $G$-scheme.

See section \ref{Z2} for the definition of the canonical divisor
and terminal singularities
on general schemes, following
\cite[section 2.1]{Kollarsings}.

For a positive integer $r$,
let $\mu_r$ be the group scheme (over any base scheme)
of $r$th roots of unity.
The Reid-Tai criterion is the following description
of which cyclic quotient singularities are canonical or terminal
\cite[Theorem 4.11]{Reid}. 
This is often stated over a field, but it works
even in mixed characteristic for $\mu_r$-quotient singularities.
The point is that
Kato's theory of log regular schemes provides a mixed-characteristic
analog of toric singularities, which includes the case of
$\mu_r$-quotient singularities \cite{Kato}.
For such schemes, resolutions of singularities
and the canonical divisor can be described in purely combinatorial terms.
In the following criterion, for $a$ an integer and $b$ a positive integer,
consider $a\bmod b$ as an integer in the set $\{0,\ldots,b-1\}$.

\begin{theorem}
\label{reidtai}
For a positive integer $r$, let $\mu_r$
act on a regular scheme $X$, fixing a closed point
$P$ with maximal ideal $\m$. Suppose that $\mu_r$ acts on a basis
for $\m/\m^2$
by $\zeta(t_1,\ldots,t_n)=(\zeta^{b_1}t_1,\ldots,
\zeta^{b_n}t_n)$, for some $b_1,\ldots,b_n\in \Z/r$.
(The quotient $X/\mu_r$ is said to be a $\mu_r$-quotient
singularity of type $\frac{1}{r}(b_1,\ldots,b_n)$.) Assume that the action
is well-formed in the sense that $\gcd(r,b_1,\ldots,\widehat{b_j},
\ldots,b_n)=1$ for all $j=1,\ldots,n$. Then $X/\mu_r$ is canonical
(resp.\ terminal) near the image of $P$ if and only if
$$\sum_{j=1}^n ib_j\bmod r\geq r$$
(resp.\ $>r$) for all $i=1,\ldots,r-1$.
\end{theorem}

We sometimes need the extension of the Reid-Tai criterion
that describes when a toric pair is terminal, as follows.
The proof is the same as Reid's:
$X/\mu_r$ has a toric resolution of singularities, and so it suffices
to compute discrepancies for toric divisors
over $X/\mu_r$.

\begin{theorem}
\label{reidtaipair}
Under the assumptions of Theorem \ref{reidtai}, let $t_1,\ldots,t_n$
be coordinates for $X$ at $P$ that are $\mu_r$-eigenfunctions with weights
$b_1,\ldots,b_n$. For $i=1,\ldots,n$, let $D_i$ be the irreducible divisor
on $X/\mu_r$ that is the image of $\{t_i=0\}$ in $X$.
Let $c_1,\ldots,c_n$ be real numbers. Then the pair
$(X/\mu_r,\sum c_j D_j)$ is terminal near the image of $P$
if and only if $c_j<1$ for each $j$,
$c_j+c_k<1$ for each $j\neq k$, and
$$\sum_{j=1}^n (1-c_j)(ib_j\bmod r)> r$$
for all $i=1,\ldots,r-1$.
\end{theorem}

\section{Recognizing $\mu_p$-quotient singularities}

Kir\'aly and L\"utkebohmert analyzed when a quotient scheme 
by $\Z/p$ is regular, the hard case being where the residue
characteristic is $p$ \cite{KL}. More generally, we now give
a sufficient condition for a quotient by $\Z/p$
to be a toric singularity, or in particular to be
a $\mu_p$-quotient singularity. (Actions of $\mu_p$ are far simpler
than actions of $\Z/p$:
they are linearizable near a fixed point, like finite group
actions in characteristic zero, because $\mu_p$ is linearly
reductive. Also, $\mu_p$-quotient singularities are always
Cohen-Macaulay.)

It would be very appealing to find a broader sufficient condition
for a quotient by $\Z/p$ to have toric singularities,
perhaps necessary and sufficient.
Theorem \ref{mup} is adapted to the situation where
the fixed point scheme is a Cartier divisor $E_1=\{e=0\}$
except on a closed subset of $E_1$.
When it applies, the theorem can be described as a $\Z/p$ -- $\mu_p$ switch.

Here is Kir\'aly and L\"utkebohmert's main result \cite[Theorem 2]{KL}.
(See Theorem \ref{klnormal} for a more detailed statement.)

\begin{theorem}
\label{kl}
Let $G$ be a cyclic group of prime order which acts on a regular scheme
$X$. If the fixed point {\it scheme }$X^G$ is a Cartier divisor in $X$, then
the quotient space $X/G$ is regular.
\end{theorem}

Here is our sufficient condition for toric singularities. Recall
that $I(f)$ means $\sigma(f)-f$, for a function $f$ on a scheme with
an action of the group $G=\Z/p=\langle \sigma:\sigma^p=1\rangle$.

\begin{theorem}
\label{mup}
Let $U$ be a regular scheme with an action of the group
$G=\Z/p=\langle \sigma:\sigma^p=1\rangle$, for a prime number $p$.
Suppose that $G$ fixes a closed point $P$ with perfect
residue field $k_U$ of characteristic $p$. Write $\m$
for the maximal ideal in the local ring $O_{U,P}$. Suppose that
there are $e,s\in \m-\{0\}$ and coordinates $x_1,\ldots,x_n$
for $U$ at $P$ such that
$$I(s)=es(\text{unit})$$
and
$$I(x_i)\in (e)$$
for $i=1,\ldots,n$. Suppose that $p\in e^{p-1}\m$. (For example,
that holds if $p=0$ on $U$.) Then $U/G$ is regular or has a $\mu_p$-quotient
singularity at the image of $P$.

More precisely, if $I(x_i)/e$ is nonzero at $P$ for some $i$,
then $U/G$ is regular near $P$. Otherwise,
let $\varphi(y)=I(y)/e$, which gives
a linear map from $\m/\m^2$ to itself. Then $\varphi$
is diagonalizable, and after multiplying it by some nonzero
scalar, its eigenvalues $b_1,\ldots,b_n$ are in $\F_p\subset k_U$.
In this case, $U/G$ has a singularity of the form
$\frac{1}{p}(b_1,\ldots,b_n)$.
That is,
near the image of $P$, $U/G$ is the quotient of a regular scheme
by an action of $\mu_p$ with
weights $b_1,\ldots,b_n$
at a fixed point.
\end{theorem}

Theorem \ref{mup-divisor} is a refinement of Theorem \ref{mup},
showing that certain divisors in $U$
can be viewed as toric divisors in $U/G$.

We will often use the special case of Theorem \ref{mup}
where $e=s=x_1$ for one of the coordinates $x_1$; in that case,
we are assuming that
$I(x_1)=x_1^2(\text{unit})$ and $I(x_i)\in (x_1)$
for $i=1,\ldots,n$.
Then the fixed point
scheme $U^G$ is the regular divisor $\{ x_1=0\}$ except on a closed
subset of $E_1$. For other applications in this paper, we need
the greater generality of Theorem \ref{mup}, where $(U^G)_{\red}$
may be a singular divisor.

One might ask whether the assumption that $p\in e^{p-1}\m$
in Theorem \ref{mup}
can be omitted. Fortunately, this assumption is automatic
in characteristic $p$, and it will be easy to check
in our mixed-characteristic examples.

\begin{proof}
The assumption that $P$ is fixed by $G$ means that
$G$ preserves $\m=\m_U$ and acts as the identity on the residue
field $k_U:=O_{U,P}/\m$, which has characteristic $p$.
For each regular function $f$ in the local ring $O_{U,P}$,
its norm $N(f):=\prod_{i=0}^{p-1}\sigma^i(f)$
is $G$-invariant, and it maps to $f^p$ in $k_U$
by triviality of the $G$-action there. So the residue field $k_{U/G}
\subset k_U$ at the image of the point $P$
contains $k_U^p$. Since we assume that $k_U$ is perfect,
$k_{U/G}$ is equal to $k_U$.
In what follows, we replace
$U$ by a $G$-invariant open neighborhood of $P$ as needed,
since we are only
trying to describe $U/G$ near the image of $P$ (which we also
call $P$). In particular,
we can assume that $U$ is affine.

The fixed point scheme $U^G$
is defined as the closed subscheme of $U$ cut out by
the ideal generated by $I(O(U))$.
The following formulas hold for any action of $G$
on a commutative ring
\cite[Remark 3]{KL}:

\begin{lemma}
\label{product}
(1) $I(xy)=I(x)\sigma(y)+xI(y).$

(2) For $m\geq 0$, $I(x^m)=I(x)\sum_{i=1}^m \sigma(x)^{i-1}x^{m-i}.$
\end{lemma}

We are given coordinates $x_1,\ldots,x_n$ for $U$ near $P$.
Under our assumptions ($U$ regular and $k_{U/G}=k_U$),
the ideal generated by $I(O_{U,P})$ in $O_{U,P}$ is generated
by $I(x_1),\ldots,I(x_n)$ \cite[Proposition 6]{KL}. That is,
after shrinking $U$ around $P$ if necessary, the fixed point scheme
$U^G$ is the closed subscheme defined by $I(x_1),\ldots,I(x_n)$.

In particular, if $I(x_i)/e$ is a unit for some $i=1,\ldots,n$, then
the fixed point scheme $U^G$ is the Cartier divisor
$\{e=0\}$, and then Theorem \ref{kl} gives that $U/G$
is regular. {\it So we can assume from now on that $I(x_i)$ is in
$e\m$ for each $i$. }Equivalently, after shrinking $U$
around $P$, $I(O(U))$ is contained in $e\m$.
In this case, we will show that $U/G$
has a $\mu_p$-quotient singularity at $P$.

The following lemma is implicit in the statement
of Theorem \ref{mup}.

\begin{lemma}
\label{map}
Let $U$ be a regular scheme with an action of the group
$G=\Z/p=\langle \sigma:\sigma^p=1\rangle$.
Suppose that $G$ fixes a closed point $P$ with perfect
residue field $k_U$ of characteristic $p$. Write $\m$
for the maximal ideal in the local ring $O_{U,P}$. Let
$e\in \m$, $e\neq 0$, such that
$I:=\sigma-1$ satisfies $I(\m)\subset e\m$.
Then $\varphi(y):=I(y)/e$ is a well-defined $k_U$-linear
map from $\m/\m^2$ to itself.
\end{lemma}

\begin{proof}
Since the local ring $O_{U,P}$ is regular, it is a domain, and so $I(y)/e$
is well-defined for each element $y$ in $\m$.
Since $G$ fixes the point $P$, $G$ maps $\m$ into itself.
Since $I(\m)\subset e\m$, Lemma \ref{product} gives
that $I(\m^2)\subset e\m^2$. Since $I$ is additive,
it follows that $\varphi$ is a well-defined additive function
from $\m/\m^2$ to itself. By Lemma \ref{product}, $\varphi$ is linear
over the ring of invariants $(O_{U,P})^G$. Since that ring has residue
field $k_{U/G}=k_U$, $\varphi$ is $k_U$-linear.
\end{proof}

We are given a function $s\in\m-\{0\}$ such that $I(s)=es(\text{unit})$.
After multiplying $e$ by a unit, we can assume that $I(s)=es$; this does
not change the other assumption that $I(x_i)\in (e)$ for $i=1,\ldots,n$.
{\it Thus, from now on, we have $I(s)=es$. }This changes
the endomorphism
$\varphi$ of $\m/\m^2$ (which is defined in terms of $e$)
by a nonzero scalar.
Having made this change,
we will show that $\varphi$ is diagonalizable, with eigenvalues
$b_1,\ldots,b_n$ in $\F_p\subset k_U$, and that $U/G$
has a $\mu_p$-quotient singularity of the form
$\frac{1}{p}(b_1,\ldots,b_n)$.

Let $v=\sigma(s)/s$. By our assumptions, $v$ is a unit
on $U$, and $v=1+e$. Since $v=\sigma(s)/s$, $v$ has norm 1
for the action of $G$. Write
$(\sigma/\id)(x)$ for $\sigma(x)/x$; this is the multiplicative action of
$\sigma-1\in \Z G$ on a commutative ring. In these terms,
define $f=(\sigma/\id)^{p-2}(v)$. In the group ring $\Z G$, we have
$(\sigma-1)^{p-1}\equiv \sigma^{p-1}+\cdots+\sigma+1\pmod{p}$,
and so we can define an element $\alpha\in \Z G$ by
\begin{equation}
(\sigma-1)^{p-1}=\sigma^{p-1}+\cdots+\sigma+1-p\alpha. \tag{$\ast$}
\end{equation}
It follows that $\sigma(f)/f=N(v)/g^p$, with $g:=\alpha(v)$
(where $\alpha$ acts multiplicatively). Since $v$ has norm 1,
we have $\sigma(f)/f=1/g^p$. This formula will be exactly what we need
to construct a $\mu_p$-torsor $W\to U$ with a commuting action
of $G$.

We first analyze the function $g$ in more detail. By equation (*),
the sum of the coefficients of $\alpha$ in $\Z G$ is 1.
Therefore, $g=\alpha(v)$ is a product of integer powers
of the functions $\sigma^j(v)=1+e(1+\m)$ with total exponent 1,
and so $g=1+e(1+\m)$.

\begin{lemma}
\label{root}
For any $p$th root of unity $\zeta$ in $O(U)$,
$1-\zeta$ is in $e\m$.
\end{lemma}

\begin{proof}
The statement is clear if $\zeta=1$. So assume that $\zeta\neq 1$.
Since $O(U)$ is a domain and $(1-\zeta)(1+\zeta+\cdots+\zeta^{p-1})=
1-\zeta^p=0$, we must have $1+\zeta+\cdots+\zeta^{p-1}=0$.
That is, we have a homomorphism from the ring of integers
$\Z[\mu_p]$ in $\Q(\mu_p)$ to $O(U)$, taking a primitive $p$th root
of unity to $\zeta$. In $\Z[\mu_p]$, $p$ is $(1-\zeta)^{p-1}$ times a unit
\cite[section IV.1]{Lang},
and so the same is true in $O(U)$. We are given that
$p$ is in $e^{p-1}\m$. Since $U$ is regular, $O_{U,P}$ is a unique
factorization domain, and so
$1-\zeta$ must be a multiple of $e$; write $1-\zeta =ea$
for some $a\in O_{U,P}$.
If $a$ is a unit, then $p$ would be $e^{p-1}$ times a unit, a contradiction.
So, as an element of $O_{U,P}$,
$1-\zeta$ is in $e\m$. After shrinking $U$ around $P$ if necessary,
this gives the same conclusion in $O(U)$.
\end{proof}

The reason for constructing units $f$ and $g$ with $\sigma(f)/f=1/g^p$
is to define a $\mu_p$-torsor over $U$ with a commuting
action of $G$.
Namely, define a $\mu_p$-torsor $W\to U$ by $w^p=f$.
Here $\mu_p$
acts on $W\subset A^1\times U$ by $\zeta(w,x)=(\zeta w,x)$,
for $\zeta\in\mu_p$. Write $\tau=\sigma^{-1}$ in $G$.
Since $\sigma(f)/f=1/g^p$, $W$ has an action of $G$ that commutes
with the action of $\mu_p$, by $\tau(w,x)=(w/g(x),
\tau(x))$. (We check these properties in the next paragraph.)
In particular, $\sigma(w)=\tau^*(w)=w/g$,
by definition of the $G$-action on functions.
We will show that the scheme $Q:=W/G$ is regular. Then
$U/G=Q/\mu_p$ will be a quotient of a regular scheme by $\mu_p$,
as we want.
$$\xymatrix@R-10pt{
W \ar[r]^{G} \ar[d]^{\mu_p} & Q\ar[d]^{\mu_p} \\
U \ar[r]^{G} & U/G.
}$$
For convenience, we write $P$ for the closed point of interest
in each of these schemes. (There is a unique closed point in $W$ over
$P$ in $U$,
and it maps to a closed point in $Q$ and in $U/G=Q/\mu_p$.) We have
seen that the points $P\in U$ and $P\in U/G$ have the same residue
field $k_U$. Since $P\in W$ is given by the equations $w=1$
and $x_1=\cdots=x_n=0$, we see that $P\in W$ has the same residue field
$k_U$. Since $k_{U/G}\subset k_Q\subset k_W$, $P\in Q$ also has the same
residue field.

For clarity, let us first check that the formulas above give
an action of $G$ on $W$ that commutes with the action of $\mu_p$.
First, to show that $\sigma$ as above maps $W=\{w^p=f\}$ into itself,
we have to show that if $w^p=f(x)$, then $(w/g(x))^p=f(\sigma^{-1}(x))$,
or equivalently that $w^p/g^p=\sigma(f)$;
this follows from the fact that $\sigma(f)/f=1/g^p$. Next, let us show
that $\sigma^p=1$ on $W$.
By induction, we have $\sigma^{-i}(w,x)
=(w/(g(x)g(\sigma^{-1} x)\cdots g(\sigma^{1-i}x)),\sigma^{-i}x)$
for each natural number $i$. Therefore, to show that $\sigma^p$ is the identity
on $W$ (hence that we have an action of $G$), it suffices to show
that $g$ has norm 1. But that is true, because $v$ has norm 1 and $g$
is a product of integer powers of $\sigma^j(v)$ for integers $j$. So we have
an action of $G$ on $W$. Finally,
to show that $G$ and $\mu_p$ commute on $W$:
for $\zeta\in\mu_p$, we have 
$\zeta \sigma^{-1}(w,x)=\zeta(w/g(x),\sigma^{-1}(x))
=(\zeta w/g(x),\sigma^{-1}(x))$
while $\sigma^{-1}\zeta(w,x)=\sigma^{-1}(\zeta w,x)
=(\zeta w/g(x),\sigma^{-1}(x))$.
These are equal as regular functions on the scheme $\mu_p\times W$.
So we have shown that $G$ and $\mu_p$ commute on $W$.

We need to show that $f$ is not a $p$th power in $O(U)$.
Suppose it is, say $f=u^p$ for a regular function $u$ on $U$.
Since $f$ is a unit, so is $u$. Then $1/g^p=\sigma(f)/f
=(\sigma(u)/u)^p$, and so $\zeta/g=\sigma(u)/u$ for some $p$th
root of unity $\zeta$ in $O(U)$. Here $\zeta=1+(\zeta-1)
=1+e\m$ by Lemma \ref{root}, and $1/g=1-e(1+\m)$,
so $\sigma(u)/u=1-e(1+\m)$. Here $\sigma(u)/u=1+I(u)/u$,
so $I(u)=-ue(1+\m)=e(\text{unit})$. This contradicts
our assumption that $I(O(U))$ is contained in $e\m$. So in fact
$f$ is not a $p$th power.

From there, we can show that the scheme $W$ is integral
(after shrinking $U$ around $P$,
if necessary). Namely, since $f$ is not a $p$th power
in $O(U)$, $f$ is also not a $p$th power in the function
field $k(U)$, and so $k(U)[f^{1/p}]$ is a degree-$p$ field extension
of $k(U)$. Write $\alpha$ for the $\mu_p$-torsor $W\to U$.
Since $W=\{w^p=f\}$,
there is a nonempty open subset $V\subset U$ with $\alpha^{-1}(V)$
integral. Since $W\to U$ is finite and flat, $W$ is
Cohen-Macaulay and equidimensional.
By equidimensionality,
every irreducible component of $W$ must dominate $U$.
Since $\alpha^{-1}(V)$ is irreducible, it follows
that $W$ is irreducible. Since $\alpha^{-1}(V)$ is reduced,
$W$ is reduced in codimension 0; since $W$ is Cohen-Macaulay,
it follows that $W$ is reduced \cite[Tag 031R]{Stacks}.
Since $W$ is reduced and irreducible,
it is integral.

It is not needed for what follows, but for clarity,
let us analyze the singularities of $W$ in the special
case where $s=e(1+\m)$ (for example when $s=e$),
so that $I(e)=e^2(1+\m)$. The equation of $W$
is $w^p=f$. Since $v=1+e$, one can show by induction
from the formula for $I(e)$
that $f=(\sigma/\id)^{p-2}(v)=1+e^{p-1}q$ for some unit $q$ on $U$,
and so we can rewrite the equation of $W$ as $w^p=1+e^{p-1}q$.
Let $w_0=w-1$, so that $w_0$ vanishes at the point $P$ of interest
in $W$. In terms of $w_0$, the equation of $W$ becomes
$(1+w_0)^p=1+e^{p-1}q$, that is,
$$w_0^p=e^{p-1}q-\binom{p}{1}w_0-\cdots-\binom{p}{p-1}w_0^{p-1}.$$
We are given that $p$ is in $e^{p-1}\m$, and so this equation
has the form
$w_0^p=e^{p-1}s$ for some unit $s$ on $W$.
If $p=2$ and $e\not\in\m^2$, it follows that $W$ is regular.
However, if $p>2$, then $W$ is not normal. For example,
if $e\not\in\m^2$, then the singularity of $W$ looks roughly
like the cuspidal curve $\{ w_0^p=x_1^{p-1}\}$ times a smooth variety.

We will need the following version of Kir\'aly and L\"utkebohmert's
results.

\begin{theorem}
\label{klnormal}
\hspace{2em}
\begin{enumerate}
\item Let $B$ be a local domain with residue field $k_B$.
Let $p$ be a prime number, and let $G=\Z/p$ act nontrivially
on $B$. Suppose that the ideal $B\cdot I(B)$ that defines
the fixed point scheme in $\Spec B$ is generated by one element.
Then $B$ is free of rank $p$ over the ring of invariants $B^G$.
More precisely,
for any element $t$ such that $I(t)$ generates the ideal $B\cdot I(B)$,
we have $B=A\{ 1,t,\ldots,t^{p-1}\}$.
\item In addition to the assumptions of (1), assume that
$B$ is regular. Then $B^G$ is regular.
\item In addition to the assumptions of (1), assume that $B$ is noetherian
and the inclusion $k_{B^G}\subset k_B$
is an equality. Then there is a minimal set of generators
$y_1,\ldots,y_r$ for $\m_B$ such that $I(y_1)$ generates
the ideal $B\cdot I(B)$ and $y_2,\ldots,y_r$
are $G$-invariants.
\end{enumerate}
\end{theorem}

\begin{proof}
Statement (1) is due to Kir\'aly and L\"utkebohmert for $B$ normal,
but their proof works without change for $B$ a domain
\cite[Theorem 2 and Proposition 5]{KL}. They also prove
statement (2). They prove statement (3)
when $B$ is regular. We now extend the proof of (3)
for $B$ only a domain.

Since the inclusion $k_{B^G}\subset k_B$ is an equality,
we have $B=B^G+\m_B$, and so $I(B)=I(\m_B)$. Since the ideal
$B\cdot I(B)$ is generated by one element, there is an element
$y_1\in\m_B$ such that $I(y_1)$ generates this ideal.
By Lemma \ref{product},
we have $I(\m_B^2)\subset \m_B I(\m_B)$. Here $I(\m_B)$ is not zero
(because the $G$-action is nontrivial), and so $y_1$ is not in $\m_B^2$.

Since $B$ is noetherian, the ideal $\m_B$ is finitely generated.
Choose elements $z_2,\ldots,z_r$ in $\m_B$ such
that $y_1,z_2,\ldots,z_r$ form a basis for $\m_B/\m_B^2$.
By part (1), we know that $B=B^G\{1,y_1,\ldots,y_1^{p-1}\}$.
For $2\leq i\leq r$, let $y_i$ be the projection of $z_i$
to $B^G$ with respect to this decomposition.
Then $y_i\equiv z_i\pmod{(y_1)+\m_B^2}$,
and so $y_1,\ldots y_r$ map to a basis for $\m_B/\m_B^2$.
Thus $y_1,\ldots,y_r$ are a minimal set of generators
for $\m_B$ (by Nakayama's lemma), and $y_2,\ldots,y_r$
are $G$-invariant.
\end{proof}

Let us write out the action of $G$ on $W$. The maximal ideal
of $P$ in $W$ is generated by $w_0,x_1,\ldots,x_n$.
We have $I(w_0)=I(w)=w(\frac{1}{g}-1)=(1+w_0)(\frac{1}{g}-1)$.
Since $1/g=1+eu$ for some unit $u$ on $U$,
we have $I(w_0)=eu(1+w_0)=e(\text{unit})$. We also have
$I(x_i)$ in the ideal $(e)$ for $i=1,\ldots,n$; so the fixed point
scheme $W^G$ is defined by the single equation $e=0$ in $W$.
As a result, even though $W$ is typically not normal,
Theorem \ref{klnormal} gives that
the morphism $W\to W/G=Q$ is finite and faithfully flat of degree $p$.
It follows that $Q$
is noetherian \cite[Tag 033E]{Stacks}.
(Beware that for a general noetherian scheme $X$
with an action of a finite group $G$, $X/G$ need not be noetherian,
and the morphism $X\to X/G$ need not be finite
\cite[Proposition 0.10]{Nagata}. These properties
do hold if $X$ is of finite type over a noetherian ring $A$ and $G$
acts $A$-linearly \cite[Theorem and Corollary 4]{Fogartyfinite}.)

The action of $\mu_p$ on the affine scheme $Q$
gives a grading of $O(Q)$ by $\Z/p$.
For each $j\in \Z/p$, since $O(Q)$ is noetherian, the ideal in $O(Q)$
generated
by the $j$th graded piece $O(Q)_j$ is finitely generated,
and so $O(Q)_j$ is a finitely generated module over $O(Q)_0=O(Q/\mu_p)$.
So the whole ring $O(Q)$ is finite over $O(Q/\mu_p)$; that is,
$Q\to Q/\mu_p$ is finite. Also, $O(Q/\mu_p)$ is a pure subring
(because it is a summand) of the noetherian ring $O(Q)$, so it is
noetherian \cite[Proposition 6.15]{HR};
that is, $Q/\mu_p=U/G$ is noetherian. Finally, the composition
$W\to Q\to Q/\mu_p=U/G$ is finite, and $O(U)$ is a sub-$O(U/G)$-module
of $O(W)$, so $O(U)$ is a finitely generated $O(U/G)$-module;
that is, $U\to U/G$ is finite.

Let $h_1,\ldots,h_r$ be a minimal set of generators
for the maximal ideal at $P$ of $O(Q)$. (So $r$ is at least the dimension
$n$ of $Q$.)

\begin{lemma}
\label{sameideal}
The ideals $(h_1,\ldots,h_r)$ and $(x_1,\ldots,x_n)$ in $O(W)$ are equal.
(That is: the fiber in $W$ over the closed point $P\in Q$
is equal to the fiber in $W$ over the closed point $P\in U$, as a closed
subscheme.)
\end{lemma}

\begin{proof}
We have seen that the degree-$p$ morphism $W\to W/G=Q$
is finite and flat.
So the fiber in $W$ over the point $P$ in $Q$ has degree $p$
over the residue field of $P\in Q$, which we have seen is $k_U$.
As a set, this fiber is one point $P\in W$, with the same residue
field $k_U$. So the quotient ring $O(W)/(h_1,\ldots,h_r)$ is an artinian
local ring of length $p$.

The ideal $(x_1,\ldots,x_n)$ in $O(W)$ defines the fiber in $W$
over the point $P$ in $U$. Since $W\to U$ is a $\mu_p$-torsor, this fiber
has degree $p$ over the residue field of $P\in U$,
which is the same field $k_U$. Again, this fiber is one point
$P\in W$ as a set, with the same residue field; so
$O(W)/(x_1,\ldots,x_n)$ is an artinian local ring of length $p$.
So if we can show that the ideal $(h_1,\ldots,h_r)$ in $O(W)$
is contained in $(x_1,\ldots,x_n)$ in $O(W)$, then they are equal. 

It suffices to show (*) that every function $y$ on $W$ that vanishes
at the point $P$ in $W$ but has nonzero image in $O(W)/(x_1,\ldots,x_n)$
has $I(y)\neq 0$ (that is, it is not $G$-invariant). (Namely, this would
imply that the $G$-invariant functions $h_1,\ldots,h_n$ on $W$
lie in the ideal $(x_1,\ldots,x_n)$, as we want.) By the formula
for the $G$-action on $W$, in particular that $\sigma(w)=w/g$ where
$g=1+e(1+\m)$, we see that $G$ fixes the closed subscheme $\{ e=0
\}$ in $W$. That is, $I$ maps $O(W)$ into the ideal $(e)$ in $O(W)$.
Also, we know that $I(x_i)\in e\m_U=e(x_1,\ldots,x_n)\subset O(U)$
for $i=1,\ldots,n$.
So $\varphi(y):=I(y)/e$ is a well-defined linear map from
$O(W)/(x_1,\ldots,x_n)$ to itself. Explicitly, by the equation
of $W$, $O(W)/(x_1,\ldots,x_n)$ is a $k_U$-vector space with basis
$1,w,\ldots,w^{p-1}$. Equivalently, in terms of $w_0=w-1$
(which vanishes at $P$ in $W$), a basis for $O(W)/(x_1,\ldots,x_n)$
is given by $1,w_0,\ldots, w_0^{p-1}$.

The claim (*) will follow if the map $\varphi$ restricted
to $k\{ w_0,w_0^2,\ldots,w_0^{p-1}\}$ is injective. Since
$g=1+e(1+\m_U)$, we have $1/g=1-e(1+\m_U)$, and hence
$I(w_0)=I(w)=(w/g)-w=-ew(1+\m_U)=-e(1+w_0)(1+\m_U)$.
So $\varphi(w_0)=-(1+w_0)$. By Lemma \ref{product}, for $m\geq 0$,
\begin{align*}
I(w_0^m)&=I(w_0)\sum_{j=1}^m \sigma(w_0)^{j-1}w_0^{m-j}\\
&=-e(1+w_0)(1+\m_U)\sum_{j=1}^m(w_0-(1+w_0)e(1+\m_U))^{j-1}w_0^{m-j}.
\end{align*}
Since $\varphi$ takes values in $O(W)/(x_1,\ldots,x_n)$ (where
$e$ is zero), it follows that
$\varphi(w_0^m)=-mw_0^m(1+w_0)$. It is clear that these elements are
linearly independent over $k_U$
for $m=1,\ldots,p-2$; to show that they are linearly
independent for $m=1,\ldots,p-1$,
it will suffice to show that $w_0^p$ is zero in $O(W)/(x_1,\ldots,x_n)
=k_U\{1,w_0,\ldots,w_0^{p-1}\}$.
(This comes up because $w_0^p$ appears in our formula for $I(w_0^{p-1})$.)

Namely, we have $w_0^p=(w-1)^p=w^p-1$ plus a multiple of $p$ in $O(W)$.
Here $w^p-1=f-1=e^{p-1}(1+\m_U)$, and $p$ is in $e^{p-1}\m_U$,
as we assumed; so $w_0^p=e^{p-1}(1+\m_U)$, which is zero
in $O(W)/(x_1,\ldots,x_n)$, as we want. Lemma \ref{sameideal} is proved.
\end{proof}

The number $r$ of generators $h_1,\ldots,h_r$ for the maximal ideal
$\m_Q$ in the local ring $O_{Q,P}$ is at least $n=\dim(Q)$.
On the other hand, Lemma \ref{sameideal} implies that the extended
ideal $(h_1,\ldots,h_r)$ in $O_{W,P}$ can be generated by only $n$ elements,
so the vector space $(h_1,\ldots,h_r)\otimes_{O_{W,P}} k_W$
has dimension at most $n$. So this vector space is spanned
by $n$ of the $h_i$'s, which we can assume are $h_1,\ldots,h_n$.
By Nakayama's lemma, it follows that the extended ideal $(h_1,\ldots,h_r)$
is equal to the extended ideal
$(h_1,\ldots,h_n)$ in $O_{W,P}$. Since $W\to Q$ is faithfully
flat, extending and contracting an ideal in $O_{Q,P}$ gives the same ideal
\cite[Tag 05CK]{Stacks}.
As a result, we have $(h_1,\ldots,h_n)=(h_1,\ldots,h_r)$ in $O_{Q,P}$.
That is, the maximal ideal $\m_Q$ can be generated by $n=\dim(Q)$
elements, which means that $Q$ is regular. (This is somewhat surprising,
since $W$ is typically not regular or even normal.)

It remains to show that the point $P$ in $Q$ is a
fixed point for $\mu_p$, with weights given by the eigenvalues
of $\varphi$.
First, let us show that $\mu_p$ fixes the point $P$ in $Q$. (This does not
seem obvious, since $\mu_p$ does not fix the point $P$ in $W$;
in fact, $\mu_p$ acts freely on $W$.)
The functions $x_1,\ldots,x_n$ on $W$ are pulled back from $U$,
hence fixed by $\mu_p$. As a result, the ideal $(x_1,\ldots,x_n)$
in $O(W)$ is preserved by $\mu_p$. Equivalently, by Lemma 
\ref{sameideal}, the ideal $(h_1,\ldots,h_n)$ in $O(W)$
is preserved by $\mu_p$. We have seen that
the morphism $W\to W/G=Q$ is faithfully flat.
As a result,
the ideal $(h_1,\ldots,h_n)$ in $O(Q)$ is equal to the intersection
of the extended ideal $(h_1,\ldots,h_n)$ in $O(W)$ with $O(Q)$.
Since $W\to Q$ is $\mu_p$-equivariant,
it follows that the ideal $(h_1,\ldots,h_n)=\m_Q$ in $O(Q)$
is preserved by $\mu_p$. Also, the residue field of $Q/\mu_p$ at $P$
maps isomorphically to the residue field of $Q$ at $P$, and so $\mu_p$
acts trivially on the latter field.
That is, $\mu_p$ fixes the point $P$ in $Q$,
as we want.

We now change our choice of the functions $h_1,\ldots,h_n$.
Since $\mu_p$ is linearly reductive,
we can choose coordinates
$h_1,\ldots,h_n$ for $Q$ near $P$ that are $\mu_p$-eigenfunctions.
That is, each $h_i$ has some weight $b_i\in \Z/p$ for the action
of $\mu_p$. In these terms, $Q/\mu_p=U/G$ is a toric singularity
of type $\frac{1}{p}(b_1,\ldots,b_n)$. It remains to show
that the endomorphism $\varphi$ of $\m_U/\m_U^2$ is diagonalizable,
with eigenvalues in $\F_p\subset k_U$, and that these eigenvalues
are equal to $b_1,\ldots,b_n$.

Consider $h_1,\ldots,h_n$ as $G$-invariant functions on $W$.
Here $W$ is a $\mu_p$-torsor over $U$ defined by $w^p=f$;
so $O(W)=O(U)\{1,w,\ldots,w^{p-1}\}$, and this grading by $\Z/p$ describes
the action of $\mu_p$ on $O(W)$. For $i=1,\ldots,n$, $h_i$ has weight $b_i$
for the action of $\mu_p$, and so we can write $h_i=g_i w^{b_i}$
for some regular function $g_i$ on $U$. (Here we think of $b_i$
as an integer in $\{0,\ldots,p-1\}$.) Clearly the functions $g_1,\ldots,
g_n$ vanish at $P$ (using that $w$ is a unit). Also,
$(g_1,\ldots,g_n)$ is equal to $(h_1,\ldots,h_n)$
as an ideal in $O(W)$, and we showed that the latter ideal
is equal to $(x_1,\ldots,x_n)$ in $O(W)$. Since $W\to U$
is faithfully flat, it follows that $(g_1,\ldots,g_n)$
is equal to $(x_1,\ldots,x_n)$ as an ideal in $O(U)$.
That is, $g_1,\ldots,g_n$ form
coordinates on $U$ near $P$.

For $1\leq i\leq n$, we have $I(g_i)\equiv e\varphi(g_i)\in e(\m_U/\m_U^2)$,
by definition of the endomorphism $\varphi$ of $\m_U/\m_U^2$.
We showed above that $I(w_0)=-e(1+w_0)(1+\m_U)$,
and so $I(w)=I(w_0)=-ew(1+\m_U)$.
For each $b\geq 0$, Lemma \ref{product} gives that
\begin{align*}
I(w^b)&=
I(w)\sum_{m=1}^b \sigma(w)^{m-1}w^{b-m}\\
&=-bew^b(1+\m_U),
\end{align*}
using that $e$ is in $\m_U$.
Since the function $g_i w^{b_i}$ is $G$-invariant on $W$,
we have $0=I(g_i w^{b_i})=\sigma(g_i)I(w^{b_i})
+I(g_i)w^{b_i}=(g_i+I(g_i))I(w^{b_i})+I(g_i)w^{b_i}$.
When we consider this equality modulo $e\m_U^2O(W)$,
the term $I(g_i)I(w^{b_i})$ can be omitted.
Namely, we have
\begin{align*}
0&\equiv e(-b_ig_iw^{g_i}(1+\m)+\varphi(g_i)w^{g_i})\pmod{e\m_U^2O(W)}\\
&\equiv ew^{g_i}(-b_ig_i+\varphi(g_i)) \pmod{e\m_U^2O(W)}.
\end{align*}
Since $w$ is a unit, it follows that $0\equiv e(-b_ig_i+\varphi(g_i))
\pmod{e\m_U^2O(W)}$. Since $W\to U$ is faithfully flat,
$e\m_U/e\m_U^2\to e\m_UO(W)/e\m_U^2O(W)$ is injective,
and so $0\equiv e(-b_ig_i+\varphi(g_i))
\pmod{e\m_U^2}$. So
$\varphi(g_i)=b_ig_i$
in $\m_U/\m_U^2$ for each $i=1,\ldots,n$. Also,
$g_1,\ldots,g_n$ form a basis
for $\m_U/\m_U^2$. So $\varphi$ is diagonalizable,
its eigenvalues $b_1,\ldots,b_n$ are in $\F_p$, and $U/G$
is a $\mu_p$-quotient singularity of the form
$\frac{1}{p}(b_1,\ldots,b_n)$,
as we want.
\end{proof}

\section{Review of ramification theory}
\label{ramsection}

We recall here how to compute the ramification behavior
of a $\Z/p$-covering in characteristic $p$
or mixed characteristic, following Xiao and Zhukov \cite{XZ}.

Let $G=\Z/p=\langle\sigma:\sigma^p=1\rangle $ act nontrivially
on a normal noetherian integral scheme $Y$. Assume
that $Y$ is of finite type over a field or over $\Z_p$, so that
we can talk about the canonical class $K_Y$. Write $f$
for the quotient map $Y\arrow Y/G$. For each irreducible divisor
$E$ in $Y$ that is mapped into itself by $G$,
let $F$ be its image (as an irreducible divisor)
in $Y/G$. Assume that $p=0$ on $E$.
There are several invariants we want to compute
in this situation:
the {\it ramification index }of the divisor $E$ in $Y$
(the positive integer
$e$ such that $f^*F=eE$), and the coefficient $c$ of $E$
in the {\it ramification divisor }(meaning that $K_{Y}=f^*K_{Y/G}+
cE$ near the generic point of $E$). Another name for $c$
is the {\it valuation of the different }$v_{E}(\mathcal{D}_{k(Y)/k(Y/G)})$,
where the valuation $v_{E}$ on the function field $k(Y)$
is the order of vanishing along $E$.
Here $ef=p$, where $f$ is the degree of the field extension
$k(E)$ over $k(F)$.

An easy case is where $G$ does not fix $E$ (in other words,
where $G$ acts nontrivially on $E$). We say that $f$ is
{\it unramified }along $E$. In this case, $f^*F=E$,
and $K_Y=f^*K_{Y/G}$ near the generic point of $E$.

Define
the {\it Artin ramification number }$i(E)$ of $Y$ over $Y/G$
along $E$ to be the coefficient of $E$ in the fixed point scheme $Y^G$.
Equivalently, in terms of $I(a)=\sigma(a)-a$:
$$i(E)=\min_{a\in O_{Y,E}} v_{E}(I(a)).$$

In the ramified case, the field $k(E)$ is purely
inseparable over $k(F)$, with degree $f$ equal to 1 or $p$. We can distinguish
the two cases as follows.
Since $ef=p$, either $e=p$ and $f=1$ (called {\it wild }ramification)
or $e=1$ and $f=p$ (called {\it fierce }ramification). Let $t$
be a defining function of $E$, that is, a rational function on $Y$
with valuation $v_E(t)=1$. It is clear that $v_E(I(t))\geq i(E)$.
A very convenient criterion is: $Y\to Y/G$ is wildly ramified
along $E$
if and only if equality holds
\cite[section 2.1]{XZ}.
Otherwise, it is fiercely ramified.

Furthermore, we can compute the relative canonical class
(that is, the valuation of the different) as follows,
correcting a typo in \cite[section 2.1]{XZ}.

\begin{lemma}
\label{different}
The valuation of the different is $(p-1)i(E)$.
When $Y$ is regular, so that $Y/G$
is $\Q$-factorial, we can equivalently say that
$$K_Y=f^*K_{Y/G}+(p-1)[Y^G],$$
where $[Y^G]$ denotes the Weil divisor associated to the fixed
point scheme.
\end{lemma}

\begin{proof}
If $G$ acts nontrivially on $E$ (the unramified case),
then $i(E)=0$ and the statement is clear. So assume that $G$ fixes $E$.
The local ring $O_{Y,E}$ is a discrete valuation ring. The algebra
$O_{Y,E}$ is generated by one element $y$ as an algebra over $O_{Y/G,F}$;
one can take $y$ to be a uniformizer in $O_{Y,E}$ if $E$ is wildly ramified,
and an element of $O_{Y,E}$ whose restriction to $k(E)$ is not in $k(F)$
if $E$ is fiercely ramified \cite[section 2.1]{XZ}. By Lemma \ref{product},
we have $i(E)=v_E(I(y))$.

Let $u(X)$ be the minimal
polynomial of $y$ over $k(Y/G)$. In this situation of a monogenic algebra
extension, the different $\mathcal{D}_{k(Y)/k(Y/G)}$ is generated by
$u'(y)$ \cite[III, Corollary 2 to Proposition 11]{Serre}. But $u(X)=
\prod_{j=0}^{p-1}(X-\sigma^j(y))$. So
$$u'(y)=\prod_{j=1}^{p-1}(y-\sigma^j(y)),$$
and hence
$$v_E(\mathcal{D}_{k(Y)/k(Y/G)})=v_E(u'(y))=(p-1)i(E),$$
using that $i(E)$ is unchanged if we replace $\sigma$ by another
generator of $G$.
\end{proof}

\section{Toric divisors}

In addition to recognizing when a quotient by $G=\Z/p$
has a $\mu_p$-quotient singularity (as in Theorem \ref{mup}),
Theorem \ref{mup-divisor} analyzes when a $G$-invariant divisor
is pulled back from a divisor on the quotient. Using this,
we can view certain $G$-invariant divisors as toric divisors
on the quotient scheme, which will be convenient for applications.

\begin{theorem}
\label{mup-divisor}
\hspace{2em}
\begin{enumerate}
\item
Let $G=\Z/p$ act on a regular scheme $U$, with the assumptions
of Theorem \ref{mup}. Assume moreover that the function $e$
is the greatest common divisor of the functions $I(x_i)$
for $i=1,\ldots,n$ in the local ring $O_{U,P}$.
Then, for each $y\in \m_U$ such that $\{y=0\}$ is an irreducible
divisor $E\subset U$ and $I(y)\in (ey)$, $E$ is the pullback
of a Weil divisor $F$ in $U/G$.
\item If in addition $I(x_i)\in e\m_U$ for $i=1,\ldots,n$, so that $U/G$
is the quotient of a regular scheme $Q$ by $\mu_p$ (by Theorem \ref{mup}),
then the pullback of $F$ to $Q$ is the divisor $\{ h=0\}$ for some
$\mu_p$-eigenfunction $h$ on $Q$.
\item Continue to assume that $I(x_i)\in e\m_U$ for $i=1,\ldots,n$,
so that $U/G$ is the quotient of a regular scheme $Q$ by $\mu_p$.
Let $y_1,\ldots,y_r$ be functions on $U$, vanishing at $P$,
that are linearly independent in $\m_U/\m_U^2$. Suppose that $I(y_j)\in (ey_j)$
for $j=1,\ldots,r$. Then the corresponding $\mu_p$-eigenfunctions
$h_1,\ldots,h_r$ on $Q$ (from (2)) are linearly independent
in $\m_Q/\m_Q^2$. That is, these functions are part of a toric
coordinate system on $Q$. Finally, if $I(s)=es$ in Theorem \ref{mup}
(as we can assume), then the $\mu_p$-weight of $h_j$ is equal
to the eigenvalue of $\varphi(y):=I(y)/e$ on $y_j\in \m_U/\m_U^2$,
which is in $\F_p$.
\end{enumerate}
\end{theorem}

\begin{proof}
(1) Let $E$ be the irreducible divisor $\{y=0\}$ in $U$.
The assumptions imply that $E$ is mapped into itself by $G$, and that
there is an $i\in \{1,\ldots,n\}$ such that $v_E(I(y))>v_E(I(x_i))$.
By section \ref{ramsection}, $f\colon U\to U/G$ is either unramified
(if $E$ is not fixed by $G$) or fiercely ramified
along $E$. In either case, there is an irreducible divisor
$F$ on $U/G$ such that $E=f^*F$, as we want. (The divisor
$F$ need not be Cartier, but the pullback of a Weil divisor
is still a Weil divisor (with integer coefficients). Indeed,
$F$ is a Cartier divisor outside a codimension-2 subset of $U/G$,
by normality of $U/G$.)

(2) Since $Q$ is regular, the pullback of $F$ to $Q$
is a Cartier divisor, hence (after shrinking $U$ and $Q$
around $P$)
of the form $\{t=0\}$ for some function $t\in \m_Q-\{0\}$.
Clearly the divisor $\{t=0\}$ is $\mu_p$-invariant. I claim that
$t$ times some unit is a $\mu_p$-eigenfunction on $Q$.
Indeed, in algebraic terms,
the action of $\mu_p$ on $Q$ makes $O(Q)$ a comodule over $O(\mu_p)$,
and the ideal $(t)$ is an sub-$O(\mu_p)$-comodule. Every $O(\mu_p)$-comodule
(with no finiteness assumption needed) is the direct sum
of its weight spaces, indexed by $\Z/p$. So we can write
$t=t_0+\cdots+t_{p-1}$ with $t_i\in (t)$ and $t_i$ of weight $i$.
Since $t_i\in (t)$, we can write $t_i=a_it$ for some $a_i\in O_{Q,P}$.
Since $O_{Q,P}$ is regular, it is a domain, and hence $1=a_0+\cdots
+a_{p-1}$. So at least one $a_i$ is not in $\m_Q$, hence is a unit.
Then $h:=t_i=a_it$ is a unit times $t$ and also a $\mu_p$-eigenfunction
(of weight $i$), as we want.

(3) After multiplying $e$ by a unit, we can assume that $I(s)=es$,
in the terminology of Theorem \ref{mup}. The assumption that
$I(y_j)\in (ey_j)$ for $j=1,\ldots,r$ implies that $y_1,\ldots,y_r$
in $\m_U/\m_U^2$ are eigenvectors of the map $\varphi$. By the proof
of Theorem \ref{mup}, the corresponding eigenvalues are in $\F_p
\subset k_U$.

For $j=1,\ldots,r$, we know from (1) that the divisor
$\{y_j=0\}$ on $U$ is pulled back from a divisor $F_j$ on $U/G$.
(Here $F_j$ is a Weil divisor, but it is a Cartier divisor
outside a codimension-2 subset of $U/G$, since $U/G$ is normal.)
By (2), $F_j$ pulls back to a divisor $\{h_j=0\}$ on $Q$ with $h_j$
a $\mu_p$-eigenfunction. It follows that the Cartier divisors
$\{y_j=0\}$ on $U$ and $\{h_j=0\}$ on $Q$
have the same pullback to $W$; that is, $h_j=y_j(\text{unit})$
on $W$. Since $h_j$ is a $\mu_p$-eigenfunction of some weight
$b_j\in \{0,\ldots,p-1\}$, we have $h_j=g_j w^{b_j}$ for some function
$g_j$ on $U$, in the notation of the proof of Theorem \ref{mup}.
Therefore, $g_j=y_j(\text{unit})$ on $U$. Since $y_1,\ldots,y_r$
are linearly independent in $\m_U/\m_U^2$, the same is true
for $g_1,\ldots,g_r$. By the proof of Theorem \ref{mup}, the
$\mu_p$-weight of $h_j$ is equal to the eigenvalue of $\varphi$
on the eigenvector $g_j\in \m_U/\m_U^2$. Since $g_j\in \m_U/\m_U^2$
is a nonzero multiple of $y_j$, this is the same as the eigenvalue
of $\varphi$ on $y_j$.

The ring $O_{W,P}$ is faithfully flat
over $O_{U,P}$ and over $O_{Q,P}$.
By Lemma \ref{sameideal}, the maximal ideals $\m_U$ and $\m_Q$
generate the same ideal in $O_{W,P}$. 
Since $g_1,\ldots,g_r$ are
linearly independent in $\m_U/\m_U^2$, it follows
that $h_1,\ldots,h_r$ are linearly independent
in $\m_Q/\m_Q^2$, as we want.
\end{proof}

\section{The example over the 2-adic integers}
\label{Z2}

\begin{theorem}
\label{Z2thm}
Let $Y = \{(x,y,i)\in A^3_{\Z_2}: x\neq 0, y\neq 0, i^2=-1\}$.
Let the group $G=\Z/2=\{1,\sigma\}$ act on $Y$
by
$$\sigma(x,y,i)=(1/x,1/y,-i).$$
Then the scheme $Y/G$ is terminal, not Cohen-Macaulay,
of dimension 3, and flat over $\Z_2$. Also, the canonical
class of $Y/G$ over $\Z_2$ is Cartier.
\end{theorem}

\begin{proof}
The scheme $Y$ is regular, being an open subset of the affine
plane over the discrete valuation ring $\Z_2[\zeta_4]=\Z_2[i]/(i^2+1)$.
Since $Y$ is a normal integral affine scheme of dimension 3, so is $Y/G$.
The ring $O(Y/G)$ of regular functions
is a torsion-free $\Z_2$-module, since it is a subring
of the torsion-free $\Z_2$-module $O(Y)$; so $Y/G$ is flat over $\Z_2$.
The fixed point scheme of $G$
on $Y$ is defined by: $I(x)=(1-x^2)/x$,
$I(y)=(1-y^2)/y$, and $I(i)=\sigma(i)-i=-2i$,
hence by $\{x^2=1, y^2=1, 2i=0\}$.
Together with the equation $i^2=-1$ on $Y$, these equations imply
set-theoretically that $2=0$, $x=1$, $y=1$, and $i=1$; so the fixed
point set of $G$ is a single
closed point $P$ in $Y$, with residue field $\F_2$.
Since $Y$ is regular, it follows that $Y/G$ is regular
outside the image of $P$, which we also call $P$.

For $Y/G$ to be Cohen-Macaulay at $P$ would mean that the local
cohomology $H^i_P(Y/G,O)$ was zero for $i<\dim(Y/G)=3$. 
Consider the exact sequence $H^1(Y/G,O)\to H^1(Y/G-P,O)
\to H^2_P(Y/G,O)$. Since $Y/G$ is affine, we have $H^1(Y/G,O)=0$.
So Cohen-Macaulayness of $Y/G$ would imply that $H^1(Y/G-P,O)=0$.

Fogarty showed that for $G=\Z/p$ acting with an isolated fixed point
on a normal scheme $W$ over $\F_p$ of dimension at least 3,
$W/G$ is not Cohen-Macaulay \cite[Proposition 4]{Fogartydepth}.
When $W$ has mixed
characteristic $(0,p)$, he needed $\dim(W)\geq 4$ to get the same conclusion.
Nonetheless, we can build on his ideas to study the 3-dimensional scheme $Y$
in mixed characteristic.

We first show that $H^1(G,O(Y))$ is not zero.
This cohomology group is $\ker(\tr)/\Im(1-\sigma)$
on $O(Y)$, where the trace is $1+\sigma$. Since $\tr(i)=0$,
$i$ defines an element of $H^1(G,O(Y))$. Note that $i$ restricts
to $1\in O(P)=\F_2$ on the fixed point $P$. Therefore, $i$ has nonzero
image under the restriction map $H^1(G,O(Y))\to H^1(G,O(P))\cong \F_2$.
So $i$ is nonzero in $H^1(G,O(Y))$, as we want.

Since $G$ acts freely on $Y$ outside $P$, we have a spectral sequence
(as discussed in \cite{Fogartydepth}):
$$E_2^{pq}=H^p(G,H^q(Y-P,O))\imp H^{p+q}(Y/G-P,O).$$
Here $H^0(Y-P,O)=O(Y-P)$ is equal to $O(Y)$, since $Y$ is normal and $P$
has codimension 3 in $Y$ (at least 2 would suffice). So $H^1(G,H^0(Y-P,O))
=H^1(G,O(Y))\neq 0$. The spectral sequence shows that this group
injects into $H^1(Y/G-P,O)$, and so $H^1(Y/G-P,O)\neq 0$. As discussed above,
it follows that $Y/G$ is not Cohen-Macaulay.

It remains to show that $Y/G$ is terminal. Let us recall the definition.
For a normal quasi-projective scheme $X$ over a regular base scheme $S$,
Hartshorne defined the {\it canonical sheaf }$\omega_{X/S}$
\cite[Definition 1.6]{Kollarsings}. It is
a reflexive sheaf of rank 1,
or equivalently the sheaf associated to a Weil divisor.
In this paper, $S$ will be Spec of the $p$-adic integers
or of a field, and we write $K_X$ for $\omega_{X/S}$. Toward the end
of the proof of Theorem \ref{Z2thm}, we compute
$K_X$ directly from the definition in our example.

A normal scheme
$X$ is {\it terminal }if $K_X$ is $\Q$-Cartier and, for every normal
scheme $Z$ with a proper birational morphism $\pi\colon Z\to X$, we have
$$K_Z=\pi^*(K_X)+\sum_j a_jE_j$$
with $a_j>0$ for every exceptional divisor $E_j$ of $\pi$. If $X$ has
a resolution of singularities, terminality of $X$
is equivalent to positivity of the discrepancies $a_j$
on this one resolution \cite[Corollary 2.12]{Kollarsings}.

Let $Y_0=Y$. To prove that our example $Y_0/G$ is terminal,
one approach would be to construct
an explicit resolution of singularities. As with the analogous
example in characteristic 2 \cite[Theorem 5.1]{Totarokodaira},
this can be done by making $G$-equivariant blow-ups
of $Y_0$ along regular closed subschemes.
Namely, we can make $G$-equivariant blow-ups $Y_2\to Y_1\to Y_0$
such that $Y_2/G$ is regular. That resolution of $Y_0/G$
has dual complex a star,
with one edge from a vertex $F_0$ to each of seven other vertices
$F_1,\ldots,F_7$.
\begin{center}
\begin{tikzpicture}
\draw (0,0) -- (0,1);
\draw (0,0) -- (0.78183,0.62349);
\draw (0,0) -- (0.97493,-0.22252);
\draw (0,0) -- (0.43389,-0.90097);
\draw (0,0) -- (-0.43389,-0.90097);
\draw (0,0) -- (-0.97493,-0.22252);
\draw (0,0) -- (-0.78183,0.62349);
\draw[fill] (0,0) circle [radius=0.05];
\draw[fill] (0,1) circle [radius=0.05];
\draw[fill] (0.78183,0.62349) circle [radius=0.05];
\draw[fill] (0.97493,-0.22252) circle [radius=0.05];
\draw[fill] (0.43389,-0.90097) circle [radius=0.05];
\draw[fill] (-0.43389,-0.90097) circle [radius=0.05];
\draw[fill] (-0.97493,-0.22252) circle [radius=0.05];
\draw[fill] (-0.78183,0.62349) circle [radius=0.05];
\end{tikzpicture}
\end{center}

However, we can simplify the proof that $Y_0/G$ is terminal
by stopping with a partial resolution with toric singularities.
Namely, after only one blow-up
$Y_1\to Y_0$, we can recognize the seven singularities of $Y_1/G$
as $\mu_2$-quotient
singularities of the form $\frac{1}{2}(1,1,1)$, thanks
to Theorem \ref{mup}. So $Y_1/G$ is terminal, and it is
then straighforward to compute that $Y_0/G$ is terminal.
This method would also
simplify the proof of terminality for the example in characteristic 2
that we are imitating \cite{Totarokodaira}. The simplification
is more striking
for our more complicated examples in characteristic 3
or mixed characteristic $(0,3)$ (Theorems
\ref{F3thm} and \ref{Z3thm}), and even more significant
for our even more complicated examples in characteristic 5
or mixed characteristic $(0,5)$ (Theorems
\ref{F5thm} and \ref{Z5thm}).

We now begin to blow up. To simplify the equations,
change coordinates by $x_0:=x-1$, $x_1:=y-1$,
and $e_2:=1+i$,
so that the $G$-fixed point is defined by $0=x_0=x_1=e_2$.
Then
$$Y_0=\{(x_0,x_1,e_2)\in A^3_{\Z_2}: 0=e_2^2-2e_2+2,\; 1+x_0\neq 0,
\; 1+x_1\neq 0\},$$
and $G$ acts by
$$\sigma(x_0,x_1,e_2)=\bigg(\frac{-x_0}{1+x_0},\;
\frac{-x_1}{1+x_1},\; 2-e_2\bigg).$$
The blow-up at the $G$-fixed point is:
\begin{multline*}
Y_1=\{ ((x_0,x_1,e_2),[y_0,y_1,y_2])\in A^3_{\Z_2}\times_{\Z_2} \P^2_{\Z_2}:
e_2^2-2e_2+2=0,\\
x_0y_1=x_1y_0,\; x_0y_2=e_2y_0,\; x_1y_2=e_2y_1,\; 1+x_0\neq 0,
\; 1+x_1\neq 0\}.
\end{multline*}
The exceptional divisor $E_0\subset Y_1$ is isomorphic to $\P^2_{\F_2}$.
Here $G$ acts on $Y_1$ by
$$\sigma((x_0,x_1,e_2),[y_0,y_1,y_2])=\bigg(\bigg(\frac{-x_0}{1+x_0},
\frac{-x_1}{1+x_1},2-e_2\bigg),\bigg[\frac{-y_0}{1+x_0},\frac{-y_1}{1+x_1},
y_2(1-e_2)\bigg]\bigg).$$

First consider the open subset $U_0=\{y_0=1\}$ in $Y_1$.
Then $x_1=x_0y_1$ and $e_2=x_0y_2$, so
$$U_0=\{(x_0,y_1,y_2)\in A^3_{\Z_2}: 0=(x_0y_2)^2-2(x_0y_2)+2,
1+x_0\neq 0, 1+x_0y_1\neq 0\}.$$
Here $E_0=\{x_0=0\}$.
The group $G=\Z/2$ acts by
$$\sigma(x_0,y_1,y_2)=\bigg( \frac{-x_0}{1+x_0},\;
\frac{y_1(1+x_0)}{1+x_0y_1},\;
-y_2(1-x_0y_2)(1+x_0). \bigg).$$
The fixed point scheme $Y_1^G$
is defined by: $I(x_0)=\sigma(x_0)-x_0=x_0^2(-1-x_0y_2^2+x_0^2y_2^3)/(1+x_0)$,
$I(y_1)=x_0y_1(1-y_1)$, and $I(y_2)=x_0y_2(-1-y_2+x_0y_2+x_0y_2^2)$.
We know that $Y_1^G$ (as a set)
is contained in $E_0$. To focus on the fixed point scheme
near $E_0$, we can say (more simply):
$I(x_0)=x_0^2(1+O(x_0))$, $I(y_1)=x_0y_1(1+y_1+O(x_0))$,
and $I(y_2)=x_0y_2(1+y_2+O(x_0))$.
We see that the fixed point
scheme $Y_1^G$ is generically the Cartier divisor $E_0=\{x_0=0\}$.
The bad locus in $E_0$ (where that fails) is given by removing
a factor of $x_0$ from the equations and setting $x_0=0$,
so we have: $0=x_0$, $0=y_1(1+y_1)$, and $0=y_2(1+y_2)$. (Note that
$x_0=0$ implies $2=0$, by the equation for $U_0$.)
So the fixed point scheme $Y_1^G$ (in this chart)
is $E_0$ as a Cartier divisor
except at the 4 points $(x_0,y_1,y_2)$ equal to $(0,0,0)$,
$(0,1,0)$, $(0,0,1)$, or $(0,1,1)$.

The open set $U_1=\{y_1=1\}$ works the same way, by the symmetry
switching $x_0$ and $x_1$, hence also switching
$y_0$ and $y_1$. Together, that gives 6 bad points in $E_0\cong
\P^2_{\F_2}$ so far, namely $[y_0,y_1,y_2]$ equal to $[1,0,0]$, $[1,1,0]$,
$[1,0,1]$, $[1,1,1]$, $[0,1,0]$, and $[0,1,1]$.

Finally, look at the open set $U_2=\{y_2=1\}$ in $Y_1$.
Then $x_0=e_2y_0$,
$x_1=e_2y_1$, and so
$$U_2=\{(y_0,y_1,e_2)\in A^3_{\Z_2}: e_2^2-2e_2+2=0, 1+y_0e_2\neq 0,
1+y_1e_2\neq 0\}.$$
Here $E_0=\{e_2=0\}$. On $U_2$, $G$ acts by
$$\sigma(y_0,y_1,e_2)=\bigg(\frac{y_0(1-e_2)}{1+y_0e_2},\;
\frac{y_1(1-e_2)}{1+y_1e_2},\;
2-e_2\bigg).$$
We know that $Y_1^G$ is contained as a set
in $E_0$. 
The fixed point scheme $Y_1^G$ (near $E_0$) is defined by:
$I(y_0)=y_0e_2(1+y_0+O(e_2))$,
$I(y_1)=y_1e_2(1+y_1+O(e_2))$, and $I(e_2)=e_2^2(1+O(e_2))$.
We see that the fixed point scheme $Y_1^G$
is generically $E_0=\{e_2=0\}$ as a Cartier divisor. The bad locus on $E_0$
(where this
fails) is given by removing a factor of $e_2$ from the equations
and setting $e_2=0$, so we have: $0=e_2$, $0=y_1(1+y_1)$,
and $0=y_2(1+y_2)$. So there are 4 bad points on $E_0=\P^2_{\F_2}$
in this open set: $(y_0,y_1,e_2)$ equal to $(0,0,0)$,
$(1,0,0)$, $(0,1,0)$, or $(1,1,0)$.

We conclude that the fixed point scheme in $Y_1$
is $E_0\cong \P^2_{\F_2}$ with multiplicity 1 except at the 7 points:
$$[y_0,y_1,y_2]=[1,0,0],[0,1,0],[0,0,1],[1,1,0],[1,0,1],[0,1,1],[1,1,1].$$
(The same thing happens for the first blow-up
of the analogous example in characteristic 2
\cite{Totarokodaira}.) By Theorem \ref{kl}, $Y_1/G$ is regular outside
the images of these 7 points.

\def\edge{1cm}
\def\rad{0.05}
\begin{figure}
\label{Z2figure}
\begin{center}
\begin{tikzpicture}
\draw (0,0) ++(180:2*\edge)
  ++(-90:\edge) node{$Y_0$};
\draw[fill,red] ++(180:2*\edge)
  ++(90:0.577350269*\edge) circle [radius=\rad];
\draw[->] ++(180:0.75*\edge)
  ++(90:0.577350269*\edge) -- ++(180:0.5*\edge);
\fill[gray!20] (0,0) -- ++(0:2*\edge) -- ++(120:2*\edge)
  -- cycle;
\draw (0,0) ++(0:\edge) ++(-90:\edge)
  node{$Y_1$};
\draw (0,0) -- ++(0:\edge) -- ++(0:\edge)
 -- ++(120:\edge) -- ++(120:\edge)
 -- ++(240:\edge) -- ++(240:\edge);
\draw (0,0) ++(0:0.2*\edge) node[anchor=south west]{$E_0$};
\draw[fill] (0,0) circle [radius=\rad];
\draw[fill] ++(0:\edge) circle [radius=\rad];
\draw[fill] ++(0:\edge) ++(0:\edge) circle [radius=\rad];
\draw[fill] ++(0:\edge) ++(0:\edge) ++(120:\edge) circle [radius=\rad];
\draw[fill] ++(0:\edge) ++(0:\edge) ++(120:\edge) 
 ++(120:\edge) circle [radius=\rad];
\draw[fill] ++(0:\edge) ++(0:\edge) ++(120:\edge) 
 ++(120:\edge) ++(240:\edge) circle [radius=\rad];
\draw[fill] ++(0:\edge)
  ++(90:0.577350269*\edge) circle [radius=\rad];
\end{tikzpicture}
\end{center}
\caption{The fixed point schemes in $Y_0$ and in the blow-up $Y_1$.
Here $E_0\cong \P^2$, which we view as a toric variety; so the three
edges of the triangle denote the coordinate lines in $\P^2$.
We consider three coordinate charts on $Y_1$, each containing one
vertex of the triangle.}
\end{figure}
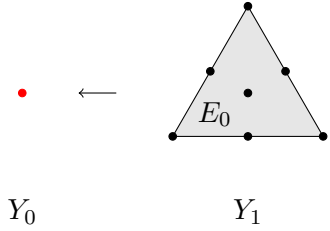

One further $G$-equivariant
blow-up at each of these 7 points suffices to resolve $Y_1/G$,
but the equations for these blow-ups are a bit messy.
Instead, we will use Theorem \ref{mup}
to show that the 7 singular points
of $Y_1/G$ are all mixed-characteristic analogs of the singularity
$A^3/\mu_2$, the simplest terminal singularity whose canonical class
is not Cartier.
More precisely, each of these 7 singular points is of the form
$\frac{1}{2}(1,1,1)$, meaning that it can be written as $Q/\mu_2$
for some regular scheme $Q$ of dimension 3 (at an isolated fixed point
of $\mu_2$). As a result, $Y_1/G$ is terminal.
With one last calculation, we will deduce that $Y_0/G$ is terminal.

We first consider the singularities of $Y_1/G$ in the chart
$U_0=\{y_0=1\}$, as above. Here $U_0$ has coordinates
$(x_0,y_1,y_2)$ with $0=(x_0y_2)^2+2(x_0y_2)-2$, and $E_0=\{x_0=0\}$.
We saw that the fixed point scheme $U_0^G$ is the Cartier
divisor $E_0=\{x_0=0\}$ except at the 4 points $(x_0,y_1,y_2)$
equal to $(0,0,0),(0,0,1),(0,1,0)$, or $(0,1,1)$.
At each of these points $P$, the $G$-action has the form
required for Theorem \ref{mup} with $e=s:=x_0$,
namely that $I(x_0)=x_0^2(1+\m)$
and $I(z_j)=x_0(z_j+\m^2)$ for $j=1,2$, for some coordinates
$x_0,z_1,z_2$ at $P$, with $\m$ the maximal ideal at $P$.
Also, since $x_0^2y_2^2=2(\text{unit})$, 2 is in the ideal $(x_0^2)$,
hence in $x_0\m$, which is another assumption in Theorem \ref{mup}.
So the theorem
gives that these 4 singular points of $Y_1/G$ are of the form
$\frac{1}{2}(1,1,1)$. 

The calculations are identical in the chart $\{y_1=1\}$ in $Y_1$.
They are slightly different in the chart $\{y_2=1\}$,
but the conclusion is the same:
the singularities of $Y_1/G$
in this chart are again of the form $\frac{1}{2}(1,1,1)$.
Namely, this chart has coordinates
$(y_0,y_1,e_2)$ with $0=e_2^2-2e_2+2$, and $E_0=\{e_2=0\}$.
The fixed point scheme $Y_1^G$ is the Cartier
divisor $E_0$ except at the 4 points $(y_0,y_1,e_2)$
equal to $(0,0,0),(1,0,0),(0,1,0)$, or $(1,1,0)$.
At each of these points $P$, the $G$-action has the form
required for Theorem \ref{mup} with $e=s:=e_2$,
namely that $I(e_2)=e_2^2(1+\m)$
and $I(z_j)=e_2(z_j+\m^2)$ for $j=0,1$,
for some coordinates $(z_0,z_1,e_2)$
at $P$, where $\m$ is the maximal ideal at $P$.
Also, the equation for $e_2$ implies that
2 is in the ideal $(e_2^2)$, hence in $e_2\m$, which is another
assumption in Theorem \ref{mup}. So the theorem
gives that these 4 singular points of $Y_1/G$ are again of the form
$\frac{1}{2}(1,1,1)$.

Thus all 7 singular points of $Y_1/G$ are of the form $\frac{1}{2}(1,1,1)$.
By the Reid-Tai criterion (Theorem \ref{reidtai}), they are terminal.
(To check that by hand:
each singular point has a resolution $Z\to U_1/G$ whose exceptional divisor
is $E_j\cong \P^2_{\F_2}$ with
normal bundle $O(-2)$. As a result, the singularities
of $Y_1/G$ are terminal, with
$K_Z=\pi^*(K_{Y_1/G})+\frac{1}{2}E_j$ near each $E_j$.)

Recall that $Y_0=Y$ is the regular scheme of dimension 3
that we started with. (Thus $Y_1$
is the blow-up of $Y_0$ at the $G$-fixed point.)
We now go on to show that $X=Y_0/G$ is terminal.
Write $F_0$ for the image in $Y_1/G$ of
the exceptional divisor $E_0$. Note that although $G$
fixes $E_0$ in $Y_1$, the morphism $E_0\arrow F_0$
is a finite purely inseparable morphism, not necessarily an isomorphism.
(Indeed, $G=\Z/2$ is not linearly reductive over $\Z_2$.
So if $G$ acts on an affine scheme $T$ preserving a closed subscheme $S$,
the morphism $S/G\arrow T/G$ need not be a closed immersion.
Equivalently, the $G$-equivariant surjection $O(T)\arrow O(S)$ need not yield
a surjection $O(T)^G\arrow O(S)^G$.)

Write $K_{Y_0}$ for the canonical sheaf $\omega_{Y_0/\Z_2}$.
Since $Y_0$ is regular, $K_{Y_0}$ is a line bundle, described as follows
\cite[Definition 1.6]{Kollarsings}. First, let $R=\Z_2[i]/(i^2+1)$.
Then we have an embedding $D=\Spec R\subset A^1_{\Z_2}$; write
$I$ for the ideal $(i^2+1)\subset \Z_2[i]$ defining this subscheme.
Then the adjunction formula $K_D=(K_{A^1}+D)|_D$ is made into a definition:
$$\omega_{R/\Z_2}=\Omega^1_{\Z_2[i]/\Z_2}
\otimes_{\Z_2[i]}R\otimes_R (I/I^2)^*.$$
In these terms, one trivializing section of $\omega_{R/\Z_2}$
is $\alpha:=\frac{di}{i}\cdot f$, where $f\colon I/I^2\to R=\Z_2[i]/I$
is the map sending $i^2+1$ to 1. (Formally, one could think of this
section of $\omega_{R/\Z_2}$
as $\frac{1}{i}\frac{di}{d(i^2+1)}$.) Next, since $\pi\colon Y_0\to
\Spec R$ is smooth
of relative dimension 2, we have $K_{Y_0}=\omega_{Y_0/\Z_2}=
\Omega^2_{Y_0/R}\otimes \pi^*\omega_{R/\Z_2}$. So one trivializing
section of $K_{Y_0}$ is $\beta:=\frac{dx}{x}\wedge \frac{dy}{y}\wedge
\frac{di}{i}\cdot f$. I claim that
this section is fixed by $G$. Indeed, if we extend the action
of $G$ on $R$ to $A^1_{\Z_2}$ by $\sigma(i)=-i$, then $\sigma(f)=f$
and $\sigma(di)=-di$, so $\sigma(\frac{di}{i}f)=\frac{di}{i}f$. Also,
$\sigma(dx/x)=-dx/x$ and $\sigma(dy/y)=-dy/y$, from which we see that
$\sigma(\beta)=\beta$ as claimed. It follows that
the divisor
class $K_{Y_0/G}$ is linearly equivalent to zero,
in particular Cartier.
Here $K_{Y_0/G}$ is the canonical sheaf
in the sense of \cite[Definition 1.6]{Kollarsings}; $Y_0/G$ is not
Gorenstein, since (as we have shown) it is not Cohen-Macaulay.

Since $K_X$ is Cartier, we can write
$$K_{Y_1/G}=\pi^*K_X+a_0F_0$$
for an integer $a_0$.
Since $Y_1/G$ is terminal, $X$ is terminal if and only
if the discrepancy $a_0$ is positive.
Here and below, we write $\pi$ for all the relevant contractions,
which in the formula above means $\pi\colon Y_1/G\arrow Y_0/G=X$.

The analogous formula for $Y_1$ is easy.
Since $Y_1$
is the blow-up of the regular 3-dimensional scheme $Y_0$ at a closed point,
$$K_{Y_1}=\pi^*K_{Y_0}+2E_0.$$

Write $f$ for the quotient map $Y_0\arrow Y_0/G$ or $Y_1\arrow Y_1/G$.
The ramification of $f$ along $E_0$ can be computed as follows.

\begin{corollary}
\label{fierce}
Let $U$ and $G$ be as in Theorem \ref{mup}.
So $U$ is a regular scheme with an action of the group
$G=\Z/p=\langle \sigma:\sigma^p=1\rangle$, for a prime number $p$,
and assume that $U^G$ is generically a regular divisor
$E_1=\{x_1=0\}$ and that $I(x_1)=x_1^2(\text{unit})$.
Assume that $U$ is of finite type over a regular base scheme $S$
and that $G$ acts on $U$ over $S$. Write $K_U$ and $K_{U/G}$ for
the canonical classes over $S$.
Then $U\to U/G$ is fiercely ramified along $E_1$.
In particular, the image $F_1$ of $E_1$ in $U/G$ is $\Q$-Cartier,
and $f\colon U\to U/G$ satisfies $f^*F_1=E_1$
and $K_U=f^*K_{U/G}+(p-1)E_1$.
\end{corollary}

\begin{proof}
The norm $N(x_1)$ is a function on $U/G$ that defines a positive
multiple of the divisor $F_1$,
and so $F_1$ is $\Q$-Cartier. Since the fixed point scheme
$U^G$ is generically the Cartier divisor $E_1$, the ramification
divisor of $f$ is $(p-1)E_1$ by Lemma \ref{different}. That is,
$K_U=f^*K_{U/G}+(p-1)E_1$. Also, the fixed point scheme
is generically $E_1=\{ x_1=0\}$ with coefficient 1, whereas
$I(x_1)$ vanishes to order 2 along $E_1$; so section
\ref{ramsection} gives that the ramification of $f$ along $E_1$
is fierce. In particular, the ramification index $e$ is 1,
meaning that $f^*F_1=E_1$.
\end{proof}

In particular, returning to our example with $p=2$,
we have seen that the divisor $E_0$ has multiplicity 1
in the fixed point scheme $(Y_1)^G$. Also,
Corollary \ref{fierce} gives that $f\colon Y_1\to Y_1/G$
is fiercely ramified along $E_0$. So we have
$$K_{Y_1}=f^*K_{Y_1/G}+E_0$$
and $f^*F_0=E_0$.
(The same is true for the example in characteristic 2 that
we are imitating \cite{Totarokodaira}.)

Since $f\colon Y_0\arrow Y_0/G$ is \'etale in codimension 1,
we have $K_{Y_0}=f^*K_{Y_0/G}$. It follows that
\begin{align*}
f^*K_{Y_1/G}&=K_{Y_1}-E_0\\
&= (\pi^*K_{Y_0}+2E_0)-E_0\\
&= \pi^*f^*K_{Y_0/G}+E_0\\
&= f^*(\pi^*K_{Y_0/G}+F_0).
\end{align*}
Therefore, 
$$K_{Y_1/G}=\pi^*K_{Y_0/G}+F_0.$$
Because the coefficient of the exceptional divisor
$F_0$ is positive, and $Y_1/G$ is terminal as shown above,
$X=Y_0/G$ is terminal.
\end{proof}

\section{Characteristic 3}
\label{char3}

\begin{theorem}
\label{F3thm}
Let the group $G=\Z/3$ with generator $\tau$
act on $\P^2$ over $\F_3$ by
$$\tau([u_0,u_1,u_2])=[u_1,u_2,u_0]$$
and on $\P^1$ by
$$\tau([y_0,y_1])=[y_0,y_0+y_1].$$
Then $(\P^2\times \P^1)/G$
is terminal, not Cohen-Macaulay,
and of dimension 3 over $\F_3$.
\end{theorem}

\begin{proof}
We work throughout over $k=\F_3$. Write $G=\Z/3=\langle \sigma:
\sigma^3=1\rangle$, with $\tau:=\sigma^{-1}$. Let $Y_0=\P^2\times \P^1$
and $X=Y_0/G$. The only fixed point of $G$ on $Y_0$
is $P=([1,1,1],[0,1])$. So $X$ is normal of dimension 3,
and $X$ is smooth over $k$ outside the image of $P$, which we also call $P$.
Also, $3K_X$ is Cartier. By Fogarty, since $P$ is an isolated
fixed point of $G=\Z/p$ on a smooth 3-fold in characteristic $p$,
$X$ is not Cohen-Macaulay at $P$ \cite[Proposition 4]{Fogartydepth}.

It remains to show that $X$ is terminal. One can resolve the singularities
of $X$ by performing $G$-equivariant blow-ups of $Y_0$. However,
as in section \ref{Z2}, we will shorten the proof by recognizing that,
after two $G$-equivariant blow-ups $Y_2\to Y_1\to Y_0$, the singularities
of $Y_2/G$ become toric, namely quotients of a regular scheme
by $\mu_3$. That makes it easy
to check that $Y_0/G$ is terminal, without having to continue making
$G$-equivariant blow-ups.

Before this approach, I found a $G$-equivariant blow-up $Y_{18}\to
\cdots \to Y_0$ with $Y_{18}/G$ regular, but the construction
involved 18 blow-ups along points or curves. The approach here,
looking for toric singularities instead of regularity,
saves a lot of work.

To put the fixed point at the origin, we change coordinates on $Y_0$
by: $x_0=(u_0+u_1+u_2)/u_1$, $x_1=(-u_1+u_2)/u_1$,
and $x_2=y_0/y_1$. Then
$G$ acts on the open subset $U$ of $Y_0$ given by
$$U=\{ (x_0,x_1,x_2)\in A^3: 1+x_2\neq 0,\; 1-x_2\neq 0,\;
1+x_1\neq 0,\; 1+x_0-x_1\neq 0\}.$$
The $G$-action on $U$ is given by
$$\tau(x_0,x_1,x_2)=\bigg( \frac{x_0}{1+x_1},\;
\frac{x_1+x_0}{1+x_1},\;
\frac{x_2}{1+x_2}\bigg).$$
As we blow up, we will not need to keep track of the precise
affine open set on which $G$ acts, since we are only concerned with
the action near the fixed point set.

Let $Y_1$ be the blow-up of $Y_0$ at the $G$-fixed point,
which is the origin in these coordinates. Then the open subset of $Y_1$
over $U\subset Y_0$ is
$$\{ ((x_0,x_1,x_2),[y_0,y_1,y_2])\in U\times \P^2:
x_0y_1=x_1y_0,\; x_0y_2=x_2y_0,x_1y_2=x_2y_1\}.$$
Clearly the fixed point set $Y_1^G$ is contained in the exceptional
divisor $E_0\cong \P^2$. It turns out
to be a curve isomorphic to $\P^1$. We need three coordinate charts
to cover $E_0$. First consider
the open subset
$\{y_0=1\}$ in $Y_1$. Here $(x_0,x_1,x_2)=(x_0,x_0y_1,x_0y_2)$,
and $G$ acts by
$$\tau(x_0,y_1,y_2)=\bigg( \frac{x_0}{1+x_0y_1},\;
1+y_1,\;
\frac{y_2(1+x_0y_1)}{1+x_0y_2}\bigg) .$$
By the action on the $y_1$ coordinate,
there are no fixed points in this open set.

Next, work in the open set
$\{y_1=1\}\subset Y_1$. Here $(x_0,x_1,x_2)=(y_0x_1,x_1,y_2x_1)$,
$E_0=\{x_1=0\}$, and $G$ acts by
$$\tau(y_0,x_1,y_2)=\bigg(\frac{y_0}{1+y_0},\;
\frac{x_1(1+y_0)}{1+x_1},\;
\frac{y_2(1+x_1)}{(1+y_0)(1+x_1y_2)}\bigg).$$
The fixed point scheme $Y_1^G$ is defined by:
$I(y_0)=-y_0^2/(1+y_0)$, $I(x_1)=x_1(y_0+O(x_1))$,
and $I(y_2)=y_2(-y_0+O(x_1))/(1+y_0+O(x_1))$. Since
we know that the fixed point set $Y_1^G$ is contained in $E_0
=\{x_1=0\}$, we read off that
the fixed point set is the line $\{0=y_0=x_1\}$ in $E_0$.

Finally, consider the open set $\{y_2=1\}$. Then $(x_0,x_1,x_2)=(x_2y_0,
x_2y_1,x_2)$, and $G$ acts by
$$\tau(y_0,y_1,x_2)=\bigg( \frac{y_0(1+x_2)}{1+y_1x_2},\;
\frac{(y_0+y_1)(1+x_2)}{1+y_1x_2},\;
\frac{x_2}{1+x_2}\bigg).$$
In these coordinates, the exceptional divisor $E_0\cong \P^2$
in $Y_1$ is $\{x_2=0\}$.
The fixed point scheme of $G$ is defined by:
$I(y_0)=y_0x_2(1-y_1+O(x_2))$, $I(y_1)=y_0+O(x_2)$,
and $I(x_2)=x_2^2(-1+O(x_2))$. We know that the fixed point set
is contained in $E_0$, and these equations imply that the fixed
point set is the line $\{0=y_0=x_2\}$ in $E_0$, the same line
as in the previous chart. Thus we have shown that the fixed point set
in all of $Y_1$ is a curve isomorphic to $\P^1$.

Let $Y_2$ be the blow-up of $Y_1$ along the (reduced) $G$-fixed curve,
with exceptional divisor $E_1\subset Y_2$.
It is clear that $Y_2^G$ is contained
in $E_1$ as a set. Write $E_0$ for the strict transform of $E_0\subset Y_1$
in $Y_2$. It turns out that
the fixed point scheme $Y_2^G$ is equal to the Cartier divisor $E_1$
except at six points in $E_1$,
three over the point $[y_0,y_1,y_2]=[0,1,0]$ in $E_0\subset Y_1$
and one each over $[y_0,y_1,y_2]$ equal to $[0,1,1]$, $[0,1,-1]$,
or $[0,0,1]$. Since $E_1$ is a $\P^1$-bundle
over $\P^1$, we will need to look
at four affine charts to see all of it.
(See Figure \ref{F3figure}. Each affine chart we consider
in $Y_2$ contains exactly one of the four vertices of the square.)

First work
over the chart $\{y_1=1\}$ in $Y_1$.
We wrote out the $G$-action
on this chart above, with coordinates $(y_0,x_1,y_2)$. Here
$E_0=\{x_0=0\}$.
We defined $Y_2$ as the blow-up of $Y_1$ along the $G$-fixed
curve $\{0=y_0=x_1\}$ in $E_0$.
So, over this open subset of $Y_1$, $Y_2$ is an open
subset of
$$\{((y_0,x_1,y_2),[w_0,w_1])\in A^3\times \P^1: y_0w_1=x_1w_0\}.$$
\def\edge{1cm}
\def\rad{0.05}
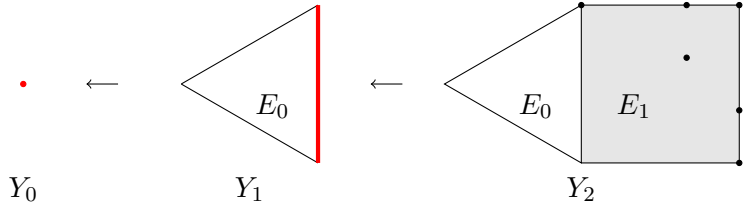
\begin{figure}
\begin{center}
\begin{tikzpicture}[scale=0.7]
\draw (0,0) 
  ++(-90:2*\edge) node{$Y_0$};
\draw[fill,red] (0,0) circle [radius=\rad];
\draw[->] ++(0:1.8*\edge)
  -- ++(180:0.6*\edge);
\begin{scope}[xshift=3*\edge]
\draw (0,0) ++(0:1.2990381057\edge) ++(-90:2*\edge)
  node{$Y_1$};
\draw (0,0) ++(0:1.7320508*\edge) node[anchor=north]{$E_0$};
\draw (0,0) -- ++(-30:3*\edge);
\draw (0,0) -- ++(30:3*\edge);
\draw[ultra thick,red] (0,0) ++(-30:3*\edge)
  -- ++(90:3*\edge);
\draw[->] ++(0:4.2*\edge)
  -- ++(180:0.6*\edge);
\end{scope}
\begin{scope}[xshift=8*\edge]
\draw (0,0) ++(0:1.5*1.7320508*\edge) ++(-90:2*\edge)
  node{$Y_2$};
\fill[gray!20] (0,0) ++(-30:3*\edge) -- ++(0:3*\edge)
  -- ++(90:3*\edge) -- ++(180:3*\edge)
  -- cycle;
\draw (0,0) -- ++(-30:3*\edge) -- ++(90:3*\edge)
  -- cycle;
\draw (0,0)  ++(-30:3*\edge) -- ++(0:3*\edge)
  -- ++(90:3*\edge) -- ++(180:3*\edge);
\draw (0,0) ++(0:1.7320508*\edge) node[anchor=north]{$E_0$};
\draw (0,0) ++(0:1.5*1.7320508*\edge) ++(0:\edge)
  node[anchor=north]{$E_1$};
\draw[fill] (0,0) ++(-30:3*\edge) ++(0:3*\edge)
  circle [radius=\rad];
\draw[fill] (0,0) ++(30:3*\edge)
  circle [radius=\rad];
\draw[fill] (0,0) ++(-30:3*\edge) ++(0:3*\edge) ++(90:\edge)
  circle [radius=\rad];
\draw[fill] (0,0) ++(-30:3*\edge) ++(0:3*\edge) ++(90:3*\edge)
  circle [radius=\rad];
\draw[fill] (0,0) ++(-30:3*\edge) ++(0:2*\edge) ++(90:2*\edge)
  circle [radius=\rad];
\draw[fill] (0,0) ++(-30:3*\edge) ++(0:2*\edge) ++(90:3*\edge)
  circle [radius=\rad];
\end{scope}
\end{tikzpicture}
\end{center}
\caption{For $j=0,1,2$, $G$ acts freely on $Y_j$ outside the shaded
or marked loci; $Y_j/G$ is regular outside the marked loci;
and $Y_j/G$ has toric singularities outside the red loci.}
\label{F3figure}
\end{figure}

First look at the chart $\{w_0=1\}$ in $Y_2$ over
$\{y_1=1\}$ in $Y_1$. (This chart contains the upper left vertex
of $E_1$, in Figure \ref{F3figure}.) Then
$(y_0,x_1,y_2)=(y_0,y_0w_1,y_2)$,
and $G=\Z/3$ acts by
$$\tau(y_0,w_1,y_2)=\bigg( \frac{y_0}{1+y_0},\;
\frac{w_1(1+y_0)^2}{1+y_0w_1},\;
\frac{y_2(1+y_0w_1)}{(1+y_0w_1y_2)(1+y_0)}\bigg) .$$
Here $E_0=\{w_1=0\}$ and $E_1=\{y_0=0\}$.
The fixed point scheme is defined by:
$I(y_0)=y_0^2(-1+O(y_0))$, $I(w_1)=y_0w_1(-1-w_1+O(y_0))$,
and $I(y_2)=y_0y_2(-1+w_1-w_1y_2+O(y_0))$. So
the fixed point scheme (near $E_1$) is $E_1$ with multiplicity 1
except at points with $0=y_0$, $0=w_1(w_1+1)$,
and $0=y_2(1-w_1+w_1y_2)$.
Thus the bad points are
$(y_0,w_1,y_2)=(0,0,0)$ in $E_0\cap E_1$, $(0,-1,0)$, and $(0,-1,-1)$.
We have found three bad points in $E_1$, with the first one
in $E_0\cap E_1$.

The other chart over $\{y_1=1\}$ in $Y_1$
is $\{w_1=1\}$ in $Y_2$ (the chart containing
the upper right vertex of $E_1$ in Figure \ref{F3figure}). Then
$(y_0,x_1,y_2)=(w_0x_1,x_1,y_2)$,
and $G$ acts by
$$\tau(w_0,x_1,y_2)=\bigg( \frac{w_0(1+x_1)}{(1+w_0x_1)^2},\;
\frac{x_1(1+w_0x_1)}{1+x_1},\;
\frac{y_2(1+x_1)}{(1+x_1y_2)(1+w_0x_1)}\bigg) .$$
Here $E_0$ does not appear, and $E_1=\{x_1=0\}$.
The fixed point scheme is defined (near $E_1$) by:
$I(w_0)=w_0x_1(1+w_0+O(x_1))$, $I(x_1)=x_1^2(-1+w_0+O(x_1))$,
and $I(y_2)=x_1y_2(1-w_0-y_2+O(x_1))$.
So $Y_2^G$ is $E_1$ with multiplicity 1 except at points
where $0=x_1$, $0=w_0(1+w_0)$,
and $0=y_2(1-w_0-y_2)$.
Thus the bad points in this chart are $(w_0,x_1,y_2)$
equal to $(0,0,0)$, $(0,0,1)$, $(-1,0,0)$, and $(-1,0,-1)$.
The points with $w_0\neq 0$ appeared in the previous chart,
$\{w_0=1\}$. So we have two new bad points, $(w_0,x_1,y_2)$
equal to $(0,0,0)$ or $(0,0,1)$, for a total of five so far.

To see all of $E_1$ in $Y_2$, we also need to work over
$\{y_2=1\}\subset Y_1$.
The coordinates
are $(y_0,y_1,x_2)$, and $E_0=\{x_2=0\}$.
The corresponding open subset of $Y_2$
is an open subset of
$$\{ ((y_0,y_1,x_2),[z_0,z_2])\in A^3\times \P^1: y_0z_2=x_2z_0\}.$$
First consider $\{z_0=1\}$ in $Y_2$,
the chart containing the lower left
vertex of $E_1$ in Figure \ref{F3figure}.
Then $(y_0,y_1,x_2)=(y_0,y_1,y_0z_2)$,
and $G$ acts by
$$\tau(y_0,y_1,z_2)=\bigg( \frac{y_0(1+y_0z_2)}{1+y_0y_1z_2},\;
\frac{(y_0+y_1)(1+y_0z_2)}{1+y_0y_1z_2},\;
\frac{z_2(1+y_0y_1z_2)}{(1+y_0z_2)^2}\bigg).$$
Here $E_0=\{z_2=0\}$ and $E_1=\{y_0=0\}$.
The fixed point scheme $Y_2^G$ is defined (near $E_1$) by:
$I(y_0)=y_0^2z_2(1-y_1+O(y_0))$, $I(y_1)=y_0(1+y_1z_2-y_1^2z_2+O(y_0))$,
and $I(z_2)=y_0z_2^2(1+y_1+O(y_0))$. This is equal to $E_1$
with multiplicity 1 except at the point
$(y_0,y_1,z_2)=(0,-1,-1)$. This appeared in the chart $\{w_0=1\}$
over $\{y_1=1\}$ (as the point $(y_0,w_1,y_2)=(0,-1,-1)$).

The other chart is $\{z_2=1\}$ in $Y_2$, which contains
the lower right vertex of $E_1$ in Figure \ref{F3figure}.
Then $(y_0,y_1,x_2)=(z_0x_2,y_1,x_2)$,
and $G$ acts by
$$\tau(z_0,y_1,x_2)=\bigg( \frac{z_0(1+x_2)^2}{1+y_1x_2},\;
\frac{(y_1+z_0x_2)(1+x_2)}{1+y_1x_2},\;
\frac{x_2}{1+x_2}\bigg).$$
In this open set, $E_0$ does not appear, and $E_1=\{x_2=0\}$.
The fixed point scheme $Y_2^G$
is defined by:
$I(z_0)=z_0x_2(-1-y_1+O(x_2))$, $I(y_1)=x_2(z_0+y_1-y_1^2+O(x_2))$,
and $I(x_2)=x_2^2(-1+O(x_2))$.
We know that $Y_2^G$ is contained in $E_1=\{x_2=0\}$ as a set.
We read off that the fixed point scheme $Y_2^G$ is generically
the Cartier divisor $E_1$. The bad locus (where that fails)
is given, in $E_1$, by: $0=x_2$, $0=z_0(1+y_1)$, and $0=z_0+y_1-y_1^2$.
So we see three bad points in $E_1$: 
$(z_0,y_1,x_2)=(0,1,0)$, $(-1,-1,0)$, or $(0,0,0)$. The first two
points appeared in the chart $\{w_1=1\}$ in $Y_2$ over
$\{y_1=1\}$ in $Y_1$ (as $(w_0,x_1,y_2)=(0,0,1)$ or $(-1,0,-1)$,
respectively). On the other hand, the origin in this chart is new;
so we have a total of six bad points.

Thus the fixed point scheme $Y_2^G$
is $E_1$ with multiplicity 1
except at six points on $E_1$.
We will use Theorem \ref{mup} to recognize
the resulting six singularities of $Y_2/G$ as toric;
so we have no further need
for $G$-equivariant blow-ups.

First consider three of the six bad points
of $E_1\subset Y_2$, the ones with $[y_0,y_1,y_2]$
equal to $[0,0,1]$, $[0,1,1]$, or $[0,1,-1]$. In these cases,
our calculation of the action of $G$ shows (by Theorem \ref{mup})
that the singularity of $Y_2/G$ at these points is
of the form $\frac{1}{3}(1,1,2)$. (For example,
for the point $[y_0,y_1,y_2]=[0,0,1]$, work in the chart
$\{z_2=1\}$ in $Y_2$ over $\{y_2=1\}$ in $Y_1$. Here we have coordinates
$(z_0,y_1,x_2)$, $(Y_2)^G_{\red}$ is
$\{ x_2=0\}$, and the point we are considering is the origin.
As above, we have
$I(z_0)=z_0x_2(-1-y_1+O(x_2))$, $I(y_1)=x_2(z_0+y_1-y_1^2+O(x_2))$,
and $I(x_2)=x_2^2(-1+O(x_2))$. So Theorem \ref{mup}
applies, with $e=s:=x_2$. The linear map $\varphi$ in the theorem
is given by: $\varphi(z_0)=-z_0$, $\varphi(y_1)=z_0+y_1$,
and $\varphi(x_2)=-x_2$. This map has eigenvalues $-1,1,-1$ in $\F_3$,
and so Theorem \ref{mup}
gives that the singularity of $Y_2/G$ at this point
is of the form $\frac{1}{3}(-1,1,-1)\cong \frac{1}{3}(1,1,2)$.)
By the Reid-Tai
criterion (Theorem \ref{reidtai}), this singularity is terminal.

Next, consider the chart $\{w_1=1\}$ in $Y_2$ over $\{y_1=1\}$
in $Y_1$.
As shown above, in coordinates $(w_0,x_1,y_2)$,
the points $(0,0,0)$ and $(-1,0,0)$
are bad for the action of $G$ on $Y_2$, and $(Y_2)_{\red}=\{x_1=0\}$.
Here again,
the singularity of $Y_2/G$ at these points is
of the form $\frac{1}{3}(1,1,2)$. (For example,
for the point $(w_0,x_1,y_2)=(0,0,0)$, we saw that
$I(w_0)=w_0x_1(1+w_0+O(x_1))$, $I(x_1)=x_1^2(-1+w_0+O(x_1))$,
and $I(y_2)=x_1y_2(1-w_0-y_2+O(x_1))$. So Theorem \ref{mup}
applies with $e=s:=x_1$. The linear map $\varphi$ in the theorem
is given by $\varphi(w_0)=w_0$, $\varphi(x_1)=-x_1$,
and $\varphi(y_2)=y_2$. This has eigenvalues $1,-1,1$,
and so the theorem gives that
the singularity
of $Y_2/G$ at this point is of the form $\frac{1}{3}(1,-1,1)
\cong \frac{1}{3}(1,1,2)$.)
In particular, these two points
are terminal. Thus $Y_2/G$ is terminal
at five of its six singular points.

Finally, we consider the last singular point of $Y_2/G$.
In the chart $\{w_0=1\}$ in $Y_2$ over $\{y_1=1\}$ in $Y_1$,
the point $P$ is $(y_0,w_1,y_2)=(0,0,0)$, and $(Y_2)^G_{\red}=\{y_0=0\}$.
In this case,
our calculation of the action of $G$ shows (by Theorem
\ref{mup})
that the singularity of $Y_2/G$ is
of the form $\frac{1}{3}(1,1,1)$. (As above, we have
$I(y_0)=y_0^2(-1+O(y_0))$, $I(w_1)=y_0w_1(-1-w_1+O(y_0))$,
and $I(y_2)=y_0y_2(-1+w_1-w_1y_2+O(y_0))$. So
Theorem \ref{mup} applies with $e=s:=y_0$. The linear map $\varphi$
in the theorem is given by $\varphi(y_0)=-y_0$,
$\varphi(w_1)=-w_1$, and $\varphi(y_2)=-y_2$.
So Theorem \ref{mup} gives that the singularity
of $Y_2/G$ at this point is of the form $\frac{1}{3}(-1,-1,-1)
\cong \frac{1}{3}(1,1,1)$.)
By the Reid-Tai criterion (Theorem \ref{reidtai}),
$Y_2/G$ is canonical at this point, but not
terminal.

We can now begin the proof that $Y_0/G$ is terminal.
Write $f$ for any of the quotient maps $Y_j\to Y_j/G$.
The fixed point scheme $Y_2^G$ is the Cartier divisor
$E_1$ except at six points on $E_1$. Clearly $3K_{Y_2/G}$ is Cartier.
Write $E_0\subset Y_2$ for the strict transform of $E_0\subset Y_1$.
For $j=0,1$ in $Y_2$, let $F_j$ be the
image (as an irreducible divisor) of $E_j$ in $Y_2/G$. Since $G$ acts
nontrivially on $E_0$, we have $f^*F_0=E_0$. The divisor
$E_1$ is fixed by $G$.
By Corollary \ref{fierce}, $f$ is fiercely ramified along $E_1$. So
the ramification divisor
is $p-1$ times $E_1$, meaning that
$K_{Y_2}=f^*(K_{Y_2/G})+2E_1$, and we have $f^*F_1=E_1$.

Write $\pi_{ij}$ for the birational morphism $Y_i\to Y_j$
or $Y_i/G\to Y_j/G$ (with $i>j$).
Since $\pi_{20}\colon Y_2\to Y_0$ is defined by blowing up points
and smooth curves on a smooth 3-fold, we have:
$$K_{Y_1}=\pi_{10}^*(K_{Y_0})+2E_0$$
and
\begin{align*}
K_{Y_2}&=\pi_{21}^*(K_{Y_1})+E_1\\
&=\pi_{20}^*(K_{Y_0})+2(E_0+E_1)+E_1\\
&=\pi_{20}^*(K_{Y_0})+2E_0+3E_1.
\end{align*}

Therefore,
\begin{align*}
f^*\pi_{20}^*K_{Y_0/G}={}&\pi_{20}^*f^*K_{Y_0/G}\\
={}&\pi_{20}^*K_{Y_0}\\
={}&K_{Y_2}-2E_0-3E_1\\
={}&f^*(K_{Y_2/G})+2E_1-2E_0-3E_1\\
={}&f^*(K_{Y_2/G})-2E_0-E_1\\
={}&f^*(K_{Y_2/G}-2F_0-F_1).
\end{align*}
So
$$K_{Y_2/G}=\pi_{20}^*(K_{Y_0/G})+2F_0+F_1.$$

Here every exceptional divisor of the birational morphism $Y_2/G\to Y_0/G$
has positive coefficient. Also, we have shown
that $Y_2/G$ is terminal outside one point which is canonical,
and that point lies on $F_1$. Therefore, $Y_0/G$ is terminal.
We showed earlier that it is not Cohen-Macaulay.
Theorem \ref{F3thm} is proved.
\end{proof}

\begin{remark}
\label{F3cartier}
The divisor class $\pi^*K_{Y_0/G}=K_{Y_2/G}-2F_0-F_1$
has some non-integer discrepancies, for example over the origin
in the chart $\{z_2=1\}$ in $Y_2$ over $\{y_2=1\}$
in $Y_1$. As a result, $K_{Y_0/G}$ is not Cartier
(as one can also check directly), in contrast
to our examples in residue characteristic 2
(\cite[Theorem 5.1]{Totarokodaira} and Theorem \ref{Z2thm}).
I expect that there is also a 3-fold $X$
over $\F_3$ that is terminal and non-Cohen-Macaulay
with $K_X$ Cartier. Namely, one should replace $\P^1$
in Theorem \ref{F3thm} by the ``Harbater-Katz-Gabber'' curve
$C=\{ 0=y^p-yz^{p-1}-x^{p-1}z\}$ in $\P^2_{\F_p}$ \cite{Bleher},
here with $p=3$,
which has a $\Z/p$-action by $\tau([x,y,z])=[x+z,y,z]$ that preserves
a nonzero 1-form near the fixed point $[1,0,0]$. For $p=3$,
$C$ is a supersingular elliptic curve.
\end{remark}

\section{The example over the 3-adic integers}
\label{Z3}

\begin{theorem}
\label{Z3thm}
Let $G$ be the group $G=\Z/3$ with generator $\tau$.
Let $R=\Z_3[e]/(e^3-3e^2+3)$, which is the ring of integers
in a Galois extension of $\Q_3$
with group $G=\Z/3$. Let $G$
act on the scheme $\P^2_R$ by
$$\tau([u_0,u_1,u_2],e)=([u_1,u_2,u_0],3+e-e^2).$$
Then the scheme $\P^2_R/G$
is terminal, not Cohen-Macaulay, of dimension 3,
and flat over $\Z_3$.
\end{theorem}

This example behaves much like the example over $\F_3$,
Theorem \ref{F3thm}. 
Theorem \ref{F5thm}. {\it In particular, Figure \ref{F3figure}
accurately depicts the blow-ups we make in mixed characteristic $(0,3)$,
just as in characteristic 3. }We can view $R$
as the subring $\Z_3[\zeta_9]^+$ of the cyclotomic ring $\Z_3[\zeta_9]$
fixed by $\zeta_9\mapsto \zeta_9^{-1}$, with $e=\zeta_9+1+\zeta_9^{-1}$.
Informally, $R$ is the simplest ramified $\Z/3$-extension of $\Z_3$.
More broadly,
this action of $G$ on $\P^2_R$ was chosen as possibly the simplest
action of $\Z/3$ on a 3-fold in mixed characteristic $(0,3)$
with an isolated fixed point. The simplicity helps to ensure
that the quotient scheme is terminal.

\begin{proof}
We write $G=\Z/3=\langle \sigma:
\sigma^3=1\rangle$, with $\tau:=\sigma^{-1}$.
Let $Y_0=\P^2_R$ with $G$ acting diagonally on $\P^2$ and on $R$,
and let $X=Y_0/G$. Write $e_2$ for the generator $e$ of $R$,
to fit better with our numbering
of coordinates on $Y_0$; so we have
$$0=e_2^3-3e_2^2+3.$$
The only fixed point of $G$ on $Y_0$
is the closed point $P\cong \Spec \F_3$
given by $([u_0,u_1,u_2],e_2)=([1,1,1],0)$.
So $X$ is normal of dimension 3,
and $X$ is regular outside the image of $P$, which we also
call $P$. Clearly $3K_X$ is Cartier.

It is not automatic from Fogarty's results \cite{Fogartydepth},
but we can use his methods to show that
$X$ is not Cohen-Macaulay at $P$. As in the proof of Theorem
\ref{Z2}, using that $G$ has an isolated fixed point on the 3-fold $Y_0$,
it suffices to show that $H^1(G,O(Y_0))$ is not zero.
This cohomology group is $\ker(\tr)/\Im(1-\sigma)$
on $O(Y_0)$, where the trace is $1+\sigma+\sigma^2$. The equation
$0=e_2^3-3e_2^2+3$ (specifically, the coefficient of $e_2^2$)
implies that $e_2$ has trace 3. So $\tr(1-e_2)=0$, and hence
$1-e_2$ defines an element of $H^1(G,O(Y_0))$. Note that $1-e_2$ restricts
to $1\in O(P)=\F_3$ on the fixed point $P$. Therefore, $1-e_2$ has nonzero
image under the restriction map $H^1(G,O(Y_0))\to H^1(G,O(P))\cong \F_3$.
So $H^1(G,O(Y_0))$ is not zero, and hence
$Y_0/G$ is not Cohen-Macaulay.

It remains to show that $X$ is terminal. One can resolve the singularities
of $X$ by performing $G$-equivariant blow-ups of $Y_0$. However,
as in sections \ref{Z2} and \ref{char3},
we will shorten the proof by recognizing that,
after two $G$-equivariant blow-ups $Y_2\to Y_1\to Y_0$, the singularities
of $Y_2/G$ become toric (in Kato's mixed-characteristic sense),
namely quotients of a regular scheme
by $\mu_3$. That makes it easy
to check that $Y_0/G$ is terminal, without having to continue making
$G$-equivariant blow-ups.

To put the fixed point at the origin, we change coordinates on $Y_0$
by: $x_0=(u_0+u_1+u_2-3)/u_1$ and $x_1=(-u_1+u_2)/u_1$. Then $G$ acts
on 
$$U:=\{(x_0,x_1,e_2)\in A^3_{\Z_3}: 0=3-3e_2^2+e_2^3, 1+x_1\neq 0,
1+x_0-x_1\neq 0\}$$
by
$$\tau(x_0,x_1,e_2)=\bigg( \frac{x_0-3x_1}{1+x_1},\;
\frac{x_0-2x_1}{1+x_1},\;
3+e_2-e_2^2\bigg).$$
In what follows, we will often not need to keep track
of the precise open set on which $G$ acts, because we are
only concerned with the $G$-fixed point scheme.

The blow-up $Y_1\to Y_0$ at the $G$-fixed point is,
over the open set $U\subset Y_0$:
$$\{ ((x_0,x_1,e_2),[y_0,y_1,y_2])\in U\times_{\Z_3}
\P^2_{\Z_3}: x_0y_1=x_1y_0,\; x_0y_2=e_2y_0,\; x_1y_2=e_2y_1\}.$$

Clearly the fixed point set $Y_1^G$ is contained in the exceptional
divisor $E_0\cong \P^2_{\F_3}$. It turns out
to be a curve isomorphic to $\P^1_{\F_3}$. We need three coordinate charts
to cover $E_0$.
First look at the open set $U_0=\{y_0=1\}$ in $Y_1$. Then
$(x_0,x_1,e_2)=(x_0,x_0y_1,x_0y_2)$, and so
$$U_0=\{(x_0,y_1,y_2)\in A^3_{\Z_3}:
0=x_0^3y_2^3-3x_0^2y_2^2+3,\; 1+x_0y_1\neq 0,
\; 1+x_0-x_0y_1\neq 0\}.$$
The exceptional divisor $E_0$ is $\{x_0=0\}$.
Here $G$ acts by
$$\tau(x_0,y_1,y_2)=\bigg( \frac{x_0(1-3y_1)}{1+x_0y_1},\;
\frac{1-2y_1}{1-3y_1},\;
\frac{y_2(1+2x_0y_2-x_0^2y_2^2)(1+x_0y_1)}{1-3y_1}\bigg).$$
We know that the fixed point scheme $Y_1^G$ is contained
in $E_0=\{x_0=0\}$ as a set, and that $3=0$ on $E_0$.  The fixed point
scheme is defined by: $I(x_0)=x_0^2y_1(-1+O(x_0))$,
$I(y_1)=1+O(x_0)$, and $I(y_2)=x_0y_2(y_1-y_2+O(x_0))$.
By the second equation, $Y_1^G$ is empty
in this open set.

Next, consider the open set $U_1=\{y_1=1\}$ in $Y_1$. Then
$(x_0,x_1,e_2)=(x_1y_0,x_1,x_1y_2)$, and $G$ acts on
$$U_1=\{(y_0,x_1,y_2)\in A^3_{\Z_3}:0=3-3x_1^2y_2^2+x_1^3y_2^3,
\; 1+x_1\neq 0,\; 1-x_1+x_1y_2\neq 0\}.$$
Namely, $G$ acts by
$$\tau(y_0,x_1,y_2)=\bigg( \frac{y_0-3}{y_0-2},\;
\frac{x_1(y_0-2)}{1+x_1},\;
\frac{y_2(1+x_1)(1+2x_1y_2-x_1^2y_2^2)}{y_0-2}\bigg).$$
Here $E_0=\{x_1=0\}$. We know that $Y_1^G$ is contained in $E_0=\{x_1=0\}$,
as a set (and hence $3=0$ on $(Y_1)^G_{\red}$). More precisely,
the fixed point scheme $Y_1^G$ is defined by:
$I(y_0)=(-y_0^2+O(x_1))/(1+y_0+O(x_1))$,
$I(x_1)=x_1(y_0+O(x_1))$, and $I(y_2)=y_2(-y_0+O(x_1))/(1+y_0+O(x_1))$.
So $Y_1^G$ is the line $\{y_0=x_1=0\}$, as a set.

Finally, consider the chart $\{y_2=1\}$ in $Y_1$.
We have coordinates $(y_0,y_1,e_2)$, with
$(x_0,x_1,e_2)=(e_2y_0,e_2y_1,e_2)$, and $E_0=\{e_2=0\}$.
Here $G$ acts on the open set
$$U_2=\{(y_0,y_1,e_2)\in A^3_{\Z_3}:0=3-3e_2^2+e_2^3,\;
1+e_2y_0\neq 0,\; 1+e_2y_0-e_2y_1\neq 0\}.$$
Namely, $G$ acts by
$$\tau(y_0,y_1,e_2)=\bigg( \frac{(y_0-3y_1)(1+e_2-e_2^2)}{1+e_2y_1},\;
\frac{(y_0-2y_1)(1+e_2-e_2^2)}{1+e_2y_1},\;
3+e_2-e_2^2\bigg).$$
The fixed point scheme $Y_1^G$ is defined by:
$I(y_0)=e_2(y_0-y_0y_1+O(e_2))$, $I(y_1)=y_0+O(e_2)$,
and $I(e_2)=e_2^2(-1+O(e_2))$. We know that the fixed point
set is contained in $E_0=\{e_2=0\}$. We read off that the fixed point set
is the line $0=y_0=e_2$, which is the same line seen in the previous
chart. Thus we have shown that the fixed point set in all of $Y_1$
is a curve isomorphic to $\P^1_{\F_3}$ in $E_0$.

Seeking to make the fixed point set a divisor,
we let $Y_2$ be the blow-up of $Y_1$ along the (reduced)
$G$-fixed curve, with exceptional divisor $E_1\subset Y_2$.
It is clear that $Y_2^G$ is contained in $E_1$ as a set.
We will see that the fixed point scheme $Y_2^G$ is equal
to the Cartier divisor $E_1$
except at six points on $E_1$. These correspond
exactly to the six bad points that occur
in the example over $\F_3$ (Figure \ref{F3figure}).
Since $E_1$ is a $\P^1$-bundle
over $\P^1_{\F_3}$, we will need to look
at four affine charts to see all of it.

First work over the open subset $\{y_1=1\}$ in $Y_1$, with
coordinates $(y_0,x_1,y_2)$, where $E_0=\{x_1=0\}$.
Since $Y_2$ is the blow-up
of $Y_1$ along the $G$-fixed curve $\{0=y_0=x_1\}$, this part
of $Y_2$ is given by
$$\{ ((y_0,x_1,y_2),[w_0,w_1])\in A^3_{\Z_3}\times \P^1_{\Z_3}:
0=3-3x_1^2y_2^2+x_1^3y_2^3,\; y_0w_1=x_1w_0\}.$$

First consider the open set $\{w_0=1\}\subset Y_2$
over $\{y_1=1\}\subset Y_1$. Then $(y_0,x_1,y_2)=(y_0,y_0w_1,
y_2)$, $E_0=\{w_1=0\}$, and $E_1=\{y_0=0\}$. Also, $e_2=y_0w_1y_2$,
and so $3=3y_0^2w_1^2y_2^2-y_0^3w_1^3y_2^3$.
Here $G$ acts by
\begin{multline*}
\tau(y_0,w_1,y_2)=\bigg(\frac{y_0-3}{y_0-2},\;
\frac{w_1(y_0-2)^2}{(1+y_0w_1)(1-3y_0w_1^2y_2^2+y_0^2w_1^3y_2^3)},\\
\frac{y_2(1+y_0w_1)(1+2y_0w_1y_2-y_0^2w_1^2y_2^2)}{y_0-2}\bigg).
\end{multline*}
The fixed point scheme $Y_2^G$ (near $E_1$) is defined by:
$I(y_0)=y_0^2(-1+O(y_0))$, $I(w_1)=y_0w_1(-1-w_1+O(y_0))$,
and $I(y_2)=y_0y_2(-1+w_1-w_1y_2+O(y_0))$.
So $Y_2^G$
is the Cartier divisor $E_1$ except where $0=y_2(1-w_1+w_1y_2)$
and $0=w_1(1+w_1)$. So we have found three bad points,
$(y_0,w_1,y_2)=(0,0,0)$ in $E_0\cap E_1$ and
$(0,-1,0)$, and $(0,-1,-1)$ in $E_1$.

The other chart over $\{y_1=1\}$ in $Y_1$
is $\{w_1=1\}$ in $Y_2$.
Then $(y_0,x_1,y_2)=(w_0x_1,x_1,y_2)$,
$E_0$ does not appear, and
$E_1=\{x_1=0\}$. Also, $e_2=x_1y_2$, and so
$3=3x_1^2y_2^2-x_1^3y_2^3$. Here $G$ acts by
\begin{multline*}
\tau(w_0,x_1,y_2)=\bigg(
\frac{(w_0-3x_1y_2^2+x_1^2y_2^3)(1+x_1)}{(w_0x_1-2)^2},\;
\frac{x_1(w_0x_1-2)}{1+x_1},\\
\frac{y_2(1+x_1)(1+2x_1y_2-x_1^2y_2^2)}{w_0x_1-2}\bigg).
\end{multline*}
So the fixed point scheme $Y_2^G$ (near $E_1$) is defined by:
$I(w_0)=x_1(w_0+w_0^2+O(x_1))$, $I(x_1)=x_1^2(-1+w_0+O(x_1))$,
and $I(y_2)=x_1y_2(1-w_0-y_2+O(x_1))$.
So $Y_2^G$ is the Cartier divisor
$E_1$ except where $x_1=0$ (so $3=0$), $0=w_0(1+w_0)$,
and $0=y_2(1-w_0-y_2)$. So the bad points are $(w_0,x_1,y_2)=(0,0,0)$,
$(0,0,1)$, $(-1,0,0)$, and $(-1,0,-1)$. 
The points with $w_0\neq 0$ appeared in the previous
chart, $\{w_0=1\}$. So we have two new bad points, $(w_0,x_1,y_2)$
equal to $(0,0,0)$ or $(0,0,1)$, for a total of five so far.

To see all of $E_1$ in $Y_2$, we also have to
work over the open set $\{y_2=1\}\subset Y_1$,
with coordinates $(y_0,y_1,x_2)$, where $E_0=\{x_2=0\}$.
The corresponding open subset of $Y_2$ is an open subset
of
$$\{ ((y_0,y_1,e_2),[z_0,z_2])\in A^3_{\Z_3}\times_{\Z_3}
\P^1_{\Z_3}: 0=3-3e_2^2+e_2^3,\; y_0z_2=e_2z_0\}.$$
First consider the chart $\{z_0=1\}\subset Y_2$.
Then $(y_0,y_1,e_2)=(y_0,y_1,y_0z_2)$, $E_0=\{z_2=0\}$,
and $E_1=\{y_1=0\}$. Also,
$e_2=y_0z_2$, and so $0=3-3y_0^2z_2^2+y_0^3z_2^3$.
Here $G$ acts by
\begin{multline*}
\tau(y_0,y_1,z_2)=\bigg(
\frac{(y_0-3y_1)(1+y_0z_2-y_0^2z_2^2)}{1+y_0y_1z_2},\\
\frac{(y_0-2y_1)(1+y_0z_2-y_0^2z_2^2)}{1+y_0y_1z_2},\;
\frac{z_2(4+y_0z_2-y_0^2z_2^2)(1+y_0y_1z_2)}{1-3y_0y_1z_2^2+y_0^2y_1z_2^3}
\bigg).
\end{multline*}
The fixed point scheme (near $E_1$)
is defined by: $I(y_0)=y_0^2z_2(1-y_1+O(y_0))$,
$I(y_1)=y_0(1+y_1z_2-y_1^2z_2+O(y_0))$,
and $I(z_2)=y_0z_2^2(1+y_1+O(y_0))$.
So $Y_2^G$ is the Cartier divisor
$E_1$ except where $y_0=0$ (so $3=0$), $0=1+y_1z_2-y_1^2z_2$,
and $0=z_2^2(1+y_1)$. So we have one bad point in this open set,
$(y_0,y_1,z_2)=(0,-1,-1)$. This already appeared in the chart $\{w_0=1\}$
over $\{y_1=1\}$ (as the point $(y_0,w_1,y_2)=(0,-1,-1)$).

The other chart is $\{z_2=1\}\subset Y_2$ over $\{y_2=1\}$ in $Y_1$,
with coordinates $(z_0,y_1,e_2)$. Here $y_0=e_2z_0$,
$E_0$ does not appear,
and $E_1=\{e_2=0\}$. Here $G$ acts by
\begin{multline*}
\tau(z_0,y_1,e_2)=\bigg( \frac{(z_0-3e_2y_1+e_2^2y_1)
(-2-e_2+2e_2^2)}{1+e_2y_1},\\
\frac{(e_2z_0-2y_1)(1+e_2-e_2^2)}{1+e_2y_1},\;
3+e_2-e_2^2\bigg).
\end{multline*}
We know that the fixed point scheme $Y_2^G$ 
is contained in $E_1$ as a set. (Also, $3=O(e_2)$ by the equation
for $Y_2$.) Explicitly, the fixed point scheme (near $E_2$)
is defined by:
$I(z_0)=e_2(-z_0-z_0y_1+O(e_2))$, $I(y_1)=e_2(z_0+y_1-y_1^2+O(e_2))$,
and $I(e_2)=e_2^2(-1+O(e_2))$.
So $Y_2^G$ is the Cartier divisor $E_1$ except
where $e_2=0$ (so $3=0$), $0=z_0(1+y_1)$, and $0=z_0+y_1-y_1^2$.
So we see three bad points in $E_1$: 
$(z_0,y_1,x_2)=(0,1,0)$, $(-1,-1,0)$, or $(0,0,0)$. The first two
points appeared in the chart $\{w_1=1\}$ in $Y_2$ over
$\{y_1=1\}$ in $Y_1$ (as $(w_0,x_1,y_2)=(0,0,1)$ or $(-1,0,-1)$,
respectively). On the other hand, the origin in this chart is new;
so we have a total of six bad points.

Thus the fixed point scheme $Y_2^G$
is $E_1$ with multiplicity 1
except at six points on $E_1$.
As in Theorem \ref{F3thm},
Theorem \ref{mup} shows that five of the singular points
of $Y_2/G$ are toric singularities of the form
$\frac{1}{3}(1,1,2)$ (hence terminal),
while the sixth is a toric singularity
of the form $\frac{1}{3}(1,1,1)$ (hence canonical).
In fact, our calculations of $I=\sigma-1$ on the coordinates
in this section are identical to those in section
\ref{char3}, to the accuracy we state. We also need
to check the assumption in Theorem \ref{mup}
that $p\in e^{p-1}\m$,
that is, that $3\in e^2\m$. This is true because
$3=e_2^3(\text{unit})$ on $Y_2$, and $e_2$ is a multiple
of the function $e$ defining $E_1$ in each coordinate chart;
so 3 is in the ideal $(e^3)$, hence in $e^2\m$ at each of the bad
points. As a result, Theorem \ref{mup} gives the
conclusions above about the 6 singular points of $Y_2/G$.

The calculation of the discrepancies of $Y_0/G$ is likewise
identical to the calculation in section \ref{char3}.
Therefore, $Y_0/G$ is terminal.
We showed earlier that it is not Cohen-Macaulay.
Theorem \ref{Z3thm} is proved.
\end{proof}

\begin{remark}
\label{Z3cartier}
In Theorem \ref{Z3thm}, the canonical class
of $Y_0/G$ is not Cartier.
I expect that there is also a 3-dimensional scheme $X$,
flat over $\Z_3$, that is terminal and non-Cohen-Macaulay
with $K_X$ Cartier. Namely, 
one should replace the $p$-adic integer ring $R=\Z_3[\zeta_9]^{\Z/2}$
in Theorem \ref{Z3thm} by $S=\Z_3[\zeta_9]$, with the action
of $G=\Z/3\subset (\Z/9)^*$. The point is that the canonical sheaf
of $S$ over $\Z_3$ has a $G$-equivariant trivialization.
\end{remark}

\section{Characteristic 5}
\label{char5}

\begin{theorem}
\label{F5thm}
Let the group $G=\Z/5$ with generator $\tau$
act on the quintic del Pezzo surface $S_5$ over $\F_5$ by an embedding
of $G$ into the symmetric group $\Sigma_5=\Aut(S_5)$, and let $G$ act
on $\P^1$ by
$$\tau([y_0,y_1])=[y_0,y_0+y_1].$$
Then $(S_5\times \P^1)/G$
is terminal, not Cohen-Macaulay,
and of dimension 3 over $\F_5$.
\end{theorem}

We define the quintic del Pezzo surface (over any field)
as the moduli space $\overline{M_{0,5}}$ of 5-pointed stable curves
of genus 0. That makes it clear that the symmetric group $\Sigma_5$
acts on this surface.

\begin{proof}
We work throughout over $k=\F_5$. Write $G=\Z/5=\langle \sigma:
\sigma^5=1\rangle$, with $\tau:=\sigma^{-1}$. Let $Y_0=S_5\times \P^1$
and $X=Y_0/G$. In characteristic 5, $G$ has only one fixed point
in $S_5$, and so $G$ has only one fixed point $P$ in $Y_0$.
So $X$ is normal of dimension 3,
and $X$ is smooth over $k$ outside the image of $P$ (which we also
call $P$). Also, $5K_X$ is Cartier. By Fogarty, since $P$ is an isolated
fixed point of $G=\Z/p$ on a smooth 3-fold in characteristic $p$,
$X$ is not Cohen-Macaulay at $P$ \cite[Proposition 4]{Fogartydepth}.

It remains to show that $X$ is terminal. This example is more complicated
than those in characteristics 2 and 3, and it may be impossible
to resolve the singularities
of $X$ by performing $G$-equivariant blow-ups of $Y_0$. (Indeed,
in the simpler situation of actions of $G=\Z/p$ in characteristic zero,
one cannot always resolve the singularities of a quotient
$Y/G$ via $G$-equivariant
blow-ups of $Y$ when $p\geq 5$ \cite[Claim 2.29.2]{Kollarres}.)
Fortunately,
as in earlier sections, we can reach toric singularities
after some $G$-equivariant blow-ups. It will then be easy
to check that $Y_0/G$ is terminal.

The $G$-action on $S_5$ over $k$ is given
on an open subset isomorphic to an open subset of $A^2$ by:
$$\tau(s_0,s_1)=\bigg( \frac{s_0-s_1+s_0^2+s_0s_1}{(1-2s_0)(1-s_0-s_1)},\;
\frac{s_1-2s_0^2-2s_0s_1}{(1-2s_0)(1-s_0-s_1)}\bigg).$$
Here the fixed point is at the origin.
(Section \ref{Z5} explains where this formula comes from.)
So the $G$-action on an open subset $U\subset Y_0$ is given
(on an open neighborhood of the origin in $A^3$) by:
$$\tau(s_0,s_1,s_2)=\bigg( \frac{s_0-s_1+s_0^2+s_0s_1}{(1-2s_0)(1-s_0-s_1)},\;
\frac{s_1-2s_0^2-2s_0s_1}{(1-2s_0)(1-s_0-s_1)},
\frac{s_2}{1+s_2}\bigg).$$
As we blow up, we will not need to keep track of the precise
affine open set on which $G$ acts, since we are only concerned with
the action near the fixed point set.

Let $Y_1$ be the blow-up of $Y_0$ at the $G$-fixed point,
which is the origin in these coordinates. Then the open subset of $Y_1$
over $U\subset Y_0$ is
$$\{ ((s_0,s_1,s_2),[y_0,y_1,y_2])\in U\times \P^2:
s_0y_1=s_1y_0,\; s_0y_2=s_2y_0,s_1y_2=s_2y_1\}.$$
The exceptional divisor $E_0$ is isomorphic to $\P^2$.
It turns out that the fixed point set of $G$ on $Y_1$
is a curve isomorphic to $\P^1$ in $E_0$. To check that,
first work in the open subset
$\{y_0=1\}$ in $Y_1$. Here $(s_0,s_1,s_2)=(s_0,s_0y_1,s_0y_2)$,
and $G$ acts by
\begin{multline*}
\tau(s_0,y_1,y_2)=\bigg(
\frac{s_0(1+s_0-y_1+s_0y_1)}{(1-2s_0)(1-s_0-s_0y_1)},\;
\frac{-2s_0+y_1-2s_0y_1}{1+s_0-y_1+s_0y_1},\;\\
\frac{y_2(1-2s_0)(1-s_0-s_0y_1)}{(1+s_0y_2)(1+s_0-y_1+s_0y_1)}\bigg) .
\end{multline*}
Here $E_0=\{s_0=0\}$.
The fixed point scheme $Y_1^G$ is defined by the vanishing of:
$I(s_0)=s_0(-y_1+O(s_0))$, $I(y_1)=(y_1^2+O(s_0))/(1-y_1+O(s_0))$,
and $I(y_2)=y_2(y_1+O(s_0))/(1-y_1+O(s_0))$. We know that $Y_1^G$
is contained (as a set) in $E_0$ (since $Y_0^G$ is only the origin).
So the fixed point set is the line
$\{0=s_0=y_1\}$, in this chart.

In the chart $\{y_1=1\}$ in $Y_1$, we have $s_0=s_1y_0$
and $s_2=s_1y_2$, so we have coordinates $(y_0,s_1,y_2)$. Here
$E_0=\{s_1=0\}$. We can write the action of $G$ in these coordinates
(for example using Magma). We find that the fixed point scheme $Y_1^G$
is defined by:
$I(y_0)=-1+O(s_1)$, $I(s_1)=s_1^2(1+y_0-2y_0^2)$,
and $I(y_2)=s_1y_2(-1-y_0-y_2+2y_0^2)$. Since $Y_1^G$ is contained
(as a set) in $E_0$, the first equation shows that $Y_1^G$ is empty,
in this chart. In the last chart $\{y_2=1\}$ in $Y_1$,
we have coordinates $(y_0,y_1,s_2)$, and $E_0=\{ s_2=0\}$.
The fixed point scheme is defined by:
$I(y_0)=-y_1+O(s_2)$, $I(y_1)=s_2(y_1+y_1^2+y_0y_1-2y_0^2+O(s_2))$,
and $I(s_2)=s_2^2(-1+O(s_2))$. Since $Y_1^G$ is contained 
(as a set) in $E_0$, the fixed point set is the line
$\{0=y_1=s_2\}$, the same line seen in an earlier chart.

Thus $(Y_1^G)_{\red}$ is isomorphic to $\P^1$.
Our criterion for a quotient by $G$
to have toric singularities (Theorem \ref{mup}) requires
the $G$-fixed locus to have codimension 1; so let $Y_2$ be the blow-up
of $Y_1$ along this $\P^1$. Clearly $G$ continues to act on $Y_2$.
The exceptional divisor $E_1$ in $Y_2$ is a $\P^1$-bundle over
$\P^1$, and so the natural way to cover $E_1$ by affine charts
involves 4 charts, as follows.

\def\edge{1cm}
\def\rad{0.05}
\begin{figure}
\begin{center}
\begin{tikzpicture}[scale=0.7]
\draw (0,0) 
  ++(-90:2*\edge) node{$Y_0$};
\draw[fill,red] (0,0) circle [radius=\rad];
\draw[->] ++(0:1.8*\edge)
  -- ++(180:0.6*\edge);
\begin{scope}[xshift=3*\edge]
\draw (0,0) ++(0:1.2990381057\edge) ++(-90:2*\edge)
  node{$Y_1$};
\draw (0,0) ++(0:1.7320508*\edge) node[anchor=north]{$E_0$};
\draw (0,0) -- ++(-30:3*\edge);
\draw (0,0) -- ++(30:3*\edge);
\draw[ultra thick,red] (0,0) ++(-30:3*\edge)
  -- ++(90:3*\edge);
\draw[->] ++(0:4.2*\edge)
  -- ++(180:0.6*\edge);
\end{scope}
\begin{scope}[xshift=8*\edge]
\draw (0,0) ++(0:1.5*1.7320508*\edge) ++(-90:2*\edge)
  node{$Y_2$};
\draw (0,0) ++(30:3*\edge) -- ++(-150:3*\edge)
  -- ++(-30:3*\edge);
\draw (0,0)  ++(-30:3*\edge) -- ++(0:3*\edge)
  -- ++(90:3*\edge);
\draw[ultra thick,red] (0,0)  ++(-30:3*\edge) ++(0:3*\edge)
  ++(90:3*\edge) -- ++(180:3*\edge) -- ++(-90:3*\edge);
\draw (0,0) ++(0:1.7320508*\edge) node[anchor=north]{$E_0$};
\draw (0,0) ++(0:1.5*1.7320508*\edge) ++(0:\edge)
  node[anchor=north]{$E_1$};
\end{scope}
\end{tikzpicture}
\end{center}
\caption{For $j=0,1,2$, $G$ acts freely on $Y_j$ outside the red loci.}
\label{F5figure012}
\end{figure}
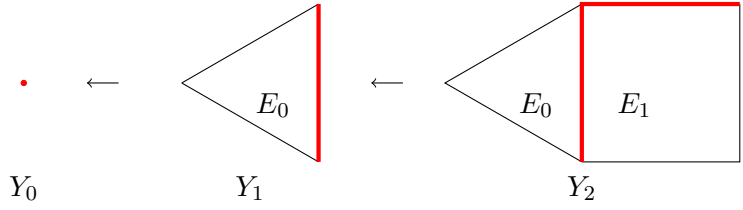

Over the open set $\{y_0=1\}$ in $Y_1$, $Y_2$ is the blow-up
along the $G$-fixed curve $\{0=s_0=y_1\}$, so $Y_2$ has coordinates
$((s_0,y_1,y_2),[w_0,w_1])$. First take $\{w_0=1\}$,
so $y_1=s_0w_1$, and we have coordinates $(s_0,w_1,y_2)$.
Here $E_0$ does not appear, and $E_1=\{s_0=0\}$. The fixed point
scheme $Y_2^G$ is defined by: $I(s_0)=s_0^2(-1-w_1+O(s_0))$,
$I(w_1)=-2+O(s_0)$, and $I(y_2)=s_0y_2(1+w_1-y_2+O(s_0))$.
We know that the fixed point set is contained in $E_1$, and so the formula
for $I(w_1)$ implies that $Y_2^G$ is empty, in this chart.
In the other chart $\{w_1=1\}$ in $Y_2$ over the same open set in $Y_1$,
we have $s_0=y_1w_0$, and so $Y_2$ has coordinates $(w_0,y_1,y_2)$.
Here $E_0=\{w_0=0\}$ and $E_1=\{y_1=0\}$. The fixed point scheme
is defined by $I(w_0)=w_0(2w_0+O(y_1))/(1-2w_0+O(y_1))$,
$I(y_1)=y_1(-2w_0+O(y_1))$, and $I(y_2)=y_1y_2(1+w_0
-w_0y_2+O(y_1))$. So $Y_2^G$ is the line $\{0=w_0=y_1\}=E_0\cap E_1$,
in this chart.

To see the rest of $E_1\subset Y_2$, work over the open set
$\{y_2=1\}$ in $Y_1$. Here $Y_2$ is the blow-up along
the $G$-fixed curve $\{0=y_1=s_2\}$, so $Y_2$ has coordinates
$((y_0,y_1,s_2),[r_1,r_2])$. First take $\{r_1=1\}$ in $Y_2$,
so $s_2=y_1r_2$, and we have coordinates $(y_0,y_1,r_2)$.
Here $E_0=\{r_2=0\}$ and $E_1=\{y_1=0\}$. Here $Y_2^G$ is given by
$I(y_0)=y_1(-1+y_0r_2-y_0^2r_2+O(y_1))$, $I(y_1)
=y_1r_2(-2y_0^2+O(y_1))$,
and $I(r_2)=r_2^2(2y_0^2+O(y_1))/(1-2y_0^2r_2+O(y_1))$.
We know that the fixed point set is contained in $E_1$,
and we read off that it is the union of the two lines
$\{0=y_1=r_2\}=E_0\cap E_1$ and $\{0=y_0=y_1\}$ in $E_1$. The first curve
appeared in an earlier chart, and the second is new.
Finally, the other open set is $\{r_2=1\}$ in $Y_2$,
so $y_1=s_2r_1$, and we have coordinates $(y_0,r_1,s_2)$.
Here $E_0$ does not appear, and $E_1=\{s_2=0\}$. Here $Y_2^G$
is given by $I(y_0)=s_2(y_0-r_1-y_0^2+O(s_2))$,
$I(r_1)=-2y_0^2+O(s_2)$, and $I(s_2)=s_2^2(-1+O(s_2))$.
We read off that the fixed point set is the curve $\{0=y_0=s_2\}$,
which is the second curve in the previous chart.

Thus $(Y_2)^G$ as a set is the union of two $\P^1$'s meeting
at a point. We are trying to make the fixed locus have codimension 1,
and so our next step is to blow up one of those curves.
Namely, let $Y_3$ be the blow-up of $Y_2$ along the $G$-fixed curve
$E_0\cap E_1$. The exceptional divisor $E_2$ in $Y_3$ is a $\P^1$-bundle
over $\P^1$, and so we need to look at four affine charts to see all of it.
\begin{figure}
\includegraphics[width=\textwidth]{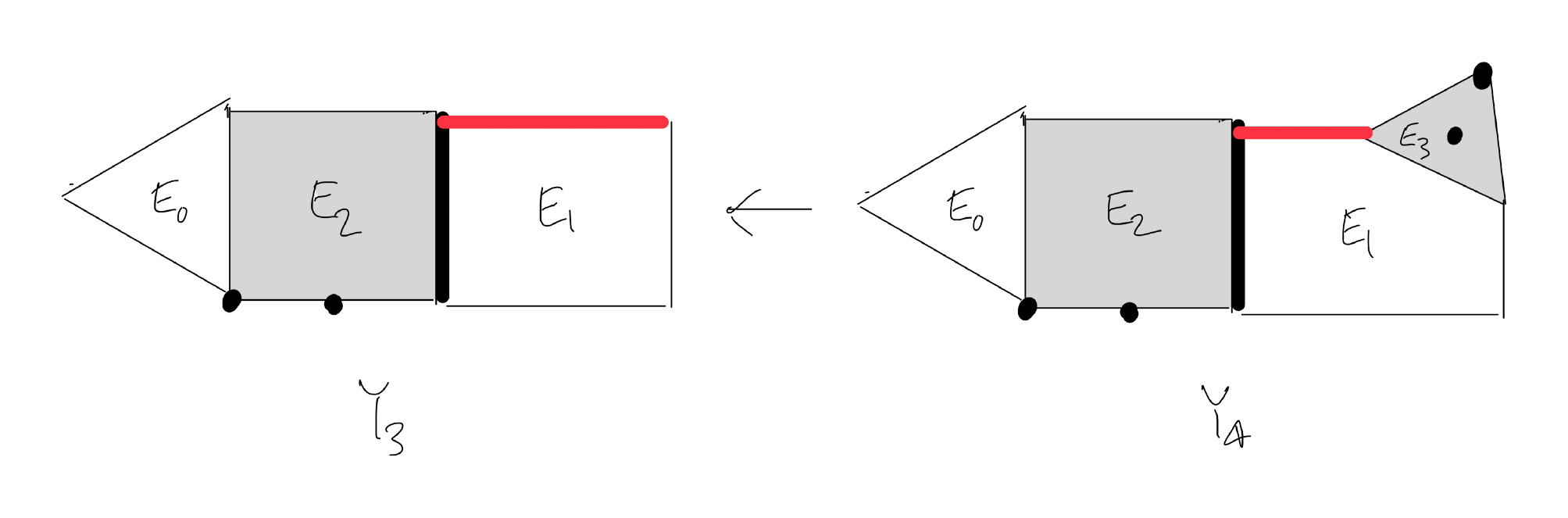}
\caption{For $j=3,4$, $G$ acts freely on $Y_j$ outside the shaded
or marked loci; $Y_j/G$ is regular outside the marked loci;
and $Y_j/G$ has toric singularities outside the red loci.}
\label{F5figure34}
\end{figure}

First, work over the open set $\{r_1=1\}$ in $Y_2$
over $\{y_2=1\}$ in $Y_1$. Then $Y_3$ is the blow-up
along the curve $\{0=y_1=r_2\}=E_0\cap E_1$, and so $Y_3$ has coordinates
$(y_0,y_1,r_2),[z_1,z_2]$. First take $\{z_1=1\}$, so $r_2=y_1z_2$,
and we have coordinates $(y_0,y_1,z_2)$. Here $E_0=\{z_2=0\}$,
$E_1$ does not appear, and $E_2=\{y_1=0\}$. The fixed point scheme $Y_3^G$
is defined by: $I(y_0)=y_1(-1+O(y_1))$, $I(y_1)=y_1^2z_2(-2y_0^2+O(y_1))$,
and $I(z_2)=y_1z_2^2(-y_0^2+O(y_1))$. These equations are equivalent
to $y_1=0$, near $E_2$; so the fixed point scheme $Y_2^G$ is the Cartier
divisor $E_2$, in this chart. (Thus, by Theorem \ref{kl},
$Y_2/G$ is smooth over $k=\F_5$, in this open set.)

The other chart
is $\{z_2=1\}$ in $Y_3$, so $y_1=r_2z_1$, and we have coordinates
$(y_0,z_1,r_2)$. Here $E_0$ does not appear, $E_1=\{z_1=0\}$,
and $E_2=\{r_2=0\}$. The fixed point scheme $Y_3^G$ is given by
$I(y_0)=z_1r_2(-1+O(r_2))$, $I(z_1)=z_1r_2(y_0^2+O(r_2))$,
and $I(r_2)=r_2^2(2y_0^2+O(r_2))$. The fixed point scheme is generically
$E_2$ with multiplicity 1, together with the other fixed curve we knew
from $Y_2$, here given by $\{0=y_0=z_1\}\subset E_1$.
In more detail, the ``bad locus'' where the scheme $Y_3^G$ is not just
$E_2$ as a Cartier divisor is given by removing a factor of $r_2$
from these equations, yielding: $0=z_1(-1+O(r_2))$, $0=z_1(y_0^2+O(r_2))$,
and $0=r_2(2y_0^2+O(r_2))$. We know the fixed locus away from $E_2$,
so assume that $r_2=0$; then these equations show that the bad locus
inside $E_2$ is the curve $\{0=z_1=r_2\}=E_1\cap E_2$.

Fortunately, Theorem \ref{mup} implies that $Y_3/G$ has toric singularities
at points of $E_1\cap E_2$ outside the origin. Namely, let $e=r_2$
and $s=z_1$; then $I(s)=es(\text{unit})$ near $E_1\cap E_2=\{0=z_1=r_2\}$
outside the origin. The theorem gives that $Y_3/G$ has singularity
$\frac{1}{5}(0,1,2)$ at points of $E_1\cap E_2$ outside the origin.

To see all of $E_2$, we also have to work
over $\{w_1=1\}$ in $Y_2$, with coordinates $(w_0,y_1,y_2)$,
over $\{y_0=1\}$ in $Y_1$.
Here $Y_3$ is the blow-up along the $G$-fixed curve
$\{0=w_0=y_1\}=E_0\cap E_1$, so $Y_3$ has coordinates
$(w_0,y_1,y_2),[v_0,v_1]$. First take $\{v_0=1\}$,
so $y_1=w_0v_1$, and we have coordinates $(w_0,v_1,y_2)$ on $Y_3$.
Here $E_0$ does not appear, $E_1=\{v_1=0\}$,
and $E_2=\{w_0=0\}$. The fixed point scheme is defined by:
$I(w_0)=w_0^2(2-2v_1+O(w_0))$, $I(v_1)=w_0v_1(1-2v_1+O(w_0))$,
and $I(y_2)=w_0v_1y_2(1+O(w_0))$. In the chart we are working
over in $Y_2$, the fixed set
$Y_2^G$ is only the curve $E_0\cap E_1$ we are blowing up, and so
$Y_3^G$ (in this chart) is contained in $E_2$ as a set. By the equations,
$Y_3^G$ is generically the Cartier divisor $E_2$, and the bad locus
(where that fails) is given by $0=w_0$, $0=v_1(1-2v_1)$,
and $0=v_1y_2$. So the bad locus is the union of the curve
$\{0=w_0=v_1\}=E_1\cap E_2$ and the point $(w_0,v_1,y_2)=(0,
-2,0)$ in $E_2$. By Theorem \ref{mup} (using $e=s=w_0$),
$Y_3/G$ has singularity $\frac{1}{5}(2,1,0)$ everywhere
on the curve $E_1\cap E_2$ (in this chart), in agreement
with an earlier calculation.

To analyze the bad point above, change coordinates temporarily
by $t_1=v_1+2$; then the bad point becomes the origin in coordinates
$(w_0,t_1,y_2)$. In these coordinates, we have
$I(w_0)=w_0^2(1-2t_1+O(w_0))$, $I(t_1)=I(v_1)=(-t_1-2t_1^2+O(w_0))$,
and $I(y_2)=w_0y_2(-2+O(w_0))$. Theorem \ref{mup} applies,
with $s=e=w_0$, and we read off that $Y_3/G$ has singularity
$\frac{1}{5}(1,-1,-2)$ at this point. That is terminal, by the Reid-Tai
criterion (Theorem \ref{reidtai}).

The last chart we need to consider in $Y_3$ is the other open set
$\{v_1=1\}$ over the open set above in $Y_2$,
$\{w_1=1\}\subset Y_2$ over $\{y_0=1\}\subset Y_1$.
So $w_0=y_1v_0$, and we have coordinates $(v_0,y_1,y_2)$.
Here $E_0=\{v_0=0\}$, $E_1$ does not appear,
and $E_2=\{y_1=0\}$. Here $Y_3^G$ is defined by:
$I(v_0)=v_0y_1(2-v_0+O(y_1))$, $I(y_1)=y_1^2(1-2v_0+O(y_1))$,
and $I(y_2)=y_1y_2(1+O(y_1))$. As in the previous chart,
we know that $Y_3^G$ is contained in $E_2$ as a set.
By the equations,
$Y_3^G$ is generically the Cartier divisor $E_2$, and the bad locus
(where that fails) is given by $0=y_1$, $0=v_0(2-v_0)$,
and $0=y_2$. Thus there are two bad points in this chart,
$(v_0,y_1,y_2)$ equal to $(2,0,0)\in E_2$ or $(0,0,0)\in E_0\cap E_2$.
The first is the bad point from the previous chart, but the second one
is new. Theorem \ref{mup} works to analyze the second point (the origin),
with $e=s=y_1$. We read off that $Y_3/G$ has singularity
$\frac{1}{5}(2,1,1)$ at this point.

That finishes the analysis of $Y_3$. In particular, as a set, $Y_3^G$
is the union of the divisor $E_2$ and a curve in $E_1$. It is tempting
to blow up the $G$-fixed curve next,
but that leads to a large number of blow-ups over one point of the curve,
where the fixed point scheme is especially complicated. We therefore
define $Y_4$ as the blow-up at that point,
and only later
blow up the whole curve. This leads
more efficiently to toric singularities.

Namely, let $Y_4$ be the blow-up of $Y_3$ at the origin
in the chart $\{r_2=1\}$ in $Y_2$ (unchanged in $Y_3$),
with coordinates $(y_0,r_1,s_2)$. So $Y_4$ has coordinates
$(y_0,r_1,s_2),[q_0,q_1,q_2]$. The exceptional divisor $E_3$
is isomorphic to $\P^2$, and so it is covered by 3 affine charts.
First take $\{q_0=1\}$ in $Y_4$, so $r_1=y_0q_1$
and $s_2=y_0q_2$, and we have coordinates $(y_0,q_1,q_2)$.
Here $E_1=\{q_2=0\}$ and $E_3=\{y_0=0\}$. The fixed point scheme
$Y_4^G$ is defined by: $I(y_0)=y_0^2q_2(1-q_1+O(y_0))$,
$I(q_1)=y_0(-2+q_1q_2+q_1^2q_2+O(y_0))$,
and $I(q_2)=y_0q_2^2(-2+q_1+O(y_0))$. So $Y_4^G$ is generically
the Cartier divisor $E_3$; the $G$-fixed curve in $E_1$ does not appear
in this chart. The bad locus (where the scheme $Y_4^G$ is not
just $E_3$) is given by $0=y_0$, $0=-2+q_1q_2+q_1^2q_2$,
and $0=q_2^2(-2+q_1)$. By the second equation, $q_2\neq 0$,
and so the third equation gives that $q_1=2$. Then the second equation
gives that $0=-2+2q_2-q_2=-2+q_2$, so $q_2=2$. That is,
there is only one bad point in this chart,
$(y_0,q_1,q_2)=(0,2,2)\in E_3$. To analyze that point,
change coordinates temporarily by $s_1=q_1-2$ and $s_2=q_2-2$.
In these coordinates, $I(y_0)=y_0^2(-2-s_1-s_2-s_1s_2+O(y_0))$,
$I(s_1)=I(q_1)=y_0(s_2+2s_1^2+s_1^2s_2+O(y_0))$,
and $I(s_2)=I(q_2)=y_0(-s_1-s_1s_2+s_1s_2^2+O(y_0))$.
By Theorem \ref{mup}, with $e=s=y_0$, $Y_4/G$
has a $\mu_5$-quotient singularity. Explicitly, the linear map
$\varphi$ over $k$ in the theorem is $\varphi(y_0)=-2y_0$,
$\varphi(s_1)=s_2$, and $\varphi(s_2)=-s_1$,
which has eigenvalues $-2,2,-2$. So $Y_4/G$ has singularity
$\frac{1}{5}(-2,2,-2)$ at this point. This is terminal,
by the Reid-Tai criterion.

Next, take the open set $\{q_1=1\}$ in $Y_4$, so $y_0=r_1q_0$
and $s_2=r_1q_2$, and $Y_4$ has coordinates
$(q_0,r_1,q_2)$. Here $E_1=\{q_2=0\}$ and $E_3=\{r_1=0\}$.
The fixed point scheme $Y_4^G$ is defined by:
$I(q_0)=r_1(-q_2-q_0q_2+2q_0^3+O(r_1))$,
$I(r_1)=r_1^2(2q_2-2q_0^2+O(r_1))$, and
$I(q_2)=r_1q_2(2q_2+2q_0^2+O(r_1))$. So $Y_4^G$ is generically
the Cartier divisor $E_3$, together with the $G$-fixed curve
$\{0=q_0=q_2\}$ in $E_1$. The bad locus in $E_3$
is given by $0=r_1$, $0=-q_2-q_0q_2+2q_0^3$,
and $0=q_2(2q_2+2q_0^2)$. This yields two bad points,
$(q_0,r_1,q_2)$ equal to $(-2,0,1)$ or $(0,0,0)$. The first one
is the bad point
from the previous chart, and the second is not surprising,
as it is the intersection point of $E_3$ with the $G$-fixed curve.

Finally, take the open set $\{q_2=1\}$ in $Y_4$,
so $y_0=s_2q_0$ and $r_1=s_2q_1$, and we have coordinates
$(q_0,q_1,s_2)$. Here $E_1$ does not appear, and $E_3=\{s_2=0\}$.
The fixed point scheme $Y_4^G$ is defined by:
$I(q_0)=s_2(2q_0-q_1+O(s_2))$, $I(q_1)=s_2(-2q_1-2q_0^2+O(s_2))$,
and $I(s_2)=s_2^2(-1+O(s_2))$. So $Y_4^G$ is generically $E_3$.
The bad locus in $E_3$ is given by: $0=s_2$,
$0=2q_0-q_1$, and $0=-2q_1-2q_0^2$. This yields two bad points,
$(q_0,q_1,s_2)$ equal to $(-2,1,0)$ (seen in the previous two charts)
or $(0,0,0)$, which is new. Theorem \ref{mup} applies
at this new point, with $e=s=s_2$. The $k$-linear map $\varphi$
is given by $\varphi(q_0)=2q_0-q_1$, $\varphi(q_1)=-2q_1$,
and $\varphi(s_2)=-s_2$. So $\varphi$ has eigenvalues $(2,-2,-1)$,
and hence $Y_4/G$ has singularity
$\frac{1}{5}(2,-2,-1)$ at this point. This is terminal,
by the Reid-Tai criterion.

That completes our description of $Y_4$. Next, let $Y_5$ be the blow-up
of $Y_4$ along the $G$-fixed curve in $E_1$. The exceptional
divisor $E_4$ in $Y_5$ is a $\P^1$-bundle over $\P^1$,
and so it is covered by four affine charts. First work over the open set
$\{z_2=1\}$ in $Y_3$ (unchanged in $Y_4$), with coordinates
$(y_0,z_1,r_2)$; this contains the point where the $G$-fixed curve
in $E_1=\{z_1=0\}$ meets $E_2=\{r_2=0\}$.
Here $Y_5$ is the blow-up along the $G$-fixed curve
$\{0=y_0=z_1\}$, and so $Y_5$ has coordinates $(y_0,z_1,r_2),[n_0,n_1]$.
First take the open set $\{n_0=1\}$ in $Y_5$, so $z_1=y_0n_1$, and we have
coordinates $(y_0,n_1,r_2)$. Here $E_1=\{n_1=0\}$, $E_2=\{r_2=0\}$,
and $E_4=\{y_0=0\}$. The fixed point scheme $Y_5^G$ is defined by:
$I(y_0)=y_0n_1r_2(-1+O(y_0))$, $I(n_1)=n_1r_2(n_1+O(y_0))/(1-n_1r_2+O(y_0))$,
and $I(r_2)=y_0r_2^2(-2n_1r_2+O(y_0))$. So $Y_5^G$, as a set, is the union
of the divisor $E_2$ and the curve $\{0=y_0=n_1\}=E_1\cap E_4$. (In particular,
the fixed point set is still not all of codimension 1.) We have analyzed
the bad locus of $E_2$ in previous steps, but we have to add here
that the bad locus of $E_2$ is disjoint from $E_2\cap E_4$ except
for the point where $E_2$ meets the $G$-fixed curve, by the formula
for $I(n_1)$.
\begin{figure}
\includegraphics[width=\textwidth]{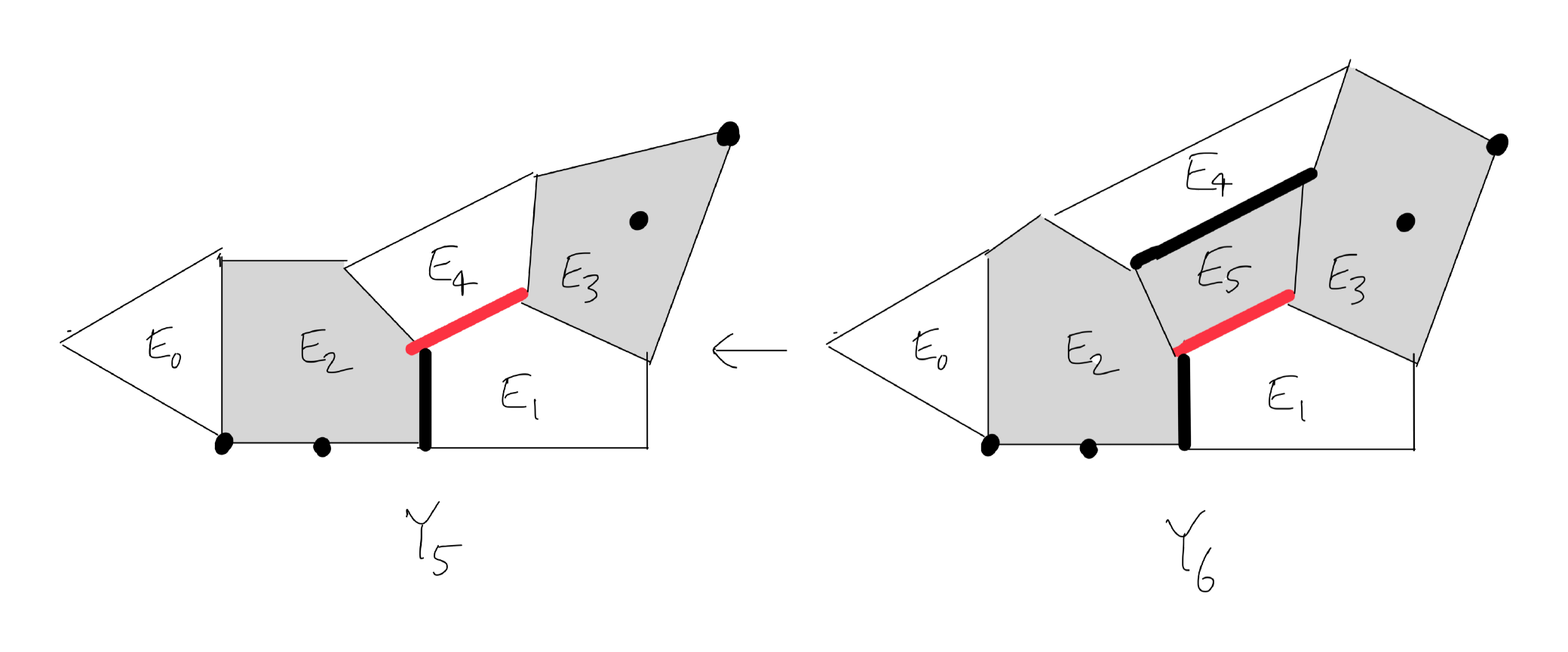}
\caption{For $j=5,6$, $G$ acts freely on $Y_j$ outside the shaded
or marked loci; $Y_j/G$ is regular outside the marked loci;
and $Y_j/G$ has toric singularities outside the red loci.}
\label{F5figure56}
\end{figure}

The other chart is $\{n_1=1\}$ in $Y_5$, so $y_0=z_1n_0$,
and we have coordinates $(n_0,z_1,r_2)$. Here $E_1$ does not appear,
$E_2=\{r_2=0\}$, and $E_4=\{z_1=0\}$. The fixed point scheme $Y_5^G$
is defined by: $I(n_0)=r_2(-1+O(z_1))$, $I(z_1)=z_1^2r_2(-2r_2+O(z_1))$,
and $I(r_2)=z_1r_2^2(-2r_2+O(z_1))$. These equations reduce to $r_2=0$
near $E_4$, and so $Y_5^G$ is the Cartier divisor $E_2$,
in this chart.

To finish our description of $E_4$ in $Y_5$, we work
over the open set where the $G$-fixed curve in $Y_4$ meets $E_3$,
namely $\{q_1=1\}$ in $Y_4$. Here $Y_4$ has coordinates
$(q_0,r_1,q_2)$, $E_1=\{q_2=0\}$, $E_3=\{r_1=0\}$,
and the $G$-fixed curve is $\{0=q_0=q_2\}$ in $E_1$.
So the blow-up $Y_5$ along the $G$-fixed curve has coordinates
$(q_0,r_1,q_2),[u_0,u_2]$. First take $\{u_0=1\}$ in $Y_5$,
so $q_2=q_0u_2$, and we have coordinates $(q_0,r_1,u_2)$.
Here $E_1=\{u_2=0\}$, $E_3=\{r_1=0\}$, and $E_4=\{q_0=0\}$.
The fixed point scheme $Y_5^G$ is defined by:
$I(q_0)=q_0r_1(-u_2+O(q_0))$, $I(r_1)=q_0r_1^2(2u_2+O(q_0))$,
and $I(u_2)=r_1u_2^2(1+O(q_0))/(1-r_1u_2+O(q_0))$. We know the fixed
set outside $E_4$, and so we read off that the fixed set is the divisor $E_3$
together with the $G$-fixed curve $E_1\cap E_4$ found earlier.
We have analyzed the bad set of $E_3$ away from $E_4$ in earlier blow-ups,
and we see from the formula for $I(u_2)$
that the bad set of $E_3$ near $E_3\cap E_4$ is only
the point $E_1\cap E_3\cap E_4$ where the $G$-fixed curve meets $E_3$.

The other chart is $\{u_2=1\}$ in $Y_5$. Here $q_0=q_2u_0$, and so we have
coordinates $(u_0,r_1,q_2)$. Here $E_1$ does not appear,
$E_3=\{r_1=0\}$, and $E_4=\{q_2=0\}$. The fixed point scheme
$Y_5^G$ is defined by: $I(u_0)=r_1(-1+O(q_2))$,
$I(r_1)=r_1^2q_2(2+O(q_2))$, and $I(q_2)=r_1q_2^2(2+O(q_2))$. These
equations reduce to $r_1=0$ near $E_4$, and so the fixed point scheme
$Y_5^G$ is the Cartier divisor $E_3$, in this chart.

That completes our description of $Y_5$. Let $Y_6$ be the blow-up
of $Y_5$ along the $G$-fixed curve $E_1\cap E_4$. The exceptional
divisor $E_5$ in $Y_6$ is a $\P^1$-bundle over $\P^1$, covered by
four affine charts. First take the open set $\{n_0=1\}$ in $Y_5$, 
which contains the point where the $G$-fixed curve meets $E_2$.
Here $Y_5$ has coordinates $(y_0,n_1,r_2)$,
with $E_1=\{n_1=0\}$, $E_2=\{r_2=0\}$,
and $E_4=\{y_0=0\}$. Since $Y_6$ is the blow-up along
the $G$-fixed curve $\{0=y_0=n_1\} =E_1\cap E_4$, $Y_6$
has coordinates $(y_0,n_1,r_2),[m_0,m_1]$. First take $\{m_0=1\}$
in $Y_6$, so $n_1=y_0m_1$, and we have coordinates
$(y_0,m_1,r_2)$. Here $E_1=\{m_1=0\}$, $E_2=\{r_2=0\}$,
$E_4$ does not appear, and $E_5=\{y_0=0\}$. The fixed point scheme
$Y_6^G$ is defined by: $I(y_0)=y_0^2m_1r_2(-1+O(y_0))$,
$I(m_1)=y_0m_1r_2(2m_1+O(y_0))$,
and $I(r_2)=y_0^2r_2^2(2-2m_1r_2+O(y_0))$. We know the fixed point
set away from $E_5$, and so we read off that the fixed point scheme
is generically the Cartier divisor $E_2+E_5$. (Since $E_5$ is fixed
by $G$, we have finally made the fixed point set of codimension 1.)
Let $e=y_0r_2$. The bad locus
(where the scheme $Y_6^G$ is more than the Cartier divisor $E_2+E_5$),
on $E_5$, is given by factoring out $e$ from the equations
and setting $y_0=0$, so we get: $0=y_0$ and $0=2m_1^2$. So,
as a set, the bad locus is the curve $\{0=y_0=m_1\}=E_1\cap E_5$.
Theorem \ref{mup} does not seem to apply to this curve,
and so $Y_6/G$ might not
have toric singularities there; we will have to blow up one more time.

For now, look at the other open set, $\{m_1=1\}$ in $Y_6$.
So $s_0=n_1m_0$, and we have coordinates $(m_0,n_1,r_2)$. Here
$E_1$ does not appear, $E_2=\{r_2=0\}$, $E_4=\{m_0=0\}$,
and $E_5=\{n_1=0\}$. The fixed point scheme $Y_6^G$ is defined by:
$I(m_0)=m_0n_1r_2(-2+O(n_1))$, $I(n_1)=n_1^2r_2(1+O(n_1))$,
and $I(r_2)=m_0n_1^2r_2^2(2m_0-2r_2+O(n_1))$. Since we know the fixed point
set outside $E_5$, we read off that the fixed point scheme
is generically the Cartier divisor $E_2+E_5$. Let $e=n_1r_2$.
The bad locus (where the scheme $Y_6^G$ is more than the Cartier
divisor $E_2+E_5$) is the curve $\{0=m_0=n_1\}=E_4\cap E_5$.
Fortunately, Theorem \ref{mup} applies, with $s=m_0$. We read
off that $Y_6/G$ has singularity $\frac{1}{5}(-2,1,0)$ along the whole
curve $E_4\cap E_5$, in this chart.

To finish describing $E_5\subset Y_6$, we have to work over the open
set $\{u_0=1\}$ in $Y_5$, where the $G$-fixed curve $E_1\cap E_4$
in $Y_5$ meets $E_3$. Here $Y_5$ has coordinates
$(q_0,r_1,u_2)$, with $E_1=\{u_2=0\}$, $E_3=\{r_1=0\}$, and $E_4=\{q_0=0\}$.
Then $Y_6$ is the blow-up along the $G$-fixed curve
$\{0=q_0=u_2\}=E_1\cap E_4$, so $Y_6$ has coordinates
$(q_0,r_1,u_2),[t_0,t_2]$. First take $\{t_0=1\}$ in $Y_6$,
so $u_2=q_0u_2$ and we have coordinates $(q_0,r_1,t_2)$.
Here $E_1=\{t_2=0\}$, $E_3=\{r_1=0\}$,
$E_4$ does not appear, and $E_5=\{q_0=0\}$. The fixed point scheme
is defined by: $I(q_0)=q_0^2r_1(-t_2+O(q_0))$,
$I(r_1)=q_0^2r_1^2(-2+2t_2+O(q_0))$, and $I(t_2)=q_0r_1t_2(2t_2+O(q_0))$.
So the fixed point scheme is generically $E_3+E_5$. Let $e=q_0r_1$.
The bad locus (where the scheme $Y_6^G$ is more than the Cartier
divisor $E_3+E_5$), in $E_5$, is given by $0=q_0$
and $0=2t_2^2$, so (as a set) it is the curve $\{0=q_0=t_2\}
=E_1\cap E_5$, which we met in an earlier chart.

The other chart is $\{t_2=1\}$ in $Y_6$, so $q_0=u_2t_0$,
and we have coordinates $(t_0,r_1,u_2)$. Here $E_1$ does not appear,
$E_3=\{r_1=0\}$, $E_4=\{t_0=0\}$, and $E_5=\{u_2=0\}$.
The fixed point scheme $Y_6^G$ is defined by: $I(t_0)=t_0r_1u_2(-2+O(u_2))$,
$I(r_1)=t_0r_1^2u_2^2(2-2t_0+O(u_2))$, and
$I(u_2)=r_1u_2^2(1+O(u_2))$. So $Y_6^G$ is generically $E_3+E_5$.
Let $e=r_1u_2$.
The bad locus (where the scheme $Y_6^G$ is more than the Cartier
divisor $E_3+E_5$), in $E_5$, is the curve $\{0=t_0=u_2\}
=E_4\cap E_5$, which we met in an earlier chart. Theorem \ref{mup}
applies, with $s=t_0$. Namely, $Y_6/G$ has singularity
$\frac{1}{5}(-2,0,1)$ everywhere on the curve $E_4\cap E_5$ in this chart
(including the origin, which did not appear in the earlier chart).

That completes our description of $Y_6$. In particular, the $G$-fixed locus
has codimension 1 in $Y_6$, and $Y_6/G$ has toric singularities
outside the image of the curve $E_1\cap E_5$. Let $Y_7$ be the blow-up
of $Y_6$ along that curve. The exceptional divisor $E_6$ in $Y_7$
is a $\P^1$-bundle over $\P^1$, and so we will cover $E_6$
with four affine charts. First, work over the open set $\{m_0=1\}$
in $Y_6$, where the bad curve $E_1\cap E_5$ meets $E_2$.
Here $Y_6$ has coordinates
$(y_0,m_1,r_2)$, with $E_1=\{m_1=0\}$, $E_2=\{r_2=0\}$,
and $E_5=\{y_0=0\}$. Since $Y_7$ is the blow-up along the curve
$\{0=y_0=m_1\}=E_1\cap E_5$, $Y_7$ has coordinates
$(y_0,m_1,r_2),[j_0,j_1]$. 
\begin{figure}
\includegraphics[width=0.5\textwidth]{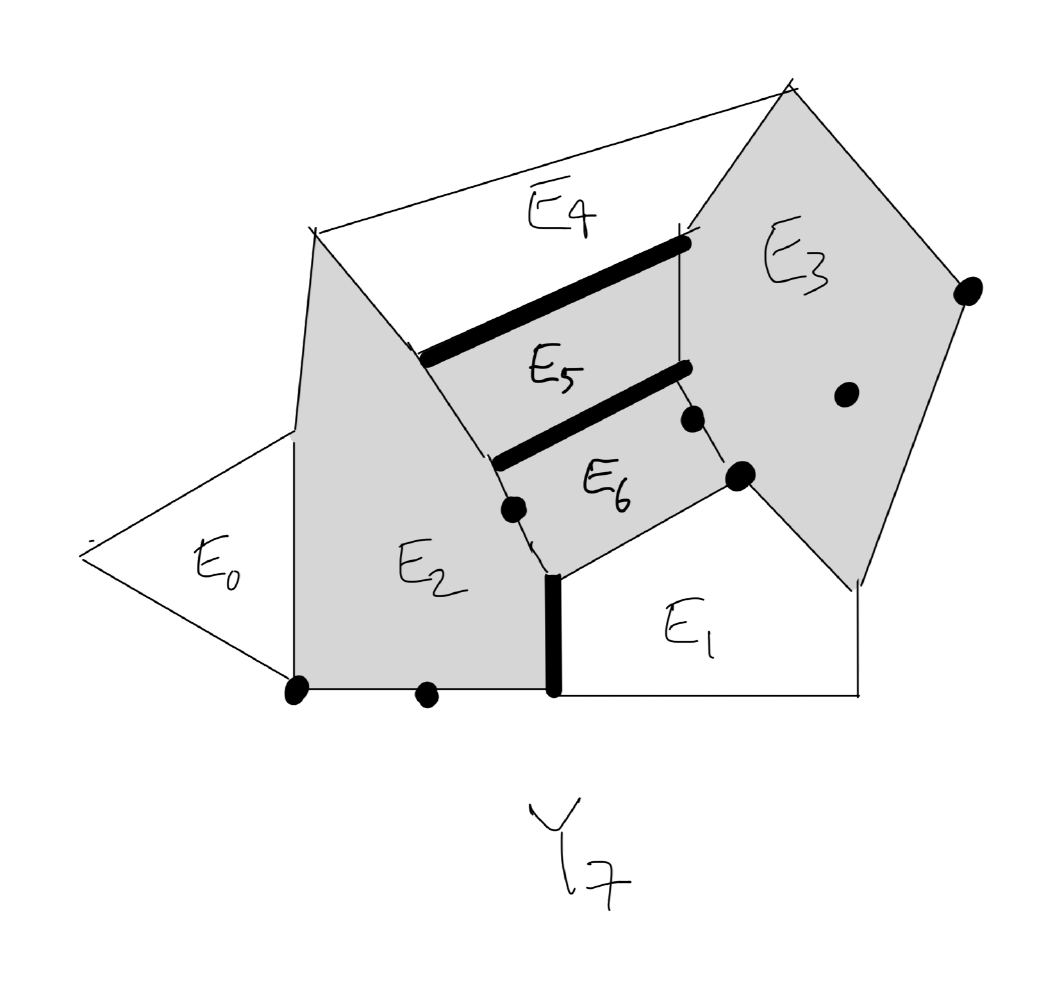}
\caption{$G$ acts freely on $Y_7$ outside the shaded
or marked loci, and $Y_7/G$ is regular outside the marked loci.
Here $Y_7/G$ has toric singularities.}
\label{F5figure7}
\end{figure}

First take $\{j_0=1\}$ in $Y_7$,
so $m_1=y_0j_1$, and we have coordinates $(y_0,j_1,r_2)$. Here
$E_1=\{j_1=0\}$, $E_2=\{r_2=0\}$, $E_5$ does not appear,
and $E_6=\{y_0=0\}$. The fixed point scheme $Y_7^G$
is defined by: $I(y_0)=y_0^3j_1r_2(-1+O(y_0))$,
$I(j_1)=y_0^2j_1r_2(1-2j_1+O(y_0))$,
and $I(r_2)=y_0^2r_2^2(2+O(y_0))$. So $Y_7^G$ is generically
the Cartier divisor $E_2+2E_6$. Let $e=y_0^2r_2$. The bad locus
(where the scheme $Y_7^G$ is more than the Cartier divisor $E_2+2E_6$),
in $E_6$, is given by $0=y_0$, $0=j_1(1-2j_1)$, and $0=r_2$,
so it consists of the two points $(y_0,j_1,r_2)$ equal to $(0,0,0)
=E_1\cap E_2\cap E_6$
or $(0,-2,0)\in E_2\cap E_6$. At the first point, Theorem \ref{mup} applies,
with $s=r_2$. We read off that $Y_7/G$ has singularity
$(0,1,2)$ everywhere on the curve $E_1\cap E_2$ (including the origin,
which did not appear when we saw $E_1\cap E_2$ in an earlier chart).
To analyze the second point, change coordinates temporarily
by $s_1=j_1+2$; then that point becomes the origin
in coordinates $(y_0,s_1,r_2)$. We have $I(y_0)=y_0^3r_2(2-s_1+O(y_0))$,
$I(s_1)=I(j_1)=y_0^2r_2(-s_1-2s_1^2+O(y_0))$,
and $I(r_2)=y_0^2r_2^2(2+O(y_0))$. Theorem \ref{mup} applies,
with $s=r_2$. We read off that $Y_7/G$ has singularity
$\frac{1}{5}(2,-1,2)$ at this point.

The other open set is $\{j_1=1\}$ in $Y_7$, so $y_0=m_1j_0$,
and we have coordinates $(j_0,m_1,r_2)$. Here $E_1$ does not appear,
$E_2=\{r_2=0\}$, $E_5=\{j_0=0\}$, and $E_6=\{m_1=0\}$.
The fixed point scheme $Y_7^G$ is defined by:
$I(j_0)=j_0^2m_1^2r_2(2-j_0+O(m_1))$,
$I(m_1)=j_0m_1^3r_2(2+j_0+O(m_1))$, and
$I(r_2)=j_0^2m_1^2r_2^2(2+O(m_1))$. So $Y_7^G$
is generically $E_2+E_5+2E_6$. Let $e=j_0m_1^2r_2$.
The bad locus (where the scheme $Y_7^G$ is more than the Cartier
divisor $E_2+E_5+2E_6$), on $E_6$, is given by:
$0=m_1$, $0=j_0(2-j_0)$, and $0=j_0r_2$. So the bad locus
is the union of the curve $\{0=j_0=m_1\}=E_5\cap E_6$
and the point $(j_0,m_1,r_2)=(2,0,0)$ in $E_2\cap E_6$.
That point is the one we analyzed in the previous chart.
For the curve, Theorem \ref{mup} applies, using
$s=j_0$. We read off that $Y_7/G$ has singularity
$\frac{1}{5}(2,2,0)$ everywhere on the curve $E_5\cap E_6$,
in this chart.

Last, work over the open set
$\{t_0=1\}$ in $Y_6$,
where the bad curve $E_1\cap E_5$ meets $E_3$. 
Here $Y_6$ has coordinates $(q_0,r_1,t_2)$,
with $E_1=\{t_2=0\}$, $E_3=\{r_1=0\}$,
and $E_5=\{q_0=0\}$. We obtain $Y_7$ by blowing up along the curve
$\{0=q_0=t_2\}=E_1\cap E_5$, so $Y_7$ has coordinates
$(q_0,r_1,t_2),[x_0,x_2]$. First take $\{x_0=1\}$,
so $t_2=q_0x_2$, and we have coordinates $(q_0,r_1,x_2)$.
Here $E_1 = \{x_2=0\}$, $E_3 = \{r_1=0\}$,
$E_5$ does not appear, and $E_6 = \{q_0=0\}$. The fixed point
scheme $Y_7^G$ is defined by:
$I(q_0)=q_0^3r_1(2-x_2+O(q_0))$, $I(r_1)=q_0^2r_1^2(-2+O(q_0))$,
and $I(x_2)=q_0^2r_1x_2(1-2x_2+O(q_0))$. So $Y_7^G$ is generically
$E_3+2E_6$. Let $e=q_0^2r_1$.
The bad locus (where the scheme $Y_7^G$ is more than
the Cartier divisor $E_3+2E_6$), in $E_6$, is given by:
$0=q_0$, $0=r_1$, and $0=x_2(1-2x_2)$, so it consists
of the two points $(q_0,r_1,x_2)$ equal to $(0,0,0)=E_1\cap E_3\cap E_6$
or $(0,0,-2)$ in $E_3\cap E_6$. Since $I(r_1)=er_1(\text{unit})$,
Theorem \ref{mup} applies at both points. At the origin,
the theorem gives that $Y_7/G$ has singularity
$\frac{1}{5}(2,-2,1)$, which is terminal.
For the other point, change coordinates
temporarily by $y_2=x_2+2$, so that the point becomes the origin
in coordinates $(q_0,r_1,y_2)$. We have $I(q_0)
=q_0^3r_1(-1-y_2+O(q_0))$, $I(r_1)=q_0^2r_1^2(-2+O(q_0))$,
and $I(y_2)=I(x_2)=q_0^2r_1(-y_2-2y_2^2+O(q_0))$.
So Theorem \ref{mup} gives that $Y_7/G$ has singularity
$\frac{1}{5}(-1,-2,-1)$ at this point.

The other chart is $\{x_2=1\}$ in $Y_7$, so $q_0=t_2x_0$,
and we have coordinates $(x_0,r_1,t_2)$. Here $E_1$ does not appear,
$E_3=\{r_1=0\}$, $E_5=\{x_0=0\}$, and $E_6=\{t_2=0\}$. The fixed
point scheme $Y_7^G$ is defined by:
$I(x_0)=x_0^2r_1t_2^2(2-x_0+O(t_2))$,
$I(r_1)=x_0^2r_1^2t_2^2(-2+O(t_2))$,
and $I(t_2)=x_0r_1t_2^3(2-2x_0+O(t_2))$. So $Y_7^G$
is generically $E_3+E_5+2E_6$. Let $e=x_0r_1t_2^2$.
The bad locus (where $Y_7^G$
is more than the Cartier divisor $E_3+E_5+2E_6$),
in $E_6$, is given by: $0=t_2$, $0=x_0(2-x_0)$, and $0=x_0r_1$,
which is the union of the curve $\{0=x_0=t_2\}=E_5\cap E_6$
and the point $(x_0,r_1,t_2)=(2,0,0)$ in $E_3\cap E_6$.
We analyzed that point in the previous chart. Theorem \ref{mup}
applies to the curve, using $s=x_0$. We read off that $Y_7/G$
has singularity $\frac{1}{5}(2,0,2)$ everywhere on the curve
$E_5\cap E_6$ (including the origin, which did not appear
in the earlier chart where we met this curve).

That completes our analysis of $Y_7$; we have shown
that $Y_7/G$ has toric singularities. It will now
be straightforward to show that $Y_0/G$ is terminal.

First, we can compute the canonical class of $Y_7$,
since $Y_7$ is obtained from $Y_0$ by a sequence
of blow-ups along points and smooth curves. Write $E_j$
for the strict transform of the exceptional divisor in $Y_{j+1}$
to any higher model. Write $\pi_{ij}$
for the morphism $Y_i\to Y_j$ (with $i>j$), and also for the
resulting morphism $Y_i/G\to Y_j/G$. First, $K_{Y_1}=\pi_{10}^*K_{Y_0}
+2E_0$, since $Y_1\to Y_0$ is the blow-up of a smooth 3-fold
at a point. Next, $K_{Y_2}=\pi_{21}^*K_{Y_1}+E_1$, since $Y_2\to Y_1$
is the blow-up along a smooth curve, and we have $\pi_{21}^*E_0=E_0+E_1$
because the curve being blown up is contained in $E_0$.
Likewise, we have:
\begin{align*}
K_{Y_3}=\pi_{32}^*K_{Y_2}+E_2,\; \pi_{32}^*E_0=E_0+E_2,\; &
\pi_{32}^*E_1=E_1+E_2,\\
K_{Y_4}=\pi_{43}^*K_{Y_3}+2E_3,\; \pi_{43}^*E_0=E_0,\; &
\pi_{43}^*E_1=E_1+E_3,\; \pi_{43}^*E_2=E_2,\\
K_{Y_5}=\pi_{54}^*K_{Y_4}+E_4,\; \pi_{54}^*E_0=E_0,&
\pi_{54}^*E_1=E_1+E_4,\; \pi_{54}^*E_2=E_2,\;
\pi_{54}^*E_3=E_3,\\
K_{Y_6}=\pi_{65}^*K_{Y_5}+E_5,\; \pi_{65}^*E_j=E_j&
\text{ for }j\in\{0,\ldots,4\}-\{1,4\},\\
\; \pi_{65}^*E_1=E_1+E_5,\; & \pi_{65}^*E_4=E_4+E_5,\\
K_{Y_7}=\pi_{76}^*K_{Y_6}+E_6, \pi_{76}^*E_j=E_j&
\text{ for }j\in\{0,\ldots,5\}-\{1,5\},\\
\; \pi_{76}^*E_1=E_1+E_6,\; & \pi_{76}^*E_5=E_5+E_6.
\end{align*}
Combining these equations gives that
$$K_{Y_7}=\pi_{70}^*K_{Y_0}+2E_0+3E_1+6E_2+5E_3+4E_4
+8E_5+12E_6.$$

Write $f$ for any of the quotient maps $Y_j\to Y_j/G$.
First, $K_{Y_0/G}$ is $\Q$-Cartier since $Y_0$ is smooth,
and $K_{Y_0}=f^*K_{Y_0/G}$ because $f\colon Y_0\to Y_0/G$
is \'etale in codimension 1. Next,
we computed that the fixed point scheme $Y_7^G$ is the Cartier
divisor $E_2+E_3+E_5+2E_6$ outside a codimension-2 subset of $Y_7$.
By section \ref{ramsection}, it follows that
\begin{align*}
K_{Y_7}&=f^*K_{Y_7/G}+(p-1)(E_2+E_3+E_5+2E_6)\\
&=f^*K_{Y_7/G}+4E_2+4E_3+4E_5+8E_6.
\end{align*}
For each $j\in \{0,\ldots,7\}$, let $F_j$ be the image of $E_j$
in $Y_7/G$, as an irreducible divisor. For $j\in \{0,1,4\}$,
$G$ acts nontrivially on $E_j$ (so $f$ is unramified
along $E_j$), and hence $E_j=f^*F_j$. For the other $j$'s,
in $\{2,3,5,6\}$, section \ref{ramsection} and our calculations
imply that $f$ is fiercely ramified along $E_j$, and so again
we have $E_j=f^*F_j$. (For example, for $E_2$, use the first
chart where $E_2$ appeared, $\{z_1=1\}$ in $Y_3$. There
$E_2$ is the divisor $\{y_1=0\}$, and $Y_3^G$ has multiplicity
1 along $E_2$, but $I(y_1)=y_1^2z_2(-2y_0^2+O(y_1))$ vanishes
to order $2>1$ along $E_2$; so section \ref{ramsection} gives that
$f$ is fiercely ramified along $E_2$.)

We can combine these results to compute the discrepancies
of the morphism $Y_7/G\to Y_0/G$. Namely, we have
\begin{align*}
f^*(K_{Y_7/G}-\pi_{70}^*K_{Y_0/G})&=
f^*K_{Y_7/G}-\pi_{70}^*f^*K_{Y_0/G}\\
&=K_{Y_7}-4E_2-4E_3-4E_5-8E_6-\pi_{70}^*K_{Y_0}\\
&=2E_0+3E_1+2E_2+E_3+4E_4+4E_5+4E_6\\
&=f^*(2F_0+3F_1+2F_2+F_3+4F_4+4F_5+4F_6).
\end{align*}
Therefore, $K_{Y_7/G}=\pi_{70}^*(K_{Y_0/G})+2F_0+3F_1+2F_2+F_3+4F_4
+4F_5+4F_6$. In particular, these coefficients are all positive,
which is part of showing that $Y_0/G$ is terminal. (That would be all
we need if $Y_7/G$ were smooth.)

To show that $Y_0/G$ is terminal, it now suffices
to show that the pair $(Y_7/G,D)$ is terminal, where
$D:=-2F_0-3F_1-2F_2-F_3-4F_4-4F_5-4F_6$.
Because the coefficients of $D$ are negative (which works
to our advantage), this is clear
at points where $Y_7/G$ is terminal. There are 6 subvarieties (points
or curves) where $Y_7/G$ is not terminal, as we now address.

(1) Along the curve $E_1\cap E_2$, $Y_7/G$ has singularity
$\frac{1}{5}(0,1,2)$, with $F_1$ a toric divisor of weight 1
and $F_2$ a toric divisor of weight 2, using Theorem \ref{mup-divisor}.
To show that $(Y_7/G,D)=(Y_7/G,-3F_1-2F_2-\cdots)$
is terminal along this curve, we need to show that
$4(i \bmod 5)+3(2i \bmod 5)>5$ for $i=1,\ldots 4$,
by Theorem \ref{reidtaipair}. This is clear, since the left side
is $\geq 4+3=7>5$.

(2) At a point in $E_0\cap E_2$, $Y_7/G$ has singularity
$\frac{1}{5}(2,1,1)$, with $F_0$ of weight 2
and $F_2$ of weight 1. To show that $(Y_7/G,D)=(Y_7/G,
-2F_0-2F_2-\cdots)$ is terminal, we need that
$3(2i \bmod 5)+3(i \bmod 5)+(i\bmod 5)>5$ for $i=1,\ldots 4$.
Indeed, the left side is $\geq 3+3+1=7>5$.

(3) Along the curve $E_4\cap E_5$, $Y_7/G$ has singularity
$\frac{1}{5}(-2,1,0)$, with $E_4$ of weight $-2$
and $E_5$ of weight 1. To show that $(Y_7/G,D)
=(Y_7/G,-4F_4-4F_5-\cdots)$ is terminal, we need that
$5(-2i\bmod 5)+5(i\bmod 5)>5$ for $i=1,\ldots 4$.
Indeed, the left side is $\geq 5+5=10>5$.

(4) At a point in $E_2\cap E_6$, $Y_7/G$ has singularity
$\frac{1}{5}(2,-1,2)$, with $E_2$ and $E_6$ both of weight 2.
To show that $(Y_7/G,D)
=(Y_7/G,-2F_2-4F_6-\cdots)$ is terminal, we need that
$3(2i\bmod 5)+5(2i\bmod 5)+(-i\bmod 5)>5$ for $i=1,\ldots 4$.
Indeed, the left side is $\geq 3+5+1=9>5$.

(5) Along the curve $E_5\cap E_6$, $Y_7/G$ has singularity
$\frac{1}{5}(2,2,0)$, with $E_5$ and $E_6$ both of weight 2.
To show that $(Y_7/G,D)
=(Y_7/G,-4F_5-4F_6-\cdots)$ is terminal, we need that
$5(2i\bmod 5)+5(2i\bmod 5)>5$ for $i=1,\ldots 4$.
Indeed, the left side is $\geq 5+5=10>5$.

(6) At a point in $E_3\cap E_6$, $Y_7/G$ has singularity
$\frac{1}{5}(-1,-2,-1)$, with $E_3$ of weight $-2$
and $E_6$ of weight $-1$.
To show that $(Y_7/G,D)
=(Y_7/G,-F_3-4F_6-\cdots)$ is terminal, we need that
$2(-2i\bmod 5)+5(-i\bmod 5)+(-i\bmod 5)>5$ for $i=1,\ldots 4$.
Indeed, the left side is $\geq 2+5+1=8>5$.

That completes the proof that $Y_0/G$ is terminal.
Theorem \ref{char5} is proved.
\end{proof}

\begin{remark}
The divisor class $\pi^*K_{Y_0/G}=K_{Y_7/G}+D$
happens to be Cartier on the loci (1)--(6),
above. However, it is not Cartier at the terminal singularity
$\frac{1}{5}(2,-2,-1)$ in $E_3$; one can compute that
some discrepancies at divisors
over that point are not integers. As a result, $K_{Y_0/G}$ is not Cartier
(as one can also check directly).
I expect that there is also a 3-fold $X$
over $\F_5$ that is terminal and non-Cohen-Macaulay
with $K_X$ Cartier. Namely, 
one should replace $\P^1$
in Theorem \ref{F5thm} by the Harbater-Katz-Gabber curve
of Remark \ref{F3cartier}, now with $p=5$.
\end{remark}

\section{The example over the 5-adic integers}
\label{Z5}

\begin{theorem}
\label{Z5thm}
Let the group $G=\Z/5$
act on the quintic del Pezzo surface $S_5$ over $\Z_5$ by an embedding
of $G$ into the symmetric group $\Sigma_5=\Aut(S_5)$.
Let $R=\Z_5[e]/(e^5-5e^4+25e^2-25e+5)$, which is the ring of integers
in a Galois extension of $\Q_5$
with group $G=\Z/5$. Let $G$
act on the scheme $(S_5)_R$ by the diagonal action on $S_5$ and on $R$.
Then the scheme $(S_5)_R/G$
is terminal, not Cohen-Macaulay, of dimension 3,
and flat over $\Z_5$.
\end{theorem}

We define the quintic del Pezzo surface $S_5$ (over any commutative ring)
as the moduli space $\overline{M_{0,5}}$ of 5-pointed stable curves
of genus 0. That makes it clear that the symmetric group $\Sigma_5$
acts on $S_5$.

This example behaves much like the example over $\F_5$,
Theorem \ref{F5thm}. {\it In particular, the figures
in section \ref{char5}
accurately depict the blow-ups we make in mixed characteristic $(0,5)$,
just as in characteristic 5.}
We can view $R$ as the subring
of the cyclotomic ring $\Z_5[\zeta_{25}]$
fixed by the automorphism $\zeta_{25}\mapsto \zeta_{25}^7$ of order 4,
with $e=1+\zeta_{25}+\zeta_{25}^{-1}+\zeta_{25}^7+\zeta_{25}^{-7}$.
Informally, $R$ is the simplest ramified $\Z/5$-extension of $\Z_5$.
More broadly,
this action of $G$ on $(S_2)_R$ was chosen as possibly the simplest
action of $\Z/5$ on a 5-fold in mixed characteristic $(0,5)$
with an isolated fixed point.
The simplicity helps to ensure that the quotient scheme is terminal.

\begin{proof}
We work throughout over $\Z_5$. Write $G=\Z/5=\langle \sigma:
\sigma^5=1\rangle$, with $\tau:=\sigma^{-1}$.
By de Fernex \cite{dF},
the action of $G$ on $S_5$
is conjugate to the birational action of $G$ on $\P^2$ by
$$\tau([x,y,z])=[x(z-y),z(x-y),xz].$$
The fixed point over $\F_5$ is $[-2,1,-1]$. Let us change variables
over $\Z_5$ to move that point to $[0,0,1]$ (although it is only
fixed over $\F_5$). Namely, let $v_0=x+2y$, 
$v_1=z-x-y$, and $v_2=y$. In these coordinates, the action of $G$ becomes
\begin{multline*}
\tau[v_0,v_1,v_2]=[3v_0^2+3v_0v_1-12v_0v_2-8v_1v_2+10v_2^2,\\
-v_0^2-v_0v_1+5v_0v_2+3v_1v_2-5v_2^2,(v_0+v_1-v_2)(v_0-3v_2)].
\end{multline*}
Therefore, in affine coordinates $(s_0,s_1):=(v_0/v_2,v_1/v_2)$,
$G$ acts by
$$\tau(s_0,s_1)=\bigg( \frac{10-12s_0-8s_1+3s_0^2+3s_0s_1}{(3-s_0)(1-s_0-s_1)},\;
\frac{-5+5s_0+3s_1-s_0^2-s_0s_1}{(3-s_0)(1-s_0-s_1)}\bigg).$$
This reduces modulo 5 to the formula for the action of $G$
on $S_5$ over $\F_5$ in section \ref{char5}.

Let $Y_0=(S_5)_R$, with the diagonal action of $G$ on $S_5$ and on $R$.
Write $e_2$ for the generator $e$ of $R$, to fit with our numbering
of coordinates on $Y_0$; so we have
$$0=e_2^5-5e_2^4+25e_2^2-25e_2+5.$$
Then $G$ acts on an affine neighborhood $U$ of the origin by:
\begin{multline*}
\tau(s_0,s_1,e_2)=\bigg( \frac{10-12s_0-8s_1+3s_0^2+3s_0s_1}{(3-s_0)(1-s_0-s_1)},\;
\frac{-5+5s_0+3s_1-s_0^2-s_0s_1}{(3-s_0)(1-s_0-s_1)},\\
\frac{1}{7}(20-53e_2+8e_2^2+9e_2^3-2e_2^4)\bigg).
\end{multline*}
(The last expression is a generator of the Galois group of $R$ over $\Z_5$,
as one can check via Magma. The denominator 7 occurs because the
ring of integers of $\Q(\zeta_{25})^{\Z/4}$ is not monogenic, hence not
generated over $\Z$ by $e_2$.
This causes no difficulties, because we are working
over $\Z_5$.) Note that $U$ is written with
three variables over $\Z_5$, but this is a regular scheme of dimension 3
because of the equation satisfied by $e_2$:
$$0=e_2^5-5e_2^4+25e_2^2-25e_2+5.$$

We will apply Theorem \ref{mup} repeatedly to recognize
the singularities of $Y_j/G$, for various blow-ups $Y_j$
of $Y_0$. We remark now
that the assumption in Theorem \ref{mup}
that $p\in e^{p-1}\m$ will be valid in each case,
that is, that $5\in e^4\m$. Indeed, we have
$5=e_2^5(\text{unit})$ on $Y_0$, hence on each blow-up $Y_j$,
and $e_2$ is a multiple
of the function $e$ defining the Weil divisor $[Y_j^G]$
in each case, that being the function $e$ we will use for Theorem \ref{mup}.
So 5 is in the ideal $(e^5)$, hence in $e^4\m$ at each of the bad
points.

Let $X=Y_0/G$. Since $G$ acts freely on $\Spec R$ outside
its closed point, the only fixed point of $G$ on $Y_0$
is the closed point $P\cong \Spec \F_5$
given by $(s_0,s_1,e_2)=(0,0,0)$.
So $X$ is normal of dimension 3,
and $X$ is regular outside the image of $P$, which we also call $P$.
Also, $5K_X$ is Cartier.

It is not automatic from Fogarty's results \cite{Fogartydepth},
but we can use his methods to show that
$X$ is not Cohen-Macaulay at $P$. As in the proof of Theorem
\ref{Z2}, using that $G$ has an isolated fixed point on the 3-fold $Y_0$,
it suffices to show that $H^1(G,O(Y_0))$ is not zero.
This cohomology group is $\ker(\tr)/\Im(1-\sigma)$
on $O(Y_0)$, where the trace is $1+\sigma+\cdots+\sigma^4$.
The equation $0=e_2^5-5e_2^4+25e_2^2-25e_2+5$ (specifically,
the coefficient of $e_2^4$) implies that $e_2$ has trace 5.
So $\tr(1-e_2)=0$, and hence
$1-e_2$ defines an element of $H^1(G,O(Y_0))$. Note that $1-e_2$ restricts
to $1\in O(P)=\F_5$ on the fixed point $P$. Therefore, $1-e_2$ has nonzero
image under the restriction map $H^1(G,O(Y_0))\to H^1(G,O(P)\cong \F_5$.
So $H^1(G,O(Y_0))$ is not zero, and hence
$Y_0/G$ is not Cohen-Macaulay.

It remains to show that $Y_0/G$ is terminal. 
This example is complicated, and it may be impossible
to resolve the singularities
of $X$ by performing $G$-equivariant blow-ups of $Y_0$. Fortunately,
as in earlier sections, we can make $Y_7/G$ have toric singularities 
after some $G$-equivariant blow-ups $Y_7\to\cdots\to Y_0$,
exactly parallel
to those in the characteristic 5 example (section \ref{char5}).
In fact, all the formulas we write for the fixed point loci
will look {\it identical }to those in the characteristic 5 example,
because we only need to write those formulas modulo suitable
error terms.
It will then be easy
to check that $Y_0/G$ is terminal.

The blow-up $Y_1\to Y_0$ at the $G$-fixed point is,
over the open set $U\subset Y_0$:
$$\{ ((x_0,x_1,e_2),[y_0,y_1,y_2])\in U\times_{\Z_5}
\P^2_{\Z_5}: x_0y_1=x_1y_0,\; x_0y_2=e_2y_0,\; x_1y_2=e_2y_1\}.$$
We will see that the fixed point set in $Y_1$
is a curve isomorphic to $\P^1_{\F_5}$.
To check that, first work in the open subset
$\{y_0=1\}$ in $Y_1$, with coordinates $(s_0,y_1,y_2)$;
here $(s_0,s_1,e_2)=(s_0,s_0y_1,s_0y_2)$. This is an open neighborhood
of the origin in 
$$\Spec \Z_5[s_0,y_1,y_2]/((s_0y_2)^5-5(s_0y_2)^4
+25(s_0y_2)^2-25(s_0y_2)+5),$$
by the equation for $e_2$.
Since $e_2=s_0y_2$, $e_2$ is in the ideal $(s_0)$, and hence 5 is also
in $(s_0)$ (which lets us simplify formulas written
modulo $(s_0)$). It is straightforward to compute how $G$ acts
in this chart, but we do not write it out, for brevity.
The exceptional divisor $E_0$ is $\{s_0=0\}$, in this chart
(and so $E_0$ is isomorphic to $\P^2$ over $\F_5$).
The fixed point scheme $Y_1^G$ is defined by the vanishing of:
$I(s_0)=s_0(-y_1+O(s_0))$, $I(y_1)=(y_1^2+O(s_0))/(1-y_1+O(s_0))$,
and $I(y_2)=y_2(y_1+O(s_0))/(1-y_1+O(s_0))$. We know that $Y_1^G$
is contained (as a set) in $E_0$ (since $Y_0^G$ is only the origin
in characteristic $5$).
So the fixed point set is the line
$\{0=s_0=y_1\}$, in this chart.

In the chart $\{y_1=1\}$ in $Y_1$, we have $s_0=s_1y_0$
and $e_2=s_1y_2$, so we have coordinates $(y_0,s_1,y_2)$. Here
$E_0=\{s_1=0\}$. We can write the action of $G$ in these coordinates
(for example using Magma). We find that the fixed point scheme $Y_1^G$
is defined by the vanishing of:
$I(y_0)=-1+O(s_1)$, $I(s_1)=s_1^2(1+y_0-2y_0^2+O(s_1))$,
and $I(y_2)=s_1y_2(-1-y_0-y_2+2y_0^2+O(s_1))$. Since $Y_1^G$ is contained
(as a set) in $E_0$, the first equation shows that $Y_1^G$ is empty,
in this chart. In the last chart $\{y_2=1\}$ in $Y_1$,
we have coordinates $(y_0,y_1,e_2)$, and $E_0=\{ e_2=0\}$.
The fixed point scheme is defined by:
$I(y_0)=-y_1+O(e_2)$, $I(y_1)=e(y_1+y_1^2+y_0y_1-2y_0^2+O(e_2))$,
and $I(e_2)=e_2^2(-1+O(e_2))$. Since $Y_1^G$ is contained 
(as a set) in $E_0$, the fixed point set is the line
$\{0=y_1=e_2\}$, the same line seen in an earlier chart.

Thus $(Y_1^G)_{\red}$ is isomorphic to $\P^1_{\F_5}$.
Our criterion for a quotient by $G$
to have toric singularities (Theorem \ref{mup}) requires
the $G$-fixed locus to have codimension 1; so let $Y_2$ be the blow-up
of $Y_1$ along this $\P^1$. Clearly $G$ continues to act on $Y_2$.
The exceptional divisor $E_1$ in $Y_2$ is a $\P^1$-bundle over
$\P^1_{\F_5}$, and so the natural way to cover $E_1$ by affine charts
involves 4 charts, as follows. (See Figure \ref{F5figure012},
which applies to the current example as well.)

Over the open set $\{y_0=1\}$ in $Y_1$, $Y_2$ is the blow-up
along the $G$-fixed curve $\{0=s_0=y_1\}$, so $Y_2$ has coordinates
$((s_0,y_1,y_2),[w_0,w_1])$. 
First take $\{w_0=1\}$,
so $y_1=s_0w_1$, and we have coordinates $(s_0,w_1,y_2)$.
As in every other chart, there are
three variables over $\Z_5$, but this is a regular scheme of dimension 3
because of the equation satisfied by $e_2$.
In this case, we have $e_2=s_0y_2$,
and so
$$0=(s_0y_2)^5-5(s_0y_2)^4+25(s_0y_2)^2-25(s_0y_2)+5.$$
In this chart, $E_0$ does not appear, and $E_1=\{s_0=0\}$. The fixed point
scheme $Y_2^G$ is defined by: $I(s_0)=s_0^2(-1-w_1+O(s_0))$,
$I(w_1)=-2+O(s_0)$, and $I(y_2)=s_0y_2(1+w_1-y_2+O(s_0))$.
We know that the fixed point set is contained in $E_1$, and so the formula
for $I(w_1)$ implies that $Y_2^G$ is empty, in this chart.

In the other chart $\{w_1=1\}$ in $Y_2$ over the same open set in $Y_1$,
we have $s_0=y_1w_0$, and so $Y_2$ has coordinates $(w_0,y_1,y_2)$.
Here $E_0=\{w_0=0\}$, $E_1=\{y_1=0\}$. Also, $e_2=w_0y_1y_2$.
The fixed point scheme
is defined by $I(w_0)=w_0(2w_0+O(y_1))/(1-2w_0+O(y_1))$,
$I(y_1)=y_1(-2w_0+O(y_1))$, and $I(y_2)=y_1y_2(1+w_0
-w_0y_2+O(y_1))$. So $Y_2^G$ is the line $\{0=w_0=y_1\}=E_0\cap E_1$
over $\F_5$, in this chart.

To see the rest of $E_1\subset Y_2$, work over the open set
$\{y_2=1\}$ in $Y_1$. Here $Y_2$ is the blow-up along
the $G$-fixed curve $\{0=y_1=e_2\}$, so $Y_2$ has coordinates
$((y_0,y_1,e_2),[r_1,r_2])$. First take $\{r_1=1\}$ in $Y_2$,
so $e_2=y_1r_2$, and we have coordinates $(y_0,y_1,r_2)$.
Here $E_0=\{r_2=0\}$ and $E_1=\{y_1=0\}$. Here $Y_2^G$ is given by
$I(y_0)=y_1(-1+y_0r_2-y_0^2r_2+O(y_1))$, $I(y_1)
=y_1r_2(-2y_0^2+O(y_1))$,
and $I(r_2)=r_2^2(2y_0^2+O(y_1))/(1-2y_0^2r_2+O(y_1))$.
We know that the fixed point set is contained in $E_1$,
and we read off that it is the union of the two lines
$\{0=y_1=r_2\}=E_0\cap E_1$ and $\{0=y_0=y_1\}$ in $E_1$. The first curve
appeared in an earlier chart, and the second is new.
Finally, the other open set is $\{r_2=1\}$ in $Y_2$,
so $y_1=e_2r_1$, and we have coordinates $(y_0,r_1,e_2)$.
Here $E_0$ does not appear, and $E_1=\{e_2=0\}$. Here $Y_2^G$
is given by $I(y_0)=e_2(y_0-r_1-y_0^2+O(e_2))$,
$I(r_1)=-2y_0^2+O(e_2)$, and $I(e_2)=e_2^2(-1+O(e_2))$.
We read off that the fixed point set is the curve $\{0=y_0=e_2\}$,
which is the second curve in the previous chart.

Thus $(Y_2)^G$ as a set is the union of two $\P^1$'s over $\F_5$ meeting
at a point. We are trying to make the fixed locus have codimension 1,
and so our next step is to blow up one of those curves.
Namely, let $Y_3$ be the blow-up of $Y_2$ along the $G$-fixed curve
$E_0\cap E_1$. The exceptional divisor $E_2$ in $Y_3$ is a $\P^1$-bundle
over $\P^1_{\F_5}$, and so we need to look
at four affine charts to see all of it.
(See Figure \ref{F5figure34},
which applies to the current example as well.)

First, work over the open set $\{r_1=1\}$ in $Y_2$
over $\{y_2=1\}$ in $Y_1$. Then $Y_3$ is the blow-up
along the curve $\{0=y_1=r_2\}=E_0\cap E_1$, and so $Y_3$ has coordinates
$(y_0,y_1,r_2),[z_1,z_2]$. First take $\{z_1=1\}$, so $r_2=y_1z_2$,
and we have coordinates $(y_0,y_1,z_2)$. 
As in every other chart, there are
three variables over $\Z_5$, but this is a regular scheme of dimension 3
because of the equation satisfied by $e_2$.
In this case, we have $e_2=y_1^2z_2$,
and so
$$0=(y_1^2z_2)^5-5(y_1^2z_2)^4+25(y_1^2z_2)^2-25(y_1^2z_2)+5.$$

In this chart, $E_0=\{z_2=0\}$,
$E_1$ does not appear, and $E_2=\{y_1=0\}$. The fixed point scheme $Y_3^G$
is defined by: $I(y_0)=y_1(-1+O(y_1))$, $I(y_1)=y_1^2z_2(-2y_0^2+O(y_1))$,
and $I(z_2)=y_1z_2^2(-y_0^2+O(y_1))$. These equations are equivalent
to $y_1=0$, near $E_2$; so the fixed point scheme $Y_2^G$ is the Cartier
divisor $E_2$, in this chart. (Thus, by Theorem \ref{kl},
$Y_2/G$ is regular, in this open set.)

The other chart
is $\{z_2=1\}$ in $Y_3$, so $y_1=r_2z_1$, and we have coordinates
$(y_0,z_1,r_2)$. Here $E_0$ does not appear, $E_1=\{z_1=0\}$,
and $E_2=\{r_2=0\}$. Also, $e_2=z_1r_2^2$.
The fixed point scheme $Y_3^G$ is given by
$I(y_0)=z_1r_2(-1+O(r_2))$, $I(z_1)=z_1r_2(y_0^2+O(r_2))$,
and $I(r_2)=r_2^2(2y_0^2+O(r_2))$. The fixed point scheme is generically
$E_2$ with multiplicity 1, together with the other fixed curve we knew
from $Y_2$, here given by $\{0=y_0=z_1\}\subset E_1$.
In more detail, the ``bad locus'' where the scheme $Y_3^G$ is not just
$E_2$ as a Cartier divisor is given by removing a factor of $r_2$
from these equations, yielding: $0=z_1(-1+O(r_2))$, $0=z_1(y_0^2+O(r_2))$,
and $0=r_2(2y_0^2+O(r_2))$. We know the fixed locus away from $E_2$,
so assume that $r_2=0$; then these equations show that the bad locus
inside $E_2$ is the curve $\{0=z_1=r_2\}=E_1\cap E_2$.

Fortunately, Theorem \ref{mup} implies that $Y_3/G$ has toric singularities
at points of $E_1\cap E_2$ outside the origin. Namely, let $e=r_2$
and $s=z_1$; then $I(s)=es(\text{unit})$ near $E_1\cap E_2=\{0=z_1=r_2\}$
outside the origin. The theorem gives that $Y_3/G$ has singularity
$\frac{1}{5}(0,1,2)$ at points of $E_1\cap E_2$ outside the origin.

To see all of $E_2\subset Y_3$, we also have to work
over $\{w_1=1\}$ in $Y_2$, with coordinates $(w_0,y_1,y_2)$,
over $\{y_0=1\}$ in $Y_1$.
Here $Y_3$ is the blow-up along the $G$-fixed curve
$\{0=w_0=y_1\}=E_0\cap E_1$, so $Y_3$ has coordinates
$(w_0,y_1,y_2),[v_0,v_1]$. First take $\{v_0=1\}$,
so $y_1=w_0v_1$, and we have coordinates $(w_0,v_1,y_2)$ on $Y_3$.
Here $E_0$ does not appear, $E_1=\{v_1=0\}$,
and $E_2=\{w_0=0\}$. Also, $e_2=w_0^2v_1y_2$.
The fixed point scheme is defined by:
$I(w_0)=w_0^2(2-2v_1+O(w_0))$, $I(v_1)=w_0v_1(1-2v_1+O(w_0))$,
and $I(y_2)=w_0v_1y_2(1+O(w_0))$. In the chart we are working
over in $Y_2$, the fixed set
$Y_2^G$ is only the curve $E_0\cap E_1$ we are blowing up, and so
$Y_3^G$ (in this chart) is contained in $E_2$ as a set. By the equations,
$Y_3^G$ is generically the Cartier divisor $E_2$, and the bad locus
(where that fails) is given by $0=w_0$, $0=v_1(1-2v_1)$,
and $0=v_1y_2$. So the bad locus is the union of the curve
$\{0=w_0=v_1\}=E_1\cap E_2$ and the point $(w_0,v_1,y_2)=(0,
-2,0)$ in $E_2$. By Theorem \ref{mup} (using $e=s=w_0$),
$Y_3/G$ has singularity $\frac{1}{5}(2,1,0)$ everywhere
on the curve $E_1\cap E_2$ (in this chart), in agreement
with an earlier calculation.

To analyze the bad point above, change coordinates temporarily
by $t_1=v_1+2$; then the bad point becomes the origin in coordinates
$(w_0,t_1,y_2)$. In these coordinates, we have
$I(w_0)=w_0^2(1-2t_1+O(w_0))$, $I(t_1)=I(v_1)=(-t_1-2t_1^2+O(w_0))$,
and $I(y_2)=w_0y_2(-2+O(w_0))$. Theorem \ref{mup} applies,
with $s=e=w_0$, and we read off that $Y_3/G$ has singularity
$\frac{1}{5}(1,-1,-2)$ at this point. That is terminal, by the Reid-Tai
criterion (Theorem \ref{reidtai}).

The last chart we need to consider in $Y_3$ is the other open set
$\{v_1=1\}$ over the open set above in $Y_2$,
$\{w_1=1\}\subset Y_2$ over $\{y_0=1\}\subset Y_1$.
So $w_0=y_1v_0$, and we have coordinates $(v_0,y_1,y_2)$.
Here $E_0=\{v_0=0\}$, $E_1$ does not appear,
and $E_2=\{y_1=0\}$. Also, $e_2=v_0y_1^2y_2$.
Here $Y_3^G$ is defined by:
$I(v_0)=v_0y_1(2-v_0+O(y_1))$, $I(y_1)=y_1^2(1-2v_0+O(y_1))$,
and $I(y_2)=y_1y_2(1+O(y_1))$. As in the previous chart,
we know that $Y_3^G$ is contained in $E_2$ as a set.
By the equations,
$Y_3^G$ is generically the Cartier divisor $E_2$, and the bad locus
(where that fails) is given by $0=y_1$, $0=v_0(2-v_0)$,
and $0=y_2$. Thus there are two bad points in this chart,
$(v_0,y_1,y_2)$ equal to $(2,0,0)\in E_2$ or $(0,0,0)\in E_0\cap E_2$.
The first is the bad point from the previous chart, but the second one
is new. Theorem \ref{mup} works to analyze the second point (the origin),
with $e=s=y_1$. We read off that $Y_3/G$ has singularity
$\frac{1}{5}(2,1,1)$ at this point.

That finishes the analysis of $Y_3$. In particular, as a set, $Y_3^G$
is the union of the divisor $E_2$ and a curve in $E_1$. It is tempting
to blow up the $G$-fixed curve next,
but that leads to a large number of blow-ups over one point of the curve,
where the fixed point scheme is especially complicated. We therefore
define $Y_4$ as the blow-up at that point,
and only later
blow up the whole curve. This leads
more efficiently to toric singularities.

Namely, let $Y_4$ be the blow-up of $Y_3$ at the origin
in the chart $\{r_2=1\}$ in $Y_2$ (unchanged in $Y_3$),
with coordinates $(y_0,r_1,e_2)$. So $Y_4$ has coordinates
$(y_0,r_1,e_2),[q_0,q_1,q_2]$. The exceptional divisor $E_3$
is isomorphic to $\P^2_{\F_5}$, and so it is covered by 3 affine charts.
First take $\{q_0=1\}$ in $Y_4$, so $r_1=y_0q_1$
and $e_2=y_0q_2$, and we have coordinates $(y_0,q_1,q_2)$.
As in every other chart, there are
three variables over $\Z_5$, but this is a regular scheme of dimension 3
because of the equation satisfied by $e_2$.
In this case, we have $e_2=y_0q_2$,
and so
$$0=(y_0q_2)^5-5(y_0q_2)^4+25(y_0q_2)^2-25(y_0q_2)+5.$$

Here $E_1=\{q_2=0\}$ and $E_3=\{y_0=0\}$. The fixed point scheme
$Y_4^G$ is defined by: $I(y_0)=y_0^2q_2(1-q_1+O(y_0))$,
$I(q_1)=y_0(-2+q_1q_2+q_1^2q_2+O(y_0))$,
and $I(q_2)=y_0q_2^2(-2+q_1+O(y_0))$.
So $Y_4^G$ is generically
the Cartier divisor $E_3$; the $G$-fixed curve in $E_1$ does not appear
in this chart. The bad locus (where the scheme $Y_4^G$ is not
just $E_3$) is given by $0=y_0$, $0=-2+q_1q_2+q_1^2q_2$,
and $0=q_2^2(-2+q_1)$. By the second equation, $q_2\neq 0$,
and so the third equation gives that $q_1=2$. Then the second equation
gives that $0=-2+2q_2-q_2=-2+q_2$, so $q_2=2$. That is,
there is only one bad point in this chart,
$(y_0,q_1,q_2)=(0,2,2)\in E_3$. To analyze that point,
change coordinates temporarily by $s_1=q_1-2$ and $s_2=q_2-2$.
In these coordinates, $I(y_0)=y_0^2(-2-s_1-s_2-s_1s_2+O(y_0))$,
$I(s_1)=I(q_1)=y_0(s_2+2s_1^2+s_1^2s_2+O(y_0))$,
and $I(s_2)=I(q_2)=y_0(-s_1-s_1s_2+s_1s_2^2+O(y_0))$.
By Theorem \ref{mup}, with $e=s=y_0$, $Y_4/G$
has a $\mu_5$-quotient singularity. Explicitly, the linear map
$\varphi$ over $\F_5$ in the theorem is $\varphi(y_0)=-2y_0$,
$\varphi(s_1)=s_2$, and $\varphi(s_2)=-s_1$,
which has eigenvalues $-2,2,-2$. So $Y_4/G$ has singularity
$\frac{1}{5}(-2,2,-2)$ at this point. This is terminal,
by the Reid-Tai criterion.

Next, take the open set $\{q_1=1\}$ in $Y_4$, so $y_0=r_1q_0$
and $e_2=r_1q_2$, and $Y_4$ has coordinates
$(q_0,r_1,q_2)$. Here $E_1=\{q_2=0\}$ and $E_3=\{r_1=0\}$.
The fixed point scheme $Y_4^G$ is defined by:
$I(q_0)=r_1(-q_2-q_0q_2+2q_0^3+O(r_1))$,
$I(r_1)=r_1^2(2q_2-2q_0^2+O(r_1))$, and
$I(q_2)=r_1q_2(2q_2+2q_0^2+O(r_1))$. So $Y_4^G$ is generically
the Cartier divisor $E_3$, together with the $G$-fixed curve
$\{0=q_0=q_2\}$ in $E_1$. The bad locus in $E_3$
is given by $0=r_1$, $0=-q_2-q_0q_2+2q_0^3$,
and $0=q_2(2q_2+2q_0^2)$. This yields two bad points,
$(q_0,r_1,q_2)$ equal to $(-2,0,1)$ or $(0,0,0)$. The first one
is the bad point
from the previous chart, and the second is not surprising,
as it is the intersection point of $E_3$ with the $G$-fixed curve.

Finally, take the open set $\{q_2=1\}$ in $Y_4$,
so $y_0=e_2q_0$ and $r_1=e_2q_1$, and we have coordinates
$(q_0,q_1,e_2)$. Here $E_1$ does not appear, and $E_3=\{e_2=0\}$.
The fixed point scheme $Y_4^G$ is defined by:
$I(q_0)=e(2q_0-q_1+O(e_2))$, $I(q_1)=e(-2q_1-2q_0^2+O(e_2))$,
and $I(e_2)=e_2^2(-1+O(e_2))$. So $Y_4^G$ is generically $E_3$.
The bad locus in $E_3$ is given by: $0=e_2$,
$0=2q_0-q_1$, and $0=-2q_1-2q_0^2$. This yields two bad points,
$(q_0,q_1,e_2)$ equal to $(-2,1,0)$ (seen in the previous two charts)
or $(0,0,0)$, which is new. Theorem \ref{mup} applies
at this new point, with $e=s:=e_2$. The $\F_5$-linear map $\varphi$
is given by $\varphi(q_0)=2q_0-q_1$, $\varphi(q_1)=-2q_1$,
and $\varphi(e_2)=-e_2$. So $\varphi$ has eigenvalues $(2,-2,-1)$,
and hence $Y_4/G$ has singularity
$\frac{1}{5}(2,-2,-1)$ at this point. This is terminal,
by the Reid-Tai criterion.

That completes our description of $Y_4$. Next, let $Y_5$ be the blow-up
of $Y_4$ along the $G$-fixed curve in $E_1$. The exceptional
divisor $E_4$ in $Y_5$ is a $\P^1$-bundle over $\P^1_{\F_5}$,
and so it is covered by four affine charts. First work over the open set
$\{z_2=1\}$ in $Y_3$ (unchanged in $Y_4$), with coordinates
$(y_0,z_1,r_2)$; this contains the point where the $G$-fixed curve
in $E_1=\{z_1=0\}$ meets $E_2=\{r_2=0\}$.
Here $Y_5$ is the blow-up along the $G$-fixed curve
$\{0=y_0=z_1\}$, and so $Y_5$ has coordinates $(y_0,z_1,r_2),[n_0,n_1]$.
First take the open set $\{n_0=1\}$ in $Y_5$, so $z_1=y_0n_1$, and we have
coordinates $(y_0,n_1,r_2)$.
As in every other chart, there are
three variables over $\Z_5$, but this is a regular scheme of dimension 3
because of the equation satisfied by $e_2$.
In this case, we have $e_2=y_0n_1r_2^2$,
and so
$$0=(y_0n_1r_2^2)^5-5(y_0n_1r_2^2)^4+25(y_0n_1r_2^2)^2-25(y_0n_1r_2^2)+5.$$

In this chart, $E_1=\{n_1=0\}$, $E_2=\{r_2=0\}$,
and $E_4=\{y_0=0\}$. The fixed point scheme $Y_5^G$ is defined by:
$I(y_0)=y_0n_1r_2(-1+O(y_0))$, $I(n_1)=n_1r_2(n_1+O(y_0))/(1-n_1r_2+O(y_0))$,
and $I(r_2)=y_0r_2^2(-2n_1r_2+O(y_0))$.
So $Y_5^G$, as a set, is the union
of the divisor $E_2$ and the curve $\{0=y_0=n_1\}=E_1\cap E_4$. (In particular,
the fixed point set is still not all of codimension 1.) We have analyzed
the bad locus of $E_2$ in previous steps, but we have to add here
that the bad locus of $E_2$ is disjoint from $E_2\cap E_4$ except
for the point where $E_2$ meets the $G$-fixed curve, by the formula
for $I(n_1)$.
(See Figure \ref{F5figure56},
which applies to the current example as well.)

The other chart is $\{n_1=1\}$ in $Y_5$, so $y_0=z_1n_0$,
and we have coordinates $(n_0,z_1,r_2)$. Here $E_1$ does not appear,
$E_2=\{r_2=0\}$, and $E_4=\{z_1=0\}$. Also, $e_2=z_1r_2^2$.
The fixed point scheme $Y_5^G$
is defined by: $I(n_0)=r_2(-1+O(z_1))$, $I(z_1)=z_1^2r_2(-2r_2+O(z_1))$,
and $I(r_2)=z_1r_2^2(-2r_2+O(z_1))$. These equations reduce to $r_2=0$
near $E_4$, and so $Y_5^G$ is the Cartier divisor $E_2$,
in this chart.

To finish our description of $E_4$ in $Y_5$, we work
over the open set where the $G$-fixed curve in $Y_4$ meets $E_3$,
namely $\{q_1=1\}$ in $Y_4$. Here $Y_4$ has coordinates
$(q_0,r_1,q_2)$, $E_1=\{q_2=0\}$, $E_3=\{r_1=0\}$,
and the $G$-fixed curve is $\{0=q_0=q_2\}$ in $E_1$.
So the blow-up $Y_5$ along the $G$-fixed curve has coordinates
$(q_0,r_1,q_2),[u_0,u_2]$. First take $\{u_0=1\}$ in $Y_5$,
so $q_2=q_0u_2$, and we have coordinates $(q_0,r_1,u_2)$.
Here $E_1=\{u_2=0\}$, $E_3=\{r_1=0\}$, and $E_4=\{q_0=0\}$.
Also, $e_2=q_0r_1u_2$.
The fixed point scheme $Y_5^G$ is defined by:
$I(q_0)=q_0r_1(-u_2+O(q_0))$, $I(r_1)=q_0r_1^2(2u_2+O(q_0))$,
and $I(u_2)=r_1u_2^2(1+O(q_0))/(1-r_1u_2+O(q_0))$. We know the fixed
set outside $E_4$, and so we read off that the fixed set is the divisor $E_3$
together with the $G$-fixed curve $E_1\cap E_4$ found earlier.
We have analyzed the bad set of $E_3$ away from $E_4$ in earlier blow-ups,
and we see from the formula for $I(u_2)$
that the bad set of $E_3$ near $E_3\cap E_4$ is only
the point $E_1\cap E_3\cap E_4$ where the $G$-fixed curve meets $E_3$.

The other chart is $\{u_2=1\}$ in $Y_5$. Here $q_0=q_2u_0$, and so we have
coordinates $(u_0,r_1,q_2)$. Here $E_1$ does not appear,
$E_3=\{r_1=0\}$, and $E_4=\{q_2=0\}$.
Also, $e_2=r_1q_2$.
The fixed point scheme
$Y_5^G$ is defined by: $I(u_0)=r_1(-1+O(q_2))$,
$I(r_1)=r_1^2q_2(2+O(q_2))$, and $I(q_2)=r_1q_2^2(2+O(q_2))$. These
equations reduce to $r_1=0$ near $E_4$, and so the fixed point scheme
$Y_5^G$ is the Cartier divisor $E_3$, in this chart.

That completes our description of $Y_5$. Let $Y_6$ be the blow-up
of $Y_5$ along the $G$-fixed curve $E_1\cap E_4$. The exceptional
divisor $E_5$ in $Y_6$ is a $\P^1$-bundle over $\P^1_{\F_5}$, covered by
four affine charts. First take the open set $\{n_0=1\}$ in $Y_5$, 
which contains the point where the $G$-fixed curve meets $E_2$.
Here $Y_5$ has coordinates $(y_0,n_1,r_2)$,
with $E_1=\{n_1=0\}$, $E_2=\{r_2=0\}$,
and $E_4=\{y_0=0\}$. Since $Y_6$ is the blow-up along
the $G$-fixed curve $\{0=y_0=n_1\} =E_1\cap E_4$, $Y_6$
has coordinates $(y_0,n_1,r_2),[m_0,m_1]$. First take $\{m_0=1\}$
in $Y_6$, so $n_1=y_0m_1$, and we have coordinates
$(y_0,m_1,r_2)$. As in every other chart, there are
three variables over $\Z_5$, but this is a regular scheme of dimension 3
because of the equation satisfied by $e_2$.
In this case, we have $e_2=y_0^2m_1r_2^2$,
and so
$$0=(y_0^2m_1r_2^2)^5-5(y_0^2m_1r_2^2)^4+25(y_0^2m_1r_2^2)^2
-25(y_0^2m_1r_2^2)+5.$$

In this chart, $E_1=\{m_1=0\}$, $E_2=\{r_2=0\}$,
$E_4$ does not appear, and $E_5=\{y_0=0\}$. The fixed point scheme
$Y_6^G$ is defined by: $I(y_0)=y_0^2m_1r_2(-1+O(y_0))$,
$I(m_1)=y_0m_1r_2(2m_1+O(y_0))$,
and $I(r_2)=y_0^2r_2^2(2-2m_1r_2+O(y_0))$.
We know the fixed point
set away from $E_5$, and so we read off that the fixed point scheme
is generically the Cartier divisor $E_2+E_5$. (Since $E_5$ is fixed
by $G$, we have finally made the fixed point set of codimension 1.)
Let $e=y_0r_2$. The bad locus
(where the scheme $Y_6^G$ is more than the Cartier divisor $E_2+E_5$),
on $E_5$, is given by factoring out $e$ from the equations
and setting $y_0=0$, so we get: $0=y_0$ and $0=2m_1^2$. So,
as a set, the bad locus is the curve $\{0=y_0=m_1\}=E_1\cap E_5$.
Theorem \ref{mup} does not seem to apply to this curve,
and so $Y_6/G$ might not
have toric singularities there; we will have to blow up one more time.

For now, look at the other open set, $\{m_1=1\}$ in $Y_6$.
So $s_0=n_1m_0$, and we have coordinates $(m_0,n_1,r_2)$. Here
$E_1$ does not appear, $E_2=\{r_2=0\}$, $E_4=\{m_0=0\}$,
and $E_5=\{n_1=0\}$. 
Also, $e_2=y_0m_1^2r_2^2$.
The fixed point scheme $Y_6^G$ is defined by:
$I(m_0)=m_0n_1r_2(-2+O(n_1))$, $I(n_1)=n_1^2r_2(1+O(n_1))$,
and $I(r_2)=m_0n_1^2r_2^2(2m_0-2r_2+O(n_1))$. Since we know the fixed point
set outside $E_5$, we read off that the fixed point scheme
is generically the Cartier divisor $E_2+E_5$. Let $e=n_1r_2$.
The bad locus (where the scheme $Y_6^G$ is more than the Cartier
divisor $E_2+E_5$) is the curve $\{0=m_0=n_1\}=E_4\cap E_5$.
Fortunately, Theorem \ref{mup} applies, with $s=m_0$. We read
off that $Y_6/G$ has singularity $\frac{1}{5}(-2,1,0)$ along the whole
curve $E_4\cap E_5$, in this chart.

To finish describing $E_5\subset Y_6$, we have to work over the open
set $\{u_0=1\}$ in $Y_5$, where the $G$-fixed curve $E_1\cap E_4$
in $Y_5$ meets $E_3$. Here $Y_5$ has coordinates
$(q_0,r_1,u_2)$, with $E_1=\{u_2=0\}$, $E_3=\{r_1=0\}$, and $E_4=\{q_0=0\}$.
Then $Y_6$ is the blow-up along the $G$-fixed curve
$\{0=q_0=u_2\}=E_1\cap E_4$, so $Y_6$ has coordinates
$(q_0,r_1,u_2),[t_0,t_2]$. First take $\{t_0=1\}$ in $Y_6$,
so $u_2=q_0u_2$ and we have coordinates $(q_0,r_1,t_2)$.
Here $E_1=\{t_2=0\}$, $E_3=\{r_1=0\}$,
$E_4$ does not appear, and $E_5=\{q_0=0\}$.
Also, $e_2=q_0^2r_1t_2$.
The fixed point scheme
is defined by: $I(q_0)=q_0^2r_1(-t_2+O(q_0))$,
$I(r_1)=q_0^2r_1^2(-2+2t_2+O(q_0))$, and $I(t_2)=q_0r_1t_2(2t_2+O(q_0))$.
So the fixed point scheme is generically $E_3+E_5$. Let $e=q_0r_1$.
The bad locus (where the scheme $Y_6^G$ is more than the Cartier
divisor $E_3+E_5$), in $E_5$, is given by $0=q_0$
and $0=2t_2^2$, so (as a set) it is the curve $\{0=q_0=t_2\}
=E_1\cap E_5$, which we met in an earlier chart.

The other chart is $\{t_2=1\}$ in $Y_6$, so $q_0=u_2t_0$,
and we have coordinates $(t_0,r_1,u_2)$. Here $E_1$ does not appear,
$E_3=\{r_1=0\}$, $E_4=\{t_0=0\}$, and $E_5=\{u_2=0\}$.
Also, $e_2=q_0^2r_1t_2$.
The fixed point scheme $Y_6^G$ is defined by: $I(t_0)=t_0r_1u_2(-2+O(u_2))$,
$I(r_1)=t_0r_1^2u_2^2(2-2t_0+O(u_2))$, and
$I(u_2)=r_1u_2^2(1+O(u_2))$. So $Y_6^G$ is generically $E_3+E_5$.
Let $e=r_1u_2$.
The bad locus (where the scheme $Y_6^G$ is more than the Cartier
divisor $E_3+E_5$), in $E_5$, is the curve $\{0=t_0=u_2\}
=E_4\cap E_5$, which we met in an earlier chart. Theorem \ref{mup}
applies, with $s=t_0$. Namely, $Y_6/G$ has singularity
$\frac{1}{5}(-2,0,1)$ everywhere on the curve $E_4\cap E_5$ in this chart
(including the origin, which did not appear in the earlier chart).

That completes our description of $Y_6$. In particular, the $G$-fixed locus
has codimension 1 in $Y_6$, and $Y_6/G$ has toric singularities
outside the image of the curve $E_1\cap E_5$. Let $Y_7$ be the blow-up
of $Y_6$ along that curve. The exceptional divisor $E_6$ in $Y_7$
is a $\P^1$-bundle over $\P^1_{\F_5}$, and so we will cover $E_6$
with four affine charts. First, work over the open set $\{m_0=1\}$
in $Y_6$, where the bad curve $E_1\cap E_5$ meets $E_2$.
Here $Y_6$ has coordinates
$(y_0,m_1,r_2)$, with $E_1=\{m_1=0\}$, $E_2=\{r_2=0\}$,
and $E_5=\{y_0=0\}$. Since $Y_7$ is the blow-up along the curve
$\{0=y_0=m_1\}=E_1\cap E_5$, $Y_7$ has coordinates
$(y_0,m_1,r_2),[j_0,j_1]$. (See Figure \ref{F5figure7},
which applies to the current example as well.)

First take $\{j_0=1\}$ in $Y_7$,
so $m_1=y_0j_1$, and we have coordinates $(y_0,j_1,r_2)$.
As in every other chart, there are
three variables over $\Z_5$, but this is a regular scheme of dimension 3
because of the equation satisfied by $e_2$.
In this case, we have $e_2=y_0^3j_1r_2^2$,
and so
$$0=(y_0^3j_1r_2^2)^5-5(y_0^3j_1r_2^2)^4+25(y_0^3j_1r_2^2)^2
-25(y_0^3j_1r_2^2)+5.$$

In this chart,
$E_1=\{j_1=0\}$, $E_2=\{r_2=0\}$, $E_5$ does not appear,
and $E_6=\{y_0=0\}$. The fixed point scheme $Y_7^G$
is defined by: $I(y_0)=y_0^3j_1r_2(-1+O(y_0))$,
$I(j_1)=y_0^2j_1r_2(1-2j_1+O(y_0))$,
and $I(r_2)=y_0^2r_2^2(2+O(y_0))$. So $Y_7^G$ is generically
the Cartier divisor $E_2+2E_6$. 
Let $e=y_0^2r_2$. The bad locus
(where the scheme $Y_7^G$ is more than the Cartier divisor $E_2+2E_6$),
in $E_6$, is given by $0=y_0$, $0=j_1(1-2j_1)$, and $0=r_2$,
so it consists of the two points $(y_0,j_1,r_2)$ equal to $(0,0,0)
=E_1\cap E_2\cap E_6$
or $(0,-2,0)\in E_2\cap E_6$. At the first point, Theorem \ref{mup} applies,
with $s=r_2$. We read off that $Y_7/G$ has singularity
$(0,1,2)$ everywhere on the curve $E_1\cap E_2$ (including the origin,
which did not appear when we saw $E_1\cap E_2$ in an earlier chart).
To analyze the second point, change coordinates temporarily
by $s_1=j_1+2$; then that point becomes the origin
in coordinates $(y_0,s_1,r_2)$. We have $I(y_0)=y_0^3r_2(2-s_1+O(y_0))$,
$I(s_1)=I(j_1)=y_0^2r_2(-s_1-2s_1^2+O(y_0))$,
and $I(r_2)=y_0^2r_2^2(2+O(y_0))$. Theorem \ref{mup} applies,
with $s=r_2$. We read off that $Y_7/G$ has singularity
$\frac{1}{5}(2,-1,2)$ at this point.

The other open set is $\{j_1=1\}$ in $Y_7$, so $y_0=m_1j_0$,
and we have coordinates $(j_0,m_1,r_2)$. Here $E_1$ does not appear,
$E_2=\{r_2=0\}$, $E_5=\{j_0=0\}$, and $E_6=\{m_1=0\}$.
Also, $e_2=j_0^2m_1^3r_2^2$.
The fixed point scheme $Y_7^G$ is defined by:
$I(j_0)=j_0^2m_1^2r_2(2-j_0+O(m_1))$,
$I(m_1)=j_0m_1^3r_2(2+j_0+O(m_1))$, and
$I(r_2)=j_0^2m_1^2r_2^2(2+O(m_1))$.
So $Y_7^G$
is generically $E_2+E_5+2E_6$. Let $e=j_0m_1^2r_2$.
The bad locus (where the scheme $Y_7^G$ is more than the Cartier
divisor $E_2+E_5+2E_6$), on $E_6$, is given by:
$0=m_1$, $0=j_0(2-j_0)$, and $0=j_0r_2$. So the bad locus
is the union of the curve $\{0=j_0=m_1\}=E_5\cap E_6$
and the point $(j_0,m_1,r_2)=(2,0,0)$ in $E_2\cap E_6$.
That point is the one we analyzed in the previous chart.
For the curve, Theorem \ref{mup} applies, using
$s=j_0$. We read off that $Y_7/G$ has singularity
$\frac{1}{5}(2,2,0)$ everywhere on the curve $E_5\cap E_6$,
in this chart.

Last, work over the open set
$\{t_0=1\}$ in $Y_6$,
where the bad curve $E_1\cap E_5$ meets $E_3$. 
Here $Y_6$ has coordinates $(q_0,r_1,t_2)$,
with $E_1=\{t_2=0\}$, $E_3=\{r_1=0\}$,
and $E_5=\{q_0=0\}$. We obtain $Y_7$ by blowing up along the curve
$\{0=q_0=t_2\}=E_1\cap E_5$, so $Y_7$ has coordinates
$(q_0,r_1,t_2),[x_0,x_2]$. First take $\{x_0=1\}$,
so $t_2=q_0x_2$, and we have coordinates $(q_0,r_1,x_2)$.
Here $E_1 = \{x_2=0\}$, $E_3 = \{r_1=0\}$,
$E_5$ does not appear, and $E_6 = \{q_0=0\}$.
Also, $e_2=q_0^3r_1x_2$.
The fixed point
scheme $Y_7^G$ is defined by:
$I(q_0)=q_0^3r_1(2-x_2+O(q_0))$, $I(r_1)=q_0^2r_1^2(-2+O(q_0))$,
and $I(x_2)=q_0^2r_1x_2(1-2x_2+O(q_0))$. So $Y_7^G$ is generically
$E_3+2E_6$. 
Let $e=q_0^2r_1$.
The bad locus (where the scheme $Y_7^G$ is more than
the Cartier divisor $E_3+2E_6$), in $E_6$, is given by:
$0=q_0$, $0=r_1$, and $0=x_2(1-2x_2)$, so it consists
of the two points $(q_0,r_1,x_2)$ equal to $(0,0,0)=E_1\cap E_3\cap E_6$
or $(0,0,-2)$ in $E_3\cap E_6$. Since $I(r_1)=er_1(\text{unit})$,
Theorem \ref{mup} applies at both points. At the origin,
the theorem gives that $Y_7/G$ has singularity
$\frac{1}{5}(2,-2,1)$, which is terminal.
For the other point, change coordinates
temporarily by $y_2=x_2+2$, so that the point becomes the origin
in coordinates $(q_0,r_1,y_2)$. We have $I(q_0)
=q_0^3r_1(-1-y_2+O(q_0))$, $I(r_1)=q_0^2r_1^2(-2+O(q_0))$,
and $I(y_2)=I(x_2)=q_0^2r_1(-y_2-2y_2^2+O(q_0))$.
So Theorem \ref{mup} gives that $Y_7/G$ has singularity
$\frac{1}{5}(-1,-2,-1)$ at this point.

The other chart is $\{x_2=1\}$ in $Y_7$, so $q_0=t_2x_0$,
and we have coordinates $(x_0,r_1,t_2)$. Here $E_1$ does not appear,
$E_3=\{r_1=0\}$, $E_5=\{x_0=0\}$, and $E_6=\{t_2=0\}$.
Also, $e_2=x_0^2r_1t_2^3$.
The fixed
point scheme $Y_7^G$ is defined by:
$I(x_0)=x_0^2r_1t_2^2(2-x_0+O(t_2))$,
$I(r_1)=x_0^2r_1^2t_2^2(-2+O(t_2))$,
and $I(t_2)=x_0r_1t_2^3(2-2x_0+O(t_2))$.
So $Y_7^G$
is generically $E_3+E_5+2E_6$. Let $e=x_0r_1t_2^2$.
The bad locus (where $Y_7^G$
is more than the Cartier divisor $E_3+E_5+2E_6$),
in $E_6$, is given by: $0=t_2$, $0=x_0(2-x_0)$, and $0=x_0r_1$,
which is the union of the curve $\{0=x_0=t_2\}=E_5\cap E_6$
and the point $(x_0,r_1,t_2)=(2,0,0)$ in $E_3\cap E_6$.
We analyzed that point in the previous chart. Theorem \ref{mup}
applies to the curve, using $s=x_0$. We read off that $Y_7/G$
has singularity $\frac{1}{5}(2,0,2)$ everywhere on the curve
$E_5\cap E_6$ (including the origin, which did not appear
in the earlier chart where we met this curve).

That completes our analysis of $Y_7$; we have shown
that $Y_7/G$ has toric singularities. The sequence
of blow-ups and the descriptions of the singularities
of $Y_7/G$ are identical to those in the characteristic 5 example,
Theorem \ref{F5thm}. Given that, the proof 
that $Y_0/G$ is terminal is unchanged from that of
the characteristic 5 example.
Theorem \ref{Z5thm} is proved.
\end{proof}

\begin{remark}
\label{Z5cartier}
As in Remark \ref{Z3cartier},
I expect that there is also a 3-dimensional scheme $X$,
flat over $\Z_5$, that is terminal and non-Cohen-Macaulay
with $K_X$ Cartier. Namely, 
one should replace the $p$-adic integer ring $R=\Z_5[\zeta_{25}]^{\Z/4}$
in Theorem \ref{Z5thm} by $S=\Z_5[\zeta_{25}]$, with the action
of $\Z/5\subset (\Z/25)^*$.
\end{remark}


\small \sc UCLA Mathematics Department, Box 951555,
Los Angeles, CA 90095-1555

totaro@math.ucla.edu
\end{document}